\documentclass[10pt,a4paper,reqno]{amsart} 

\usepackage{amsmath}

\usepackage{amssymb}
\usepackage{graphicx}
\usepackage{color}
\usepackage{latexsym}
\usepackage[utf8]{inputenc}
\usepackage[T1]{fontenc}
\usepackage{comment}
\usepackage{stmaryrd}
\usepackage{placeins}

\newcommand{\ns}{{\mathbb N}} 
\newcommand{\zs}{{\mathbb Z}} 
\newcommand{\cs}{{\mathbb C}} 

\renewcommand{\chi}{{\bf 1}}

\newcommand{\spacebreak}
{\begin{displaymath} \triangleleft \; \lhd \;
\diamond \; \rhd \; \triangleright
  \end{displaymath}}

\newcommand{\cA}{\mathcal A}

\newcommand{\cS}{\mathcal S}
\newcommand{\cT}{\mathcal T}

\newcommand{\bmc}{\mbox{\boldmath$c$}}
\newcommand{\bmsc}{\mbox{\scriptsize\boldmath$c$}}

\newcommand{\llb}{\llbracket }
\newcommand{\rrb}{\rrbracket }
\newcommand{\tn}{\tilde n}

\newtheorem{Theorem}{Theorem}
\newtheorem{Proposition}[Theorem]{Proposition}

\newtheorem{Observation}[Theorem]{Observation}

\newtheorem{Lemma}[Theorem]{Lemma}

\newtheorem{Definition}[Theorem]{Definition}

\newcommand{\beq}{\begin{equation}}
\newcommand{\eeq}{\end{equation}}

\newcommand{\gf}{generating function}

\def\emm#1,{{\em #1}}

\graphicspath{{Figures/}}

\catcode`\@=11
\def\section{\@startsection{section}{1}%
 \z@{.7\linespacing\@plus\linespacing}{.5\linespacing}%
 {\normalfont\bfseries\scshape\centering}}

\def\subsection{\@startsection{subsection}{2}%
  \z@{.5\linespacing\@plus\linespacing}{.5\linespacing}%
  {\normalfont\bfseries\scshape}}

\def\subsubsection{\@startsection{subsubsection}{3}%
 \z@{.5\linespacing\@plus\linespacing}{-.5em}
  {\normalfont\bfseries\itshape}}
\catcode`\@=12

\def\qed{$\hfill{\vrule height 3pt width 5pt depth 2pt}$}
\def\boite{${\vrule height 5pt width 5pt depth 0pt}$\ }

\begin{document}
\title
[The vertical profile of embedded trees]
{The vertical profile of embedded trees}

\author[M. Bousquet-M\'elou]{Mireille Bousquet-M\'elou}
\author[G. Chapuy]{Guillaume Chapuy}
\thanks{GC was partially supported by the ERC grant StG 208471
  - ExploreMaps, and thanks the LaBRI for a ``Junior Invitation'' in September 2010.}
\address{MBM: CNRS, LaBRI, Universit\'e Bordeaux 1, 
351 cours de la Lib\'eration, 33405 Talence, France}
\email{mireille.bousquet@labri.fr}
\address{GC: CNRS, LIAFA, Universit\'e Paris 7, 
175 rue du Chevaleret, 75203 Paris, France}
\email{guillaume.chapuy@liafa.jussieu.fr}

\thanks{MBM was partially supported by  the French ``Agence Nationale
de la Recherche'', project A3 ANR-08-BLAN-0190.}



\begin{abstract}
Consider a rooted binary tree with $n$ nodes. Assign with the root the
abscissa $0$, and with the left (resp.~right) child of a node of
abscissa $i$  the abscissa $i-1$ (resp.~$i+1$). We prove that the number of binary
trees of size $n$ having exactly $n_i$ nodes at abscissa $i$, for
$\ell \le i \le r$ (with $n = \sum _i n_i$), is 
$$
\frac{n_0}{n_\ell n_r} {{n_{-1}+n_1} \choose {n_0-1}}
\prod_{\ell\le i\le r \atop i\not = 0}{{n_{i-1}+n_{i+1}-1} \choose {n_i-1}},
$$
with $n_{\ell-1}=n_{r+1}=0$. 
The sequence $(n_\ell, \ldots, n_{-1};n_0, \ldots, n_r)$ is called the
\emm vertical profile, of the tree. The vertical profile of a uniform
random tree of size $n$ is known to converge, in a certain sense and after
normalization, to a random mesure called the \emm integrated
superbrownian excursion,,  which  motivates our interest
in the profile.

We prove similar looking formulas for other families of trees whose
nodes are embedded in $\zs$. We also refine these formulas  by taking into account the number of nodes at abscissa
$j$ whose parent lies at abscissa $i$, and/or the number of vertices at
abscissa $i$ having a prescribed number of children at abscissa $j$,
for all $i$ and $j$.

Our proofs are bijective. 
\end{abstract}

\date{\today}
\maketitle

%

\section{Introduction}
Consider a rooted binary tree: each node has a left child and/or a
right child. The \emm height, of a node is its distance
to the root. The \emm horizontal profile, of the tree is $(h_0, h_1,
\ldots, h_k)$, where 
$h_i$ is the number of nodes at height $i$ and $k$ is the maximal height of a node
(Figure~\ref{fig:binary-profiles}, left). It is
easy to see that the number of trees with horizontal profile $(1, h_1,
\ldots, h_k)$ is
\beq
\label{profil-hor}
\prod_{i=0}^{k-1} {{2h_i}\choose h_{i+1}},
\eeq
with $h_0=1$. Indeed, the binomial coefficient ${{2h_i}\choose h_{i+1}}$ describes
how to spread $h_{i+1}$ nodes of height $i+1$ in the $2h_i$ slots
created by the $h_i$ nodes lying at height $i$.
The horizontal profile of trees has been much studied in the
literature
and is very well
understood~\cite{aldous-crtII,drmota-profile,flajolet-prodinger-level-sequences,meir-moon-altitude,takacs}. Expression~\eqref{profil-hor}
appears for instance in~\cite{brown-schubert}. 

\begin{figure}[htb]
\scalebox{0.7}{\input{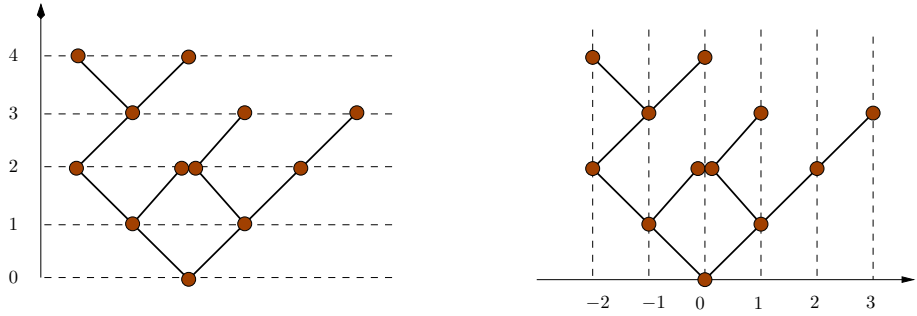}}
\caption{A rooted binary tree having horizontal profile $(1,2,4,3,2)$ and vertical profile $(2,2;4,2,1,1)$.}
\label{fig:binary-profiles}
\end{figure}

Now, assign to each node, instead of an \emm ordinate, (its height), an \emm
abscissa,: the root lies at abscissa 0, and the abscissa of the right
(resp. left) child of a node of abscissa $i$ is $i+1$
(resp.~$i-1$). We say that the tree is (canonically) \emm embedded, in
$\zs$.
The \emm vertical profile, of the tree is $(n_\ell, \ldots, n_{-1};
n_0, n_1, \ldots, n_r)$, where $n_i$ is the number of nodes at
abscissa $i$, and $\ell$ (resp.~$r$) is the smallest (resp.~largest)
abscissa occurring in the tree (Figure~\ref{fig:binary-profiles},
right). We prove in this paper that the number of trees with a 
prescribed vertical profile is given by a formula that is as
compelling as~\eqref{profil-hor}, but, we believe, far less obvious.

\begin{Theorem}\label{thm:binary-profile-pm}
  Let $\ell\le 0\le r$, and let $(n_i)_{\ell\le i\le r}$ be a
  sequence of positive integers. The number of binary trees having
  vertical profile $(n_\ell, \ldots, n_{-1};
n_0, n_1, \ldots, n_r)$ is
$$
\frac{n_0}{n_\ell n_r} {{n_{-1}+n_1} \choose {n_0-1}}
\prod_{\ell\le i\le r \atop i\not = 0}{{n_{i-1}+n_{i+1}-1} \choose {n_i-1}},
$$
with $n_{\ell-1}=n_{r+1}=0$. 
\end{Theorem}
For instance, the number of binary trees having vertical profile $(2;2,1)$ is
$$
\frac 2 {2  \times 1} {3\choose 1}{1 \choose 1}{1\choose 0}=3,
$$
and these trees are shown in Figure~\ref{fig:binary-small}.

\begin{figure}[htb]
\includegraphics[scale=0.7]{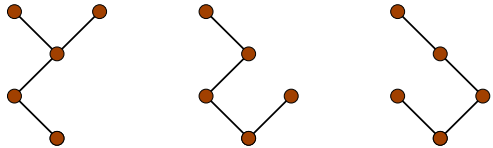}
\caption{The three  rooted binary trees having vertical profile $(2;2,1)$.}
\label{fig:binary-small}
\end{figure}

\medskip
 This unexpected formula has first  an obvious combinatorial
 interest: its proof --  
especially a bijective proof -- has to shed  a new light on the 
combinatorics of binary trees, which are of course  eminently classical objects.
But our original motivation lies in the link between the vertical profile of
binary trees and a certain random probability measure, called the
\emm integrated superbrownian excursion,, or ISE. The ISE is  the limit, as 
$n$ increases, of the 
(normalized) 
\emm occupation measure, of a
uniform random tree $T$ having  $n$ vertices~\cite{jf-homothetie}. The
normalized occupation measure  of $T$ is defined to be
$$
\mu_n = \frac 1 n \sum_{v\in T} \delta_{a(v)n^{-1/4}},
$$
where $\delta_x$ is the Dirac measure at $x$ and $a(v)$ denotes the abscissa of the vertex $v$. Note the double
normalization, first by $1/n$ (to obtain a probability 
distribution) and then  by $n^{-1/4}$ (which is known to be the
correct scaling to obtain a non-trivial limit).
Theorem~\ref{thm:binary-profile-pm} thus describes explicitly the law of
$\mu_n$: 
indeed, the probability that 
$$
\mu_n(i\,n^{-1/4})=\frac{n_i}n \quad \hbox{ for all } i \in \llbracket \ell,
r\rrbracket
$$  (and $\mu_n(in^{-1/4})=0$ for other
values of $i$), with $n=\sum_i n_i$, is the number of Theorem~\ref{thm:binary-profile-pm},
divided by $C_n={2n\choose n}/(n+1)$, the
number of binary trees of size $n$.  

 The ISE is not only related to binary trees. In fact, it appears to be a
``universal'' measure associated with numerous embedded branching
structures~\cite{aldous,derbez1,janson-jfm,legall-livre,marckert-mokka-snake,
miermont-multitypeGW}. Due
to the existence of bijections between certain 
families of \emm rooted planar maps, and  
 embedded trees,
it also describes (up to a translation) the limiting distribution of distances to the root vertex in
 planar maps of large
 size~\cite{bdg-geodesic,bdg-statistics,chassaing-schaeffer,cori-vauquelin,
 miermont-weill, miermont-marckert-bipartite}.
 Similar connections actually exist for maps on any orientable
 surface, for which the limiting distribution of distances
is explicitly related to the  ISE~\cite{chapuy-ISEgenus}.
 The law of the ISE
is  the subject of a very active research~\cite{mbm-ise,mbm-janson,delmas,devroye-janson,chassaing-janson,janson,legall-livre,legall-weill},
and we hope that knowing explicitly the law of $\mu_n$  will eventually
yield a better understanding of the law of the ISE.  
For instance, the law of the support of the ISE, and the law of its
density at one point, have already been determined though the
study of embedded binary trees~\cite{mbm-ise,mbm-janson}.

\medskip
 Let us now return to Theorem~\ref{thm:binary-profile-pm}. This theorem is
 not isolated: for instance, we  prove a  similar
formula for embedded ternary trees. 
But our results also deal with  embedded
\emm Cayley, trees. Recall that a (rooted) Cayley tree of size $n$ is a tree (in the graph-theoretic sense) on the
vertex set $V=\{1, 2, \ldots, n\}$, with a distinguished vertex $\rho$
called the \emm root,.  An \emm embedding, of such a tree in $\zs$   is a
map $a: V\rightarrow \zs$ such that $a(\rho)=0$ and $|a(v)-a(v')|=1$ if
$v$ and $v'$ are neighbours. We call $a(v)$ the \emm abscissa, of
$v$. The \emm vertical profile, of this
embedded tree is  $(n_\ell, \ldots, n_{-1}; n_0, n_1, \ldots, n_r)$,
where $n_i$ is the number of vertices
at abscissa $i$, and
$\ell$ (resp.~$r$) is the smallest (resp.~largest) 
abscissa occurring in the tree. 
The counterpart of Theorem~\ref{thm:binary-profile-pm} for Cayley trees reads as follows.

\begin{Theorem}\label{thm:cayley-profile-pm}
  Let $\ell\le 0\le r$, and let $(n_i)_{\ell\le i\le r}$ be a
  sequence of positive integers. The number of embedded rooted Cayley trees having
  vertical profile $(n_\ell, \ldots, n_{-1};n_0, n_1, \ldots, n_r)$ is
$$
\frac{n_0}{n_\ell n_r} \frac{n!}{\prod \limits_{i=\ell}^r (n_i-1)!}
\prod_{i=\ell}^{ r}(n_{i-1}+n_{i+1})^ {n_i-1},
$$
where $n=\sum_i n_i$ is the number of vertices and $n_{\ell-1}=n_{r+1}=0$.
\end{Theorem}
For example, the number of Cayley trees having vertical profile $(2;2,1)$ is
$$
\frac 2{2\times 1}\frac{5!}{1!\times  1!\times  0!}\ 2^1\times 3^1\times 2^0= 5! \times 6.
$$
The shapes of these trees are shown in
Figure~\ref{fig:cayley-small}. The positions of the vertices indicate
their abscissas, but the labels of the nodes (in the interval
$\llbracket 1,5\rrbracket$)
are not indicated. To each of the first 5 shapes there corresponds $5!$
Cayley trees. To each of the last 2 shapes there corresponds $5!/2$
Cayley trees.

\begin{figure}[htb]
\includegraphics[scale=0.7]{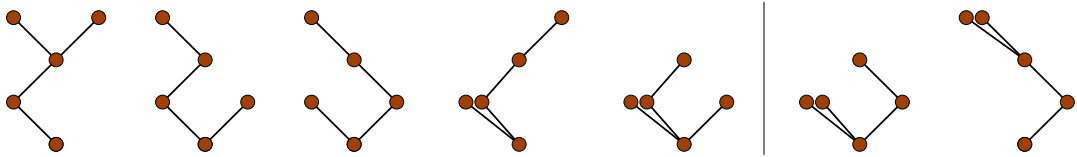}
\caption{The shapes of the rooted  Cayley trees having vertical profile $(2;2,1)$.}
\label{fig:cayley-small}
\end{figure}

\medskip
Theorems~\ref{thm:binary-profile-pm} and~\ref{thm:cayley-profile-pm} can be
proved using the multivariate Lagrange inversion formula, as will be
shown in Section~\ref{sec:other}. 
Theorem~\ref{thm:cayley-profile-pm}
can also be proved using the matrix-tree theorem. However, non-trivial
cancellations occur during the calculation, and the simple product
forms remain mysterious. This is why we focus in the paper on
\emm bijective, proofs,  which explain directly the product
forms. Moreover, these proofs allow us to  consider other abscissa increments
than $\pm 1$ (with the condition that the largest increment is
1). They also allow us to refine the enumeration,
by taking into account the number of vertices at abscissa
$j$ whose parent 
lies at abscissa $i$, and/or the number of vertices at
abscissa $i$ having a prescribed number of children at abscissa $j$,
for all $i$ and $j$. 
In particular, we can impose that each vertex has at most one child of
each abscissa, and  the binary trees
of Theorem~\ref{thm:binary-profile-pm} will in fact be seen
(up to a symmetry factor of $n!$, where $n$ is the 
number of vertices) as
rooted Cayley trees such that each vertex of abscissa
$i$ has at most one child at abscissa $i+1$ and at most one child at
abscissa $i-1$.

Our enumerative results are presented in the next section. We describe in
Section~\ref{sec:basic} a first, basic bijection. It transforms certain functions into
embedded trees, and is  close to a bijection constructed by
Joyal to count Cayley trees~\cite{joyal}. Section~\ref{sec:enum-nonneg} collects
simple enumerative results on functions, and converts them, via our
bijection, into results on trees.
Unfortunately, this basic bijection only proves the results of
Section~\ref{sec:main} for trees with non-negative labels. In
Section~\ref{sec:general}, we design a much more involved variant of the basic bijection,
which proves the remaining results (in Section~\ref{sec:enum-neg}). Finally, we discuss in
Section~\ref{sec:other} two other approaches to count embedded trees,
namely functional equations coupled with the multivariate Lagrange inversion formula, and
the matrix-tree theorem. These approaches require less invention, but
they do not explain the product forms, and they prove only part of our results.

\section{Main results}
\label{sec:main}
\subsection{Embedded trees: definitions}
\label{sec:def}
 A \emm rooted Cayley tree of size $n$, is a tree (that is, an acyclic connected graph)
on the vertex set $V=\{1, \ldots, n\}$, with a distinguished vertex
called the root. 

Let $\cS \subset \zs$ be a set of integers. An
\emm $\cS$-embedded, Cayley tree is a rooted Cayley tree in which every
vertex $v$ is assigned an abscissa $a(v)\in \zs$ in such a way:
\begin{itemize}
\item the abscissa of the root vertex is $0$,
\item if $v'$ is a child of $v$, then $a(v')-a(v) \in \cS$.
\end{itemize}
The \emm vertical profile, of the tree is $(n_\ell, \ldots, n_{-1};
n_0, n_1, \ldots, n_r)$, where $n_i$ is the number of nodes at
abscissa $i$, and $\ell$ (resp.~$r$) is the smallest (resp.~largest)
abscissa found in the tree. The tree is \emm non-negative, if all
vertices lie at a non-negative abscissa.
Equivalently, $\ell=0$.

Let $m =\min \cS$ and $M=\max \cS$. The \emm type, of a vertex $v$ is
$(i; s; c^m,  \ldots, c^M)$, where $i=a(v)$ is the abscissa of $v$,
$i-s$ is the abscissa of its parent (if $v$ is the root, we take
$s=\varepsilon$), and for $m\le k\le M$, $c^k$ is the
number of children of $v$ at abscissa $i+k$. 
Note that $s\in \cS \cup \{\varepsilon\}$. 
 We often denote $\bmc:=(c^m,  \ldots, c^M)\in
\ns^{M+1-m}$. The \emm out-type, of $v$ is
simply $(i;s)$. Its \emm in-type, is $(i; \bmc)$. The
reason for this terminology is that the edges are considered to be
oriented towards the root. We sometimes call $(i;s;\bmc)$ the \emm complete, type of $v$.

An embedded Cayley tree is \emm injective, if two distinct vertices
lying at the same abscissa have different parents. Equivalently,
every vertex $v$ has at most one child at abscissa $a(v)+s$, for all
$s\in \cS$. Two embedded Cayley $T$ and $T'$ trees are \emm equivalent, if they
only differ by a renaming of the vertices.
More precisely, if $T$ and $T'$ have size $n$, they are equivalent if
there exists a bijection $f$ on $\llbracket 1, n\rrbracket$ that 
\begin{itemize}
\item respects the tree: if $w$ is the parent of $v$ in $T$,
  then $f(w)$ is the parent of $f(v)$ in $T'$,
\item respects abscissas: the abscissas of $v$ in $T$ and $f(v)$ in
  $T'$ are the same.
\end{itemize}
An example in shown in Figure~\ref{fig:equivalence}.
 Finally, an \emm $\cS$-ary tree, is an equivalence
class of $\cS$-embedded injective Cayley trees. Thus an $\cS$-ary tree can
be seen as an (unlabelled) rooted plane tree, drawn in the plane is
such a way the root lies at abscissa 0, and each vertex $v$ has at most
one child at abscissa $a(v)+s$, for all $s\in \cS$. For instance, a
$\{-1, 1\}$-ary tree is a binary tree that is canonically embedded, as
shown in Figure~\ref{fig:binary-small}. Similarly, a
$\{-1, 0, 1\}$-ary tree is a canonically embedded ternary tree~\cite{kuba}.
Since injective trees have no symmetry, the  $n!$ ways one can label
the vertices of a given $\cS$-ary tree of size $n$ give rise to exactly $n!$
distinct injective $\cS$-embedded Cayley trees.
\begin{figure}[htb]
\includegraphics[scale=0.7]{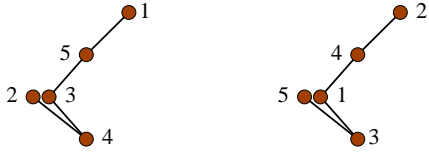}
\caption{Two equivalent embedded Cayley trees (the positions of the
  vertices indicate their  abscissas).}
\label{fig:equivalence}
\end{figure}
\subsection{The vertical profile}
Our first results deal with the number of embedded trees having a
prescribed profile. As all results in this paper, they require
the largest element of $\cS$ to be 1.
This condition is reminiscent of the enumeration of lattice paths in
the half-line $\ns$, for which simple formulas exist provided the set
$\cS$ of allowed steps satisfies $\max \cS=1$;
see for instance~\cite[p.~75]{flajolet-sedgewick}. 
\begin{Theorem}[Embedded Cayley trees]
\label{thm:cayley-profile}
Let $\cS \subset \zs$ such that $\max
  \cS=1$.     Let $\ell\le 0\le r$, and let $(n_i)_{\ell\le i\le r}$ be a
  sequence of positive integers. If $\min \cS=-1$ or $\ell=0$, the
  number of $\cS$-embedded 
  Cayley trees having   vertical profile $(n_\ell, \ldots, n_{-1};
n_0, n_1, \ldots, n_r)$ is
$$\frac{n_0}{n_\ell n_r}\frac {n!}{\prod\limits_{i=\ell}^r
  (n_i-1)!}\prod\limits_{i=\ell}^r\left(\sum_{s\in \cS}n_{i-s}\right)^{n_i-1},
$$
where $n=\sum_i n_i$ is the number of vertices and  $n_i=0$ if $i<\ell$ or $i>r$.
\end{Theorem}

When $\cS=\{-1, 1\}$, this theorem specializes to Theorem~\ref{thm:cayley-profile-pm}.

\medskip
\noindent
{\bf Remarks}\\
1. This formula has several interesting specializations. 
When $0\in \cS$ and $\ell=r=0$, every vertex lies at abscissa $0$
and we are just  counting  rooted Cayley trees of size
$n=n_0$. Accordingly, the above formula is $n^{n-1}$.

If $\cS=\{1\}$, each rooted Cayley tree has a unique $\cS$-embedding,
and a vertex at distance $i$ from the root lies at abscissa $i$. Hence
the vertical profile of the tree coincides with its horizontal
profile. The above theorem thus gives the number of rooted Cayley
trees with horizontal profile $(1, n_1, \ldots , n_r)$ as
$$
\frac {n!}{n_1! \cdots n_r!} \prod_{i=1}^{r-1} n_i^{n_{i+1}}
$$
with $n=1+ n_1 + \cdots +n_r$.
This formula has a straightforward explanation: the multinomial
coefficient describes  the choice of the vertices lying at height $i$,
for all $i$,  
and the factor $n_i^{n_{i+1}}$ describes how to choose a parent for
each vertex lying at height $i+1$. 

If $\cS=\{-1,1\}$ and $\ell=0$, $r=1$, the above theorem gives the
number of bicolored Cayley trees, rooted at a white vertex, having
$n_0$ white vertices and $n_1$ black vertices:
$$
n_0 {n_0+n_1 \choose n_0} n_1^{n_0-1} n_0^{n_1-1}.
$$
Equivalently, the number of spanning trees of the complete bipartite
graph $K_{n_0,n_1}$ is $ n_1^{n_0-1} n_0^{n_1-1}$ (for references on
this result, see
the solution of Exercise 5.30 in~\cite{stanley-vol2}).

 \medskip\noindent
2. The assumption that the numbers $n_i$ are positive is
 not restrictive. Indeed, if $(n_\ell, \ldots;
n_0, \ldots, n_r)$ is the profile of an $\cS$-embedded tree, and the
above conditions on $\cS$ and $\ell$ hold, then $n_\ell, \ldots, n_r$ are positive. Indeed, by 
definition of the profile, $n_ \ell >0$ and $n_r>0$.  Moreover, the fact that the root lies at
abscissa 0, and the condition $\max \cS=1$, imply that $n_0, n_1,
\ldots, n_{r-1}$ are also positive. Finally, if $\ell <0$, then we are
assuming that $\min \cS=-1$, so that, symmetrically, $n_{\ell+1},
  \ldots, n_{-1}$ are positive.

  \medskip\noindent
3. It seems that no simple product formula exists\footnote{See however the note at the end of this
  paper and the more recent paper~\cite{bernardi} for a (less explicit)
  formula that applies  more generally.}  when $\ell <0$
and $\min \cS<-1$.  For instance, when $\cS=\{-2,-1,1\}$, the number
of $\cS$-embedded trees with vertical profile $(1,1,1,2,1;1)$ is 
$6!\,  \frac{3\cdot 107}2$, and $107$ is prime (this number can be
easily obtained from a recursive description of trees, as discussed in
Section~\ref{sec:lagrange}). 

Symmetrically, if  $\cS=\{-1,1,2\}$, there are $6!\,  \frac{3\cdot
  107}2$ $\cS$-embedded trees with vertical profile $(1,1,2,1,1,1)$,
which shows that the assumption $\max \cS=1$ is also needed, even when
$\ell=0$. 

\medskip\noindent
4. 
It suffices to prove the theorem when $\ell=0$. 
Indeed, assume $\ell<0$ and $\min \cS=-1$. We claim that  the number of $\cS$-embedded trees
having profile $(n_\ell, \ldots; n_0, \ldots, n_r)$ and a marked
vertex $v$ at abscissa $\ell$ equals the number of $\cS$-embedded trees
with profile $(m_0, \ldots, m_{r-\ell}):=(n_\ell, \ldots, n_0,
\ldots, n_r)$ (no semi-colon!) having a marked vertex at abscissa $-\ell$. 
This follows from re-rooting  the tree at $v$
(that is, choosing the vertex $v$ as the new root of the tree),
marking the former root
and then translating the abscissas by $-\ell$. 
The resulting tree is a
$\cS$-embedded tree because, when $\min \cS=-1$,
the set of increments is
symmetric. Thus, the number  $T(n_\ell, 
\ldots; n_0, \ldots, n_r)$ of trees having profile $(n_\ell,
\ldots; n_0, \ldots, n_r)$ satisfies
\begin{eqnarray*}
  n_\ell\, T(n_\ell, \ldots; n_0, \ldots, n_r)&=&m_{-\ell} T(m_0, \ldots,
m_{r-\ell})\\
&=&
\frac{m_{-\ell} }{m_{r-\ell}} \prod_{i=0}^{r-\ell} 
\left(\sum_{s\in \cS}m_{i-s}\right)^{m_i-1}  \hskip 2mm (\hbox{case } \ell=0 \hbox{
  of Theorem~\ref{thm:cayley-profile}})\\
&=&
\frac{n_0 }{n_{r}} \prod_{i=\ell}^{r}
\left(\sum_{s\in \cS}n_{i-s}\right)^{n_i-1} \hskip 8mm  (\hbox{because } m_i=n_{i+\ell})
,
\end{eqnarray*}
which gives the announced expression of $ T(n_\ell, \ldots; n_0, \ldots, n_r)$.

\medskip
Let us now state the counterpart of Theorem~\ref{thm:cayley-profile}
for $\cS$-ary trees. 
\begin{Theorem}[$\cS$-ary trees]
\label{thm:S-ary-profile}
Let $\cS \subset \zs$ such that $\max
  \cS=1$.     Let $\ell\le 0\le r$, and let $(n_i)_{\ell\le i\le r}$ be a
  sequence of positive integers. If $\min \cS=-1$ or $\ell=0$, the
  number of $\cS$-ary trees having   vertical profile $(n_\ell, \ldots, n_{-1};
n_0, n_1, \ldots, n_r)$ is
$$
\frac{n_0}{n_\ell n_r}
{\sum\limits_{s\in \cS}n_{-s} \choose n_0-1} 
\prod_{\ell\le i\le r \atop i\not = 0 } {{\sum\limits_{s\in \cS} n_{i-s} -1} \choose{ n_i-1}},
$$
with $n_i=0$ if $i<\ell$ or $i>r$.
\end{Theorem}

When $\cS=\{-1, 1\}$, this theorem specializes to
Theorem~\ref{thm:binary-profile-pm}.

\medskip
\noindent
{\bf Remarks}\\
1. It seems that no simple product formula exists\footnote{We refer again to the more recent paper~\cite{bernardi} for a (less explicit)
  formula that applies  more generally.}  when $\ell <0$
and $\min \cS<-1$.  For instance, when $\cS=\{-2,-1,1\}$, the number
of $\cS$-ary trees with vertical profile $(1,1,1,2,1;1)$ is 
$ 107$, which is prime.

Symmetrically, if  $\cS=\{-1,1,2\}$, there are $  107$ $\cS$-ary trees with vertical profile $(1,1,2,1,1,1)$,
which shows that the assumption $\max \cS=1$ is also needed, even when
$\ell=0$. 

\smallskip\noindent
2. Rerooting  an $\cS$-ary tree does not always give an $\cS$-ary
tree, even if $\min \cS=-1$ (think of re-rooting the first tree of
Figure~\ref{fig:binary-small} at the lowest vertex of abscissa $-1$). Thus the case $\ell <0$ of
Theorem~\ref{thm:S-ary-profile} does \emm not, follow from the case
$\ell=0$, at least in an obvious way. It would  be interesting
to explore the 
combinatorial connection between these two cases.

\subsection{The out-types}
We now prescribe the number $n(i,s)$ of vertices of out-type $(i;s)$, for all
$i \in \zs$ and $s\in \cS$. In particular, the number $n_i$ of vertices at
abscissa $i$ is determined, equal to $\chi_{i=0}+\sum_s n(i,s)$. In other
words, the profile is fixed.

\begin{Theorem}[Embedded Cayley trees]\label{thm:cayley-out}
Let $\cS \subset \zs$ such that $\max \cS=1$.   
Let $n(i,s)$ be non-negative integers, for $i\in\zs$ and $s\in \cS$,
and assume that either $\min \cS=-1$, or $n(i,s)=0$ for all
$i<0$ and $s\in \cS$.
The  number of $\cS$-embedded   Cayley trees in which, for
  all  $i\in \zs$ and $s\in \cS$, exactly $n(i,s)$ non-root vertices have
  out-type $(i;s)$ is    
$$
\frac {n!\prod\limits_{i=\ell}^rn_i^{c(i)-1} \prod\limits_{i=\ell}^{-1} n(i,-1)\prod\limits_{i=1}^r
  n(i,1)
}
{\prod\limits_{i,s} n(i,s)!},
$$
where  $n$ is the number of vertices,
$(n_\ell, \ldots, n_{-1};n_0, \ldots, n_r)$ is the profile
corresponding to the numbers $n(i,s)$,
and $c(i)$ is the number of 
vertices whose parent lies at abscissa $i$:
$$
n=\sum_i n_i, \quad \quad 
n_i= \sum_{s} n(i,s) + \chi_{i=0},
\quad \quad 
c(i)= \sum_{s} n(i+s,s).
$$
When the range of a product or sum is not indicated, it is the
`natural' one ($s\in \cS$, $i\in \zs$). It is assumed that $n_i>0$ for
$\ell \le i \le r$.
\end{Theorem}

\noindent{\bf Remarks}\\
1. An equivalent formulation consists in giving the \gf\ of $\cS$-embedded Cayley
trees of vertical profile $(n_\ell, \ldots; n_0, \ldots, n_r)$, where
a variable $x_{i,s}$ keeps track of the number of vertices of out-type
$(i;s)$. One easily checks that the above theorem boils down to saying
that this \gf\ is
\beq\label{cayley-profile-refined}
\frac{n_0}{n_\ell n_r}
 \frac {n!}{\prod\limits_{i=\ell}^r
  (n_i-1)!}\prod_{i=\ell}^{-1} x_{i,-1}\prod_{i=1}^r
x_{i,1}\prod\limits_{i=\ell}^r\left(\sum_{s\in \cS}n_{i-s}x_{i,s}\right)^{n_i-1}
.\eeq 
This formula refines of course Theorem~\ref{thm:cayley-profile},
obtained by setting $x_{i,s}=1$ for all $i$ and $s$.

\noindent 2. As pointed out in~\cite{bernardi},  this theorem follows
from~\cite[Eq.~(23)]{bousquet-chauve}, upon identifying the cofactor
that occurs in that formula as a number of trees (which is simple to
determine).

\medskip
Let us now state the counterpart of Theorem~\ref{thm:cayley-out}
for $\cS$-ary trees. 
\begin{Theorem}[$\cS$-ary trees]\label{thm:S-ary-out}
Let $\cS \subset \zs$ such that $\max \cS=1$.   
Let $n(i,s)$ be non-negative integers, for $i\in\zs$ and $s\in \cS$,
and assume that either $\min \cS=-1$, or $n(i,s)=0$ for all
$i<0$ and $s\in \cS$.
The  number of $\cS$-ary trees in which, for
  all  $i\in \zs$ and $s\in \cS$, exactly $n(i,s)$ non-root vertices have
  out-type $(i;s)$ is    
$$
\frac {\prod\limits_{i=\ell}^{-1} n(i,-1) \prod\limits_{i=1}^r
  n(i,1)}{\prod\limits_{i=\ell}^{r} n_i} 
\prod\limits_{i,s} {n_{i-s} \choose n(i,s)},
$$
where $(n_\ell, \ldots, n_{-1};n_0, \ldots, n_r)$ is the profile
corresponding to the numbers $n(i,s)$.
Again, it is assumed that $n_i>0$ for
$\ell \le i \le r$.
\end{Theorem}

\noindent
{\bf Remark.} It seems that no simple counterpart
of~\eqref{cayley-profile-refined} exists. That is, the \gf\ of $\cS$-ary
trees of vertical profile $(n_\ell, \ldots; n_0, \ldots, n_r)$,
taking into account the out-types of the vertices, does not factor nicely.

\subsection{The in-types}
We now prescribe the number $n(i,\bmc)$ of vertices of in-type
$(i;\bmc)$, for all $i$ and $\bmc=(c^m, \ldots, c^M)$, with $m=\min
\cS$ and $M=\max \cS$. By definition of the in-types,
it suffices to study the case where $\cS=\llb m,M\rrb$. 

Note that the number $n_i$ of vertices at
abscissa $i$ is determined, equal to $\sum_{\bmsc} n(i,\bmc)$. Hence
the profile $(n_\ell, \ldots; n_0, \ldots, n_r)$ is 
fixed.   The number $n(i,s)$ of vertices of out-type $(i;s)$ is also
determined by the  choice of the numbers $n(i, \bmc)$. Indeed,
$n(i,s)$ is the number of edges going from a vertex of
abscissa $i$ to its parent of abscissa $i-s$, so that
$$
n(i,s)= \sum_{\bmsc} c^sn(i-s, \bmc).
$$
Since we can express  $n_i$ in terms of the numbers $n(i,s)$ or in
terms of the numbers $n(i,\bmc)$,  the following compatibility
condition is required: for $\ell\le i \le r$,
$$
\chi_{i=0}+  \sum_{s,\bmsc} c^s n(i-s, \bmc)
=\sum_{\bmsc} n(i,\bmc).
$$
We will also assume, as before, that $n_\ell, \ldots, n_r$ are positive.

\begin{Theorem} 
\label{thm:cayley-in}
Let $m\le 1$ and $\cS =\llbracket m, 1\rrbracket$.  
Let $n(i,\bmc)$ be non-negative integers, for $i \in \zs$ and $\bmc
\in \ns^{2-m}$, %
satisfying the above compatibility condition.
 Assume moreover  that either $m=-1$, or  $n(i,\bmc)=0$ for
all $i<0$ and $\bmc\in \ns^{2-m}$. 
The  number of $\cS$-embedded   Cayley trees in which, for
  all $i\in \zs$ and $\bmc\in \ns^{2-m}$, exactly $n(i,\bmc)$ vertices
  have in-type $(i;\bmc)$ is    
$$
\frac {n!\prod\limits_{i=\ell}^r(n_i-1)!}
{\prod\limits_{i, \bmsc}n(i,\bmc)!
\prod\limits_{ b\ge 0,s}b!^{n_s(b)}}
\prod\limits_{i=\ell}^{-1}n(i,-1) \prod\limits_{i=1}^{r}n(i,1),
 $$
where   $n$ is the number of vertices, $(n_\ell, \ldots, n_{-1};n_0, \ldots, n_r)$ is the profile
corresponding to the numbers $n(i, \bmc)$,
$n_s(b)$ is the number of vertices $v$ that
have exactly $b$ children at abscissa $a(v)+s$, and 
$n(i,s)$ is the number of vertices of out-type $(i,s)$.
Equivalently,
\begin{align*}
n=\sum_i n_i, \quad \quad 
  n_i= \sum_{\bmsc} n(i,\bmc),
\quad \quad
n_s(b)=\sum_{i} \sum_{\bmsc: c^s=b} n(i,\bmc),
\quad \quad
n(i,s)= \sum_{\bmsc} c^sn(i-s, \bmc).
\end{align*}
\end{Theorem}

\noindent{\bf Remarks} \\
1. Again, this theorem has interesting specilizations. Assume for
instance that $n(i, \bmc)=0$ as soon as $i\not = 0$, and $m=0$. Then
the only non-zero numbers $n(i,\bmc)$ are of the form $n(0,i,0):=k_i$,
giving the number of vertices of the tree having $i$ children. In
particular, $n_0=n$. The above formula reads
$$
\frac{n!(n-1)!}{\prod_i k_i! \prod_i i!^{k_i}},
$$
and gives the number of rooted Cayley trees having $k_i$
vertices of in-degree $i$~\cite[Corollary 5.3.5]{stanley-vol2}.

If $\cS=\{1\}$, each rooted Cayley tree has a unique $\cS$-embedding,
and a vertex at distance $i$ from the root lies at abscissa $i$. The above theorem  gives the number of rooted Cayley
trees in which $k(i,c)$ vertices have height $i$ and (in-)degree $c$,
for all $i$ and $c$:
$$
\frac {n!}{\prod_{i,c} k(i,c)!} \prod_{i=1}^{r} \frac {n_i!}{\prod_c
c! ^{k(i-1,c)}}.
$$
This formula has a direct explanation: the multinomial
coefficient describes  the choice of the vertices of indegree $c$
lying at height $i$, for all $i$ and $c$
and the multinomial $\frac {n_i!}{\prod_c
c! ^{k(i-1,c)}}$ how to assign children to vertices of height $i-1$.

If $\cS=\{-1,1\}$ and $\ell=0$, $r=1$, the above theorem gives the
number of
bicolored Cayley trees, rooted at a white vertex, having 
$k(0,c)$ white vertices of (in-)degree $c$ and $k(1,c)$ black vertices of
(in-)degree $c$ for all $c$:
\beq\label{bicolored-in}
\frac{(n_0+n_1)! (n_0-1)! n_1!}{\prod_{i,c} k(i,c)! \prod_c c! ^{k(0,c)+k(1,c)}},
\eeq
with $n_i=\sum_c k(i,c)$.
This is related to a known formula that gives the number of
spanning trees of the complete bipartite graph $K_{m_0,m_1}$  (with
white vertices labelled $u_1, \ldots, u_{m_0}$ and black vertices
labelled $v_1, \ldots, v_{m_1}$)
in which each vertex has a prescribed degree:
\beq\label{bicolored-total}
\frac{(m_0-1)! (m_1-1)!}{\prod_c! ^{k(0,c)+k(1,c)}},
\eeq
where $k(0,c)$ (resp. $k(1,c)$) is the number of white (resp. black)
vertices of total degree $c+1$ 
(see~\cite[Exercise 5.30]{stanley-vol2} and~\cite{bacher}).
Indeed, \eqref{bicolored-in} can be derived
from~\eqref{bicolored-total}, taken with $m_0=n_0$ and $m_1=1+n_1$,
when $v_1$ has degree 1. Conversely, our results actually allow us to
prescribe the in-type of the root \emm and, the number of white and
black vertices of fixed in-type (see Section~\ref{sec:proof-cayley-in}), and this refined
formula implies~\eqref{bicolored-total}.

\medskip \noindent
2. If $n(i,\bmc)=0$ as soon as one of the
components of $\bmc$ is larger than~1, then the trees counted by the
above formula are 
injective.  Thus it suffices to divide this formula by $n!$ to obtain
the number  of $\cS$-ary trees having $n(i,\bmc)$ vertices
of in-type $(i;\bmc)$ for all $i$ and $\bmc$.

\subsection{The complete types}
We finally prescribe the in-type $(0;\bmc_0)$ of the root and the
number $n(i,s,\bmc)$ of (non-root) vertices of (complete) type $(i;s;\bmc)$, for all
$i$, $s$ and $\bmc$. 
In particular, the number of vertices of out-type
$(i,s)$  is fixed, and can be expressed in terms of the numbers
$n(i,s,\bmc)$ in two different ways. This yields  the following
compatibility condition, for all $i$ and $s$:
$$
\chi_{i=s} c_0^s+ \sum_{t,\bmsc} c^sn(i-s,t,\bmc)=\sum_{\bmsc} n(i,s,\bmc) .
$$
We also assume, as before, that $n_\ell, \ldots, n_r$ are positive.

We have  only obtained a formula when $\ell=0$ and $0\not \in \cS$ (plus the usual condition $\max \cS=1$).
We thus focus on the case $\cS=\llbracket m,-1\rrbracket \cup \{1\}$. 
\begin{Theorem}
\label{thm:cayley-complete}
Let $m\le 1$ and $\cS =\llbracket m, -1\rrbracket \cup \{1\}$.    
Let $\bmc_0=(0, \ldots, 0, c_0^1) \in \ns^{2-m}$.
Let $n(i,s,\bmc)$ be non-negative integers, for $i \in \zs$, $s\in
\cS$ and $\bmc 
\in \ns^{2-m}$, such that  $n(i,s,\bmc)=0$ if $i<0$ or $c^0>0$. 
Assume that the above compatibility condition holds.
The  number of $\cS$-embedded   Cayley trees in which the root has
in-type $(0;\bmc_{0})$ and for
  all $i\in \zs$, $s\in \cS$ and $\bmc\in \ns^{2-m}$, exactly
  $n(i,s,\bmc)$ non-root vertices have type $(i;s;\bmc)$ is  
$$
\frac {c_0^1 \,n!\prod\limits_{i,s}n(i,s)!}
{\prod\limits_{i,s,\bmsc}n(i,s,\bmc)!
\prod\limits_{i=1}^rn(i,1)\prod\limits_{b\ge 0,s}b!^{n_s(b)}}
\prod\limits_{i=1}^{r-1} \left( \sum\limits_{b> 0} b\,  n_1(i,1,b)\right),
$$
where  $n$ is the number of vertices, $r$ is the maximal abscissa,
$n(i,s)$ is the number of vertices 
of out-type $(i;s)$, $n_s(b)$ is the number of vertices $v$ that have
exactly $b$ children at abscissa $a(v)+s$, 
and $n_1(i,1,b)$ is the number of vertices of out-type $(i;1)$  that
have exactly $b$ children at abscissa $i+1$. 
That is,
\begin{align*}
n(i,s)= \sum_{\bmsc} n(i,s,\bmc), \quad \quad 
n_s(b)=\sum_{i,t} \sum_{\bmsc: c^s=b} n(i,t,\bmc)
+ \chi_{s=1, b= c_0^1},
\quad \quad n_1(i,1,b)=\sum_{\bmsc: c^1=b} n(i,1,\bmc).
\end{align*}
\end{Theorem}

\section{A  bijection for non-negative trees}
\label{sec:basic}
Let $n_0, \ldots, n_r$ be a
sequence of positive integers. Let $V= \cup_{i=0}^r V_i$ with
$V_i=\{i^1, \ldots, i^{n_i}\}$. The elements of $V$ are called \emm
vertices,, and the vertices of $V_i$ are said to have \emm abscissa,
$i$. The abscissa of a vertex $v$ is denoted by $a(v)$.
The vertices are totally ordered as follows:
$$
i^k \le j^p \Longleftrightarrow (i<j) \hbox{ or } (i=j \hbox{ and } k\le
p).
$$
We consider partial functions $f$ on $ V$, which we regard as
digraphs on the vertex set $V$: for each
vertex $v$ such that $f(v)$ is defined,
 an arc joins $v$ to $f(v)$. 

Let $\cS \subset \zs$ with $\min \cS=m$  and $\max \cS=1$. 
A partial function $f$ on $V$ is an \emm $\cS$-function, if for all $v=i^k$
such that $f(v)$ is defined,   
$$
f(i^k) \in \cup_{s\in \cS} V_{i-s}.
$$
The \emm type, (or: \emm complete type,) of the vertex $v$ is
$(i; s; c^m,  \ldots, c^1)$, where $i$ is the abscissa of $v$,
$i-s$ is the abscissa of $f(v)$ (or $\varepsilon$ if $f(v)$ is not
defined), and for $m\le k\le 1$, $c^k$ is the 
number of pre-images of $v$ at abscissa $i+k$. We often denote
$\bmc=(c^m, \ldots, c^1)$. The \emm out-type, of $v$ is
 $(i;s)$. Its \emm in-type, is $(i; \bmc)$. 
In the function shown in Figure~\ref{fig:function-ex}, which is an
$\cS$-function for $\cS=\llbracket -2, 1\rrbracket$, the vertex  $0^1$
has type $(0;\varepsilon; 0,0,0,1)$, and  the vertex
 $2^1$ has type   $(2; 1; 1, 0,0,1)$.
An edge of  a digraph is an \emm $\cS$-edge, if it joins a
vertex of abscissa $i$ to a vertex of abscissa $i-s$, for $s\in \cS$.

\begin{figure}[htb]
\scalebox{0.9}{\input{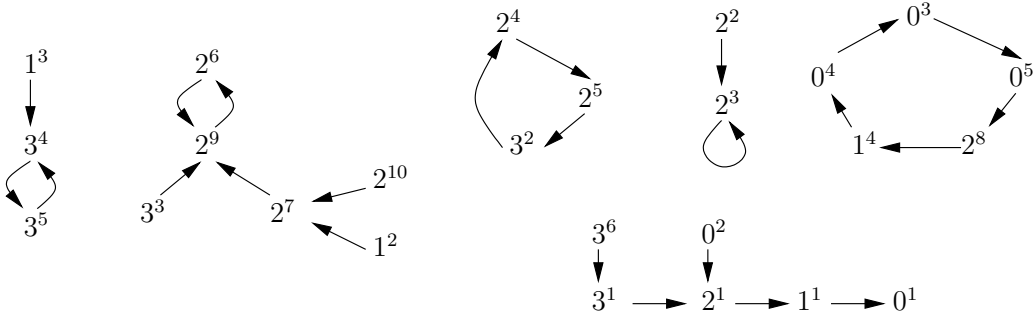}}
\caption{A $\{-2,-1,0,1\}$-function on $V=\{0^ 1, \ldots, 0^5, 1^1,
  \ldots, 1^4, 2^1, \ldots, 2^{10}, 3^1, \ldots, 3^6\}$.
}
\label{fig:function-ex}
\end{figure}

Consider now a rooted tree $T$ on the vertex set $V$.
We say that $T$ is  an \emm $\cS$-tree, if the parent of
any (non-root) vertex of $V_i$ belongs to $\cup_{s\in \cS} V_{i-s}$. We orient
the edges of $T$ towards the root: then $V$ can be seen as an
$\cS$-function. This allows us to define the type, in-type and
out-type of a vertex.
A \emm marked $\cS$-tree, is a pair $(T,r^q)$ consisting of an
$\cS$-tree and a marked vertex $r^q\in V_r$.
There is a simple connection between $\cS$-trees and $\cS$-embedded
Cayley trees,  which will be made explicit in the next section. We
focus for the moment on $\cS$-trees.

\begin{Theorem}\label{thm:basic}
  Let $n_0, \ldots, n_r$, $V$ and $\cS$ be as above.
There exists a bijection $\Phi$ between $\cS$-functions  $f:
V\setminus\{0^1\} \rightarrow V$ satisfying
\begin{itemize}
\item [{\mbox(F)}] for $1 \le i \le r$, $f(i^1)= (i-1)^1$,
\end{itemize}
 and marked $\cS$-trees $(T,r^q)$ on the
  vertex set $V$, rooted at the vertex $0^1$ and satisfying
\begin{itemize}
\item [{\mbox(T)}] on the path going from 
$r^q$ to $0^1$, the first vertex belonging to $V_{i-1}$ is preceded by
$i^1$, for all $i \in \llbracket 1, r \rrbracket$.
\end{itemize}
(The condition $\max \cS=1$ implies that this path contains a vertex
of $V_{i-1}$ for all $i$.)

\medskip
Moreover, the bijection $\Phi$
\begin{enumerate}
\item[(a)]
 preserves the out-type of every vertex;
\item[(b)]
 preserves the number of vertices of in-type $(i;\bmc)$, for all $i$
 and $\bmc$; 
\end{enumerate}
\begin{enumerate}
\item[(c)]
 preserves the complete type of each vertex if $0\not \in \cS$. Of
 course, this implies {\rm{(a)}} and {\rm{(b)}}.
\end{enumerate}
\end{Theorem}
\begin{proof} 
   The bijection $\Phi$ is the composition of two bijections $\Phi_1$
   and $\Phi_2$.  The first bijection, $\Phi_1$, transforms an
$\cS$-function into an $\cS$-tree. It is a simple adaptation of 
a construction
designed by Joyal to count Cayley trees~\cite[p.~16]{joyal}.
It satisfies
Properties (a) and (c). However, if $0 \in \cS$, it does not
satisfy (b). The second bijection, $\Phi_2$, remedies this problem by
performing a simple re-arrangement of subtrees. 
 If $0\not\in \cS$, then $\Phi_2$ is the identity,  so that $\Phi=\Phi_1$.

\begin{figure}[htb]
\scalebox{0.9}{\input{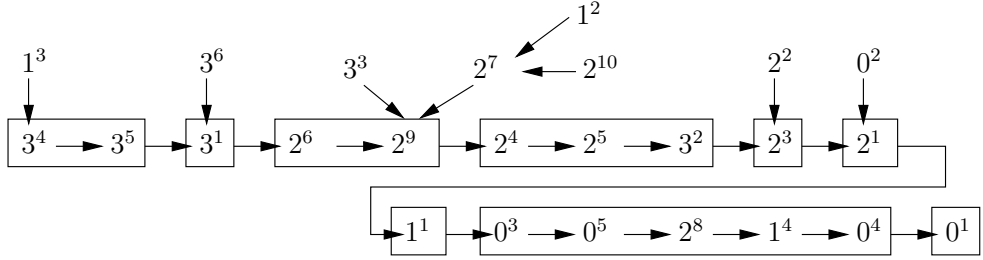}}
\caption{The map $\Phi_1$  applied to the function of Figure~\ref{fig:function-ex}. The marked vertex is $3^4$.}
\label{fig:basic}
\end{figure}

\noindent$\bullet$ {\bf The map $\Phi_1$: from functions to trees}\\
Let us describe the construction $\Phi_1$. It  is illustrated in
  Figure~\ref{fig:basic}, where we construct the tree associated with
  the function of Figure~\ref{fig:function-ex} (which satisfies indeed
  Condition (F)). 

Let $f:V\setminus\{0^1\} \rightarrow V$ be a function satisfying (F), and
consider the associated digraph $G_f$. One of its connected components
contains the path $r^1\rightarrow (r-1)^1 \rightarrow \cdots
\rightarrow 1^1\rightarrow 0^1$. Split each of the $r$ edges of this
path into two half-edges (one in-going, one out-going), thus forming
$r+1$ \emm pieces,. Each piece is of the form
\begin{center}
  \scalebox{0.8}{\input{piece1.pstex_t}}.
\end{center}
We say that $i^1$ (or $r^1$, or $0 ^1$) is the \emm source, (and also the
\emm sink,) of this
piece.

Each of the other components of $G_f$ contains exactly one  cycle. In each of
them, split into two half-edges the edge that goes into the smallest
vertex, say  $i^k$, of the cycle. The vertices that formed the cycle now form
a distinguished path in the resulting \emm piece,, starting from
$i^k$. Trees are attached to the vertices of this path, as shown below:
\begin{center}
  \scalebox{0.8}{\input{piece2.pstex_t}}.
\end{center}
We say that $i^k$ is the \emm source, of this piece.   The endpoint of
the distinguished path is the \emm sink, of the piece. We say that the
piece has \emm type, $(i,j$).
\begin{Lemma}\label{prop:piece}
  If $i^k$ and $j^p$ are respectively  the source and the sink of a
  piece and $k\not = 1$, then $j=i+1$, or $j=i$ with $p \ge k$. The latter case only occurs if $0 \in \cS$.
\end{Lemma}
\begin{proof}
 We have $j^p
\ge i^k$ (since $i^k$ was minimal in its cycle). Given that
$f(j^p)=i^k$, and $\max \cS=1$, this
means that either $j=i+1$, or $j=i$ with $p \ge k$. 
\end{proof}

\begin{figure}[b]
  \centering
  \scalebox{0.8}{\input{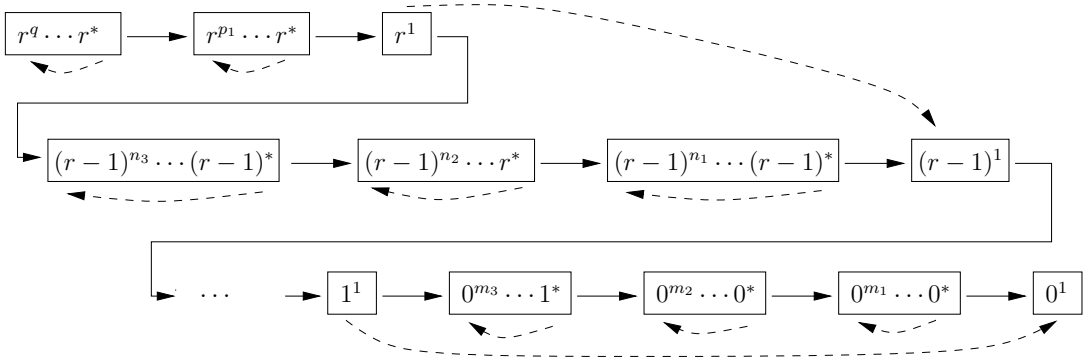}}  
\caption{A typical path from $r^q$ to $0^1$, after concatenation of
  the pieces. It is assumed that  $q>p_1$, $n_3>n_2>n_1$ and
  $m_3>m_2>m_1$. The boxes show the distinguished  paths of the
  pieces. The edges that join the boxes were not in the graph of the
  function $f$,
unless they are of the form $(i+1)^1 \rightarrow i^1$. 
 The dashed edges 
that are not of the form $(i+1)^1 \rightarrow i^1$
 were in the graph   of $f$, but are not in the final tree.} 
 \label{fig:concatenation}
\end{figure}

Now, order the pieces from left to right by decreasing source: the first
(leftmost) piece has source
$r^q$, for some 
$q \in \{1, \ldots, n_r\}$, and the last (rightmost) piece has
source $0^1$. Concatenate them to form a single path going from $r^q$ to
$0^1$, keeping in place the attached trees. Note that there is a piece
of source  $i^1$, for all $i \in 
\llbracket 0,r\rrbracket$. A typical path from $r^q$
to $0^1$ is shown in Figure~\ref{fig:concatenation}.  The resulting graph is
 a tree rooted at $0^1$. In this tree, we mark the vertex $r^q$. We
 define this marked tree to be $\Phi_1(f)$.

 \begin{Proposition}\label{prop:Phi1}
   The map $\Phi_1$ satisfies all properties stated in Theorem~{\rm\ref{thm:basic}},
   apart from {\rm (b)}.
 \end{Proposition}
 \begin{proof}
 Let us first prove that $\Phi_1$ is injective.
  Observe that  the sources of the pieces are the lower records met on the
 path from $r^q$ to $0^1$. This allows us to recover the
 collection of pieces by splitting into two half-edges each edge of
 this path that  goes into a lower 
 record. From the pieces, it is easy to reconstruct the function
 $f$: one adds an edge from $i^1$ to $(i-1)^1$ for
$i\in\llbracket 1,r\rrbracket$, and then connects the two extremal half-edges in
each remaining piece to form a cycle. Thus  $\Phi_1$ is injective. 

 It is clear from the construction (and its illustration in
Figure~\ref{fig:concatenation}) that the out-types are preserved. In
particular, $\Phi_1(f)$ is an $\cS$-tree. It satisfies (T) by
construction.
Note that the \emm in-type, of vertices may change. However, this only
happens for source vertices. In the example of Figure~\ref{fig:basic}, the
in-type of each source other than $1^1$ has changed.

Let us now prove (c). Assume $0\not \in \cS$. By
Lemma~\ref{prop:piece}, a piece of source $i^k$ and sink
$j^{p}$ satisfies  $j=i+1$, provided $k>1$. Thus the  path going
from $r^q$ to $0^1$ in
$\Phi_1(f)$ has a simpler form:
 \begin{center}
 \scalebox{0.8}{\input{concatenation-simple.pstex_t}}     
 \end{center}
In particular, the marked vertex is $r^1$.
It is now easy to check that each vertex has the same type in $\Phi_1(f)$ as in
$f$, so that (c) holds.

It remains to prove that $\Phi_1$ is surjective. Take
a marked $\cS$-tree $(T,r^q)$, rooted at $0^1$ and satisfying (T).
Let us call the path going from $r^q$ to $0^1$ the \emm distinguished
path, of $T$. Each lower record of this path is called a \emm source,. Each
vertex that precedes a source (plus the vertex $0^1$) is called a \emm
sink,. In the distinguished
path,  split into two half-edges each edge that goes into a
source. This gives a number of \emm pieces,, formed of a distinguished 
path starting at a source and ending at a sink, to which trees are attached.  By Condition (T), each
vertex $i^1$ is a lower record, followed  (when
$i>0$) by another one. Thus there is a piece of source and sink $i^1$, for all $i \in
\llbracket 0, r \rrbracket$. Concatenate the pieces of source $i^1$ to form a path $r^1\rightarrow \cdots \rightarrow 0^1$. In each
of the other pieces, close the distinguished path to form a cycle. The
resulting graph $G$ is the graph of a function $g:V\setminus\{0^1\}
\rightarrow V$. It satisfies (F) by construction.
As observed when proving injectivity, if $(T,r^q)=\Phi_1(f)$, where $f$ is
an $\cS$-function satisfying (F), then $f$ and $g$ coincide.

We claim that $g$ is always an $\cS$-function.  
We only need to  check
that each edge of $G$ that was not in the tree is an $\cS$-edge.
Since $1 \in \cS$, the edges $r^1\rightarrow \cdots \rightarrow
0^1$ are $\cS$-edges. Now consider  a piece of source $i^k$, with
$k>1$. Its sink $j^p$ is followed in $T$ by a lower record that has also abscissa
$i$ (because one of the lower records is $i^1$), say $i^m$. The
edge we add to form a cycle is $j^p \rightarrow i^k$, and it is an
$\cS$-edge since the edge $j^p \rightarrow i^m$ was an $\cS$-edge of $T$.
 \end{proof}

 Proposition~\ref{prop:Phi1} says that $\Phi_1$ fulfills almost all requirements of
 Theorem~\ref{thm:basic}, apart from Condition~(b). More specifically, the
 in-types of sources may change, and this only happens if $0\in
 \cS$. Our second transformation, $\Phi_2$, performs on the tree
 $(T,r^q):=\Phi_1(f)$ a little surgery that remedies this problem.

\medskip
\noindent$\bullet$ {\bf The map $\Phi_2$: rearranging subtrees}\\
Let $(T,r^q)$ be a marked tree satisfying the conditions of
Theorem~\ref{thm:basic}. Let $f=\Phi_1^{-1}(T,r^q)$.  Consider the section
of the distinguished path of $T$ comprised between the 
 vertices $(i+1)^1$ and $i^1$ (we assume for the moment that
 $i<r$). The first piece of 
 this section has source and sink $(i+1)^1$, while all other pieces,
 including the final one with source $i^1$, have type
 $(i,i)$ or $(i,i+1)$:
\begin{center}
 \scalebox{0.8}{\input{portion.pstex_t}}   
\end{center}
We say that a source $v$ of abscissa $i$ is \emm frustrated, if its
 in-type is not the same in $T$ and $f$. The above
 figure, in which frustrated sources are indicated with an exclamation
 mark, shows that this happens in two  cases. 
 \begin{Observation}\label{obs:frustrated-i}
 A source $v$ of abscissa $i<r$ is  frustrated if and only if one of
 the following conditions holds:
  \begin{itemize}
 \item the sink of the piece containing $v$ has abscissa $i$, but the
   sink of piece that    precedes $v$ has abscissa $i+1$,
\item   the sink of the piece containing $v$ has abscissa $i+1$, but the
   sink of piece that    precedes $v$ has abscissa $i$.
 \end{itemize}
 \end{Observation}
 In the former (resp.~latter) case, we say that $v$ is \emm $i$-frustrated,
(resp.~$(i+1)$-frustrated). The key observation that will allow us to
preserve in-types after a small surgery is the following.
\begin{Observation}\label{obs:frustrated-ii}
  Consider the frustrated vertices of abscissa $i$ met on the
  distinguished path of $T$. Then $i$-frustrated vertices and
  $(i+1)$-frustrated vertices alternate, starting with an $i$-frustrated
  vertex, and ending with an $(i+1)$-frustrated vertex.
\end{Observation}
Recall that the marked tree $(T,r^q)$ consists of subtrees attached to a
distinguished path. If there are $2m$ frustrated vertices of abscissa
$i$, say $i^{k_1}, \ldots, i^{k_{2m}}$, to which the subtrees $\tau_1,
\ldots, \tau_{2m}$ are attached, exchange the subtrees $\tau_{2p-1}$ and
$\tau_{2p}$ for all 
$p \in \llbracket 1, m\rrbracket$
(leaving their roots $i^{k_1}, \ldots, i^{k_{2m}}$ in place).
 In the resulting tree, $i^{k_{2p-1}}$ inherits the in-type
that $i^{k_{2p}}$ has in the function $f$, and vice-versa. The in-type
of non-frustrated vertices has not changed.

\smallskip
The surgery is  simpler for frustrated vertices of
abscissa $r$, because there are no pieces of type $(r,r+1)$. 
\begin{Observation}\label{obs:frustrated-r}
 If the marked vertex  is $r^q=r^1$, there are no frustrated vertices 
at abscissa $r$. Otherwise, the only two frustrated vertices at
abscissa $r$ are $r^1$ and the marked vertex $r^q$.
\end{Observation}
In the latter case we simply exchange the subtrees attached to $r^1$
and $r^q$, so that $r^q$ inherits the in-type that $r^1$ has in 
$f$, and vice-versa.

 We define $\Phi_2(T,r^q)$ to be the marked tree obtained
 after rearranging the subtrees of $(T,r^q)$, and $\Phi:=\Phi_2\circ
 \Phi_1$. An example is shown  in Figure~\ref{fig:basic-step2},
 where we rearrange 
the subtrees of the tree $(T,r^q)=\Phi_1(f)$ shown in Figure~\ref{fig:basic}.

\begin{figure}[htb]
\scalebox{0.9}{\input{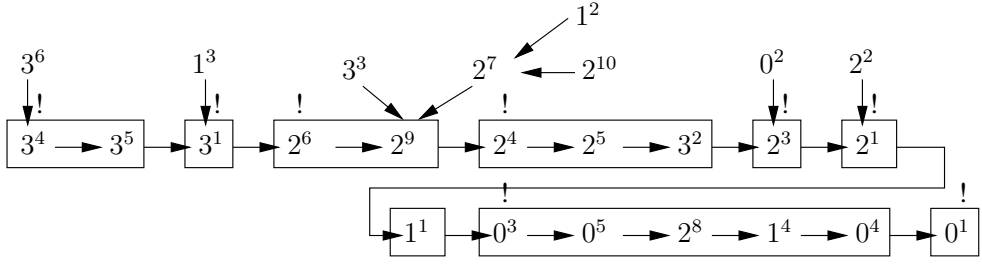}}
\caption{The map $\Phi_2$  applied to the tree $(T,r^q)$ of
  Figure~\ref{fig:basic} 
 (with $r^q=3^4$). 
The frustrated vertices are the same in $(T,r^q)$
  and $\Phi_2(T)$, and are indicated by ``!''.}
\label{fig:basic-step2}
\end{figure}

We claim that $\Phi$ satisfies
 Theorem~\ref{thm:basic}. Of course, the proof of this fact uses the properties of $\Phi_1$ stated  in
 Proposition~\ref{prop:Phi1}. 

First, observe that the out-types are the same in $T$ and $\Phi_2(T,r^q)$
(because vertices pointing to a frustrated source of abscissa $i$ in
$T$ also point to a source of abscissa $i$ in $\Phi_2(T,r^q)$). Hence
$\Phi$ satisfies (a). It also satisfies (c), because there are no
frustrated vertices if $0\not \in \cS$, so that $\Phi_2$ leaves all trees
unchanged. Finally, the surgery performed by $\Phi_2$ is precisely
designed to ``correct'' the in-types of frustrated vertices, while
leaving unchanged the in-types of non-frustrated vertices. Hence
$\Phi$ satisfies (b).

It remains to prove that $\Phi_2$ is bijective. A stronger property
actually holds: $\Phi_2$ is an involution. Indeed, the distinguished
paths of  $(T,r^q)$ and $\Phi_2(T,r^q)$ coincide, and Observations~\ref{obs:frustrated-i}
and~\ref{obs:frustrated-r} imply that the frustrated
vertices of $(T,r^q)$ and $\Phi_2(T,r^q)$ are the same. Applying $\Phi_2$ to
$\Phi_2(T,r^q)$ just restores the marked tree $(T,r^q)$.
\end{proof}

\section{Enumeration of non-negative embedded trees}
\label{sec:enum-nonneg}
In this section, we prove the  enumerative results of
Section~\ref{sec:main} in the case $\ell=0$, that is, for non-negative trees. 
These results follow from the bijection of 
Theorem~\ref{thm:basic}, combined with the enumeration of
$\cS$-functions (which, as we shall see,  is an elementary exercise). We also need to
relate the $\cS$-trees 
occurring in Theorem~\ref{thm:basic} to the $\cS$-embedded
Cayley trees  of Section~\ref{sec:main}. This is done in the following lemma.
We adopt  the same notation as in the previous section:  $V= \cup_{i=0}^r V_i$ with
$V_i=\{i^1, \ldots, i^{n_i}\}$, and $\cS\subset \zs$ satisfies $\min
\cS=m$ and $\max \cS=1$.  The \emm type distribution, of a tree is the
collection  of numbers $n(i,s,\bmc)$ (with $i\in
 \llbracket 0,r\rrbracket$, $s\in \cS$ and $\bmc \in \ns^{-m+2}$)
 giving the number of vertices of type $(i;s;\bmc)$.

\begin{Lemma}\label{lem:link}
   The number of non-negative
 $\cS$-embedded Cayley trees having a prescribed type distribution is
$$
\frac 1 {n_r} \frac {n!}{\prod_{i=0}^r (n_i-1)!}
$$
times the number of marked $\cS$-trees satisfying Condition $(\rm T)$ of
Theorem~{\rm\ref{thm:basic}} and having the  same  type distribution (as
always, $(n_0, \ldots, n_r)$ denotes the profile of the tree, and $n$
its size).
\end{Lemma}
\begin{proof}
Equivalently, we want to prove that the number  of non-negative
 $\cS$-embedded Cayley trees having a prescribed type distribution 
\emm and  a marked vertex at abscissa~$r$, is
$$
 \frac {n!}{\prod_i (n_i-1)!}
$$
times the number of marked $\cS$-trees satisfying  (T) and having the  same  type
distribution. We will construct a 1-to-$n!/\prod_i (n_i-1)!$
correspondence between marked
$\cS$-trees satisfying (T) and marked $\cS$-embedded Cayley trees,
preserving  the type distribution.

Let $(T,r^q)$ be a marked $\cS$-tree on $V$ satisfying (T). 
On the path going from $r^q$ to $0^1$, the  first vertex
of $V_{i-1 }$ is preceded by $i^1$, for $1\le i \le r$. Let us rename
the vertex $i^1$ by $i^k$, for a $k$ chosen in
$\{1, \ldots, n_i\}$; conversely, let us rename $i^k$ by $i^1$. Let us
also exchange the names of the vertices $0^1$ and $0^p$, for a $p$ chosen in $\{1, \ldots, n_0\}$. This gives
an arbitrary  $\cS$-tree $T_1$, rooted at $0^p$, with a marked vertex at abscissa $r$. This tree may
or may not satisfy (T).
The number of different trees $T_1$ that can be constructed  from $T$
in such a way is $\prod_{i=0}^r n_i$.
The marked tree $(T,r^q)$ can be recovered
from $(T_1,r^q)$ by restoring vertex names, since we know that $T$ is
rooted at $0^1$ and satisfies (T). 

Let us now assign labels from $\{1, \ldots,
n\}$, with $n=\sum_i n_i$, to the vertices of $T_1$, in such a way
that the labels $a$ and $b$ assigned to $i^k$ and $i^p$ satisfy
$a<b$ if $k<p$. There are $n!/\prod_i n_i!$ ways to do so. Finally,
erase all names $i^k$ from the tree, for all $i$ and $k$. This gives
an arbitrary rooted $\cS$-embedded
Cayley tree $T_2$, with a marked vertex at abscissa $r$. The tree
$T_1$ can be recovered from $T_2$ by renaming the vertices of abscissa
$i$ with $i^1, \ldots, i^{n_i}$ in the unique way that is consistent
with the order on labels: if two vertices of labels $a$ and $b$, with
$a<b$, lie at abscissa $i$, then  their names $i^k$ and $i^p$ must
satisfy $k<p$. 

 The marked $\cS$-tree $T$ has given rise to $n!/\prod_i (n_i-1)!$
marked embedded trees $T_2$. Moreover,
$T_2$, $T_1$ and $T$ have the same type distribution. The result
follows.
\end{proof}

In what follows, we will count trees by counting functions, using the
correspondence of Theorem~\ref{thm:basic}. We will sometimes prescribe
the type (or in-type, or out-type) of every vertex of $V$, or just the
number of vertices of each type (or in-type, or out-type). We will
always assume that these type distributions are compatible with the
conditions required for the function (by writing ``assuming
compatibility...''). For instance, 
\begin{itemize}
\item if we fix the type $(i_v; s_v; \bmc_v)$ of each vertex $v\in V$
  (or just the in-type or out-type),
  we assume that 
  \begin{itemize}
  \item $i_v=i$ if $v \in V_i$, 
\item $s_v \in \cS$ and $i_v-s_v \in \llbracket 0, r\rrbracket$,
\item $s_v=1$ if $v=i^1$ with $i>0$, and $s_v=\varepsilon$ if $v=0^1$,
  \end{itemize}
\item if the number $n(i,s)$ of vertices of $V$ of out-type $(i;s)$,
  for all $i \in \llbracket 0, r\rrbracket$ and $s\in \cS$,  is
   prescribed, we assume that $n(i,s)=0$ if $i-s \not \in \llbracket
   0, r\rrbracket$, and that
$$
n_i=\sum_{s} n(i,s) +\chi_{i=0},
$$
\item if the number $n(i,\bmc)$ of vertices of $V$ of in-type $(i;\bmc)$,
  for all $i \in \llbracket 0, r\rrbracket$ and $\bmc\in \ns^{2-m}$  is
   fixed (whether we prescribe it directly, or  whether we prescribe
  the in-type of each vertex), we assume that $n(i, \bmc)=0$ if
  $c^s>0$ for some $s\not \in \cS$, that
$$
n_i=\sum_{\bmsc} n(i,\bmc), 
$$
and also that the following condition,
  obtained by counting in two different ways the vertices of $V_i$,
  holds:
$$
\chi_{i=0}+ \sum_{s, \bmsc} c^s n(i-s,\bmc)=\sum_{\bmsc} n(i,\bmc),
$$
\item finally, if we fix the out-type $\bmc_0$ of the vertex $0^1$, and if the number $n(i,s,\bmc)$ of vertices of $V$ of type $(i;s;\bmc)$,
  for all $i \in \llbracket 0, r\rrbracket$, $s\in \cS$  and $\bmc\in
  \ns^{2-m}$,  is also
   fixed (whether we prescribe it directly, or  whether we prescribe
  the type of each vertex), we assume that $n(i, t,\bmc)=0$ if
  $c^s>0$ for some $s\not \in \cS$, that
$$
n_i=\sum_{s,\bmsc} n(i,s,\bmc), 
$$
and also that the following condition,
  obtained by counting in two different ways the vertices of out-type
  $(i,s)$,   holds:
$$
\chi_{i=s} c_0^s+ \sum_{t,\bmsc} c^sn(i-s,t,\bmc)=\sum_{\bmsc} n(i,s,\bmc) .
$$
\end{itemize}

\subsection{The  profile of $\cS$-embedded Cayley trees: 
proof of
  Theorem~\ref{thm:cayley-profile}}
\label{sec:proof-cayley-profile}
By Lemma~\ref{lem:link}, the number of $\cS$-embedded Cayley trees having 
  vertical profile $(n_0, \ldots, n_r)$ is $n!/n_r/\prod_i (n_i-1)!$
  times the number of marked
  $\cS$-trees satisfying (T) and having the same profile.
By Theorem~\ref{thm:basic}, the number of such marked trees is also  the number of $\cS$-functions from $V\setminus\{0^1\}$  to $V$ satisfying
 (F). This number is given by the following
lemma. Theorem~{\rm\ref{thm:cayley-profile}} follows, in the case
$\ell=0$. As explained in the remarks that follow
Theorem~\ref{thm:cayley-profile}, this   suffices to prove this
theorem in full generality (that is, also when  $\ell<0$ and $\min \cS =-1$).
\begin{Lemma}
\label{thm:functions-profile}
The number of $\cS$-functions from $V\setminus\{0^1\}$  to $V$ satisfying
 $(\rm F)$ is
$$
\prod\limits_{i=0}^r\left(\sum_{s\in \cS}n_{i-s}\right)^{n_i-1},
$$
where $n_i=0$ if $i<0$ or $i>r$.
\end{Lemma}
\begin{proof}
We  choose the image of  each vertex of $V_i\setminus\{i^1\}$ in the
set $\cup_s V_{i-s}$, for $i=0, \ldots, r$. 
\end{proof}

\subsection{The  profile of $\cS$-ary trees: 
proof of  Theorem~\ref{thm:S-ary-profile}}
\label{sec:proof-S-ary-profile}
 As discussed at the end of
 Section~\ref{sec:def}, the number of
 $\cS$-ary trees having vertical profile  $(n_0, \ldots, n_r)$ is obtained
 by dividing by $n!$ the number of \emm injective, $\cS$-embedded
 Cayley trees having this profile.  Whether an $\cS$-embedded Cayley tree is
injective can be decided from its type distribution, and more
precisely from its distribution of in-types. Hence by
Lemma~\ref{lem:link}, the number of injective $\cS$-embedded Cayley
trees having    vertical profile $(n_0, \ldots, n_r)$ is
$n!/n_r/\prod_i (n_i-1)!$   times the number of marked injective
  $\cS$-trees satisfying (T) and having the same profile (by \emm
  injective,, we mean again that distinct vertices that lie
at the same abscissa have different parents).
By Theorem~\ref{thm:basic}
(and particularly Property (b) of this theorem),
 the number of such marked trees is also
the number of  $\cS$-functions from $V\setminus\{0^1\}$  to
$V$ satisfying (F) that are injective on each $V_i$. This number is given by the following
lemma. Theorem~{\rm\ref{thm:S-ary-profile}} follows, in the case
$\ell=0$.

\begin{Lemma}
\label{thm:injections-profile}
The number of $\cS$-functions from $V\setminus\{0^1\}$  to $V$,
injective on each $V_i$ and  satisfying  $(\rm F)$ is
$$
{\sum\limits_{s\in \cS}n_{-s} \choose n_0-1} 
\prod_{i=1}^r {{\sum\limits_{s\in \cS} n_{i-s} -1} \choose{ n_i-1}}
\prod\limits_{i=0}^r(n_i-1)! ,
$$
where $n_i=0$ if $i<0$ or $i>r$.
\end{Lemma}
\begin{proof}
For $i=0$, we choose the (distinct) images of  the vertices of
$V_0\setminus\{0^1\}$ in the set $\cup_s V_{-s}$.
There are ${{\sum\limits_{s\in \cS} n_{-s} } \choose{
    n_0-1}}(n_0-1)!$ ways to do so.

For $i\ge 1$, we choose the (distinct) images of  the vertices of
$V_i\setminus\{i^1\}$ in the set $\cup_s V_{i-s}\setminus\{(i-1)^1\}$.
There are ${{\sum\limits_{s\in \cS} n_{i-s} -1} \choose{
    n_i-1}}(n_i-1)!$ ways to do so.

The lemma follows.
\end{proof}

\subsection{The out-types of $\cS$-embedded Cayley trees: 
proof of
  Theorem~\ref{thm:cayley-out}}
\label{sec:proof-cayley-out}
We argue as in Section~\ref{sec:proof-cayley-profile}. By
Lemma~\ref{lem:link} and Theorem~\ref{thm:basic} (in particular
Property (a) of this theorem), the number of $\cS$-embedded Cayley
trees having $n(i,s)$ non-root vertices of out-type $(i;s)$ is
$n!/n_r/\prod_i (n_i-1)!$   times  the number of $\cS$-functions from
$V\setminus\{0^1\}$ to $V$ satisfying (F) and having the
same distribution of out-types. This number is given by the second
part of the following lemma. Theorem~\ref{thm:cayley-out} follows, in
the case $\ell=0$.

  \begin{Lemma}\label{thm:functions-out}
$1$. The number of $\cS$-functions from $V\setminus\{0^1\}$ to $V$ satisfying
 $(\rm F)$ and in which each  $v \in V$  has a prescribed
out-type $(i_v;s_v)$ is, assuming compatibility,
$$
n_r \prod\limits_{i=0}^{r}n_i^{c(i)-1},
$$
where  $c(i)$ is the number of vertices whose image lies in $V_i$:
$$
c(i)=\sharp\{v\in V: i_v-s_v =i\}.
$$

\noindent
$2$. Let $n(i,s)$ be non-negative integers, for $i\in\llbracket 0,r
\rrbracket$ and $s\in \cS$, satisfying the compatibility conditions of
an out-type distribution. The number of $\cS$-functions from $V\setminus\{0^1\}$ to $V$ satisfying
 $(\rm F)$ and in which, for
  all  $i\in \llbracket 0,r \rrbracket$ and $s\in \cS$, exactly
  $n(i,s)$  vertices have   out-type $(i;s)$ is    
$$
\frac {n_r \prod\limits_{i=0} ^r (n_i-1)!\prod\limits_{i=0}^rn_i^{c(i)-1} \prod\limits_{i=1}^r n(i,1)}
{\prod\limits_{i,s} n(i,s)!},
$$
where $c(i)$ is the number of vertices whose image lies in $V_i$:
$$
c(i)= \sum_{s} n(i+s,s).
$$
\end{Lemma}
\begin{proof}
1. We first choose the images of the $c(r)$ vertices that have their
image in $V_r$. There are $n_r^{c(r)}$ possible choices. For $i<r$, we
only choose in $V_i$ the images of the $c(i)-1$ vertices distinct from
$(i+1)^1$ that have their image in $V_i$.  There are $n_i^{c(i)-1}$
possible choices. 

2. We first choose the out-type of every vertex, and then apply the
previous result. For all $i$ and $s$, we must choose the $n(i,s)$
vertices of $V_i$ that have out-type $(i;s)$, keeping in mind that $i^1$ has
out-type $(i;1)$ when $i>0$ and $(0;\varepsilon)$ when $i=0$. Thus the
number of ways to assign the out-types is 
$$
\frac{(n_0-1)!}{\prod\limits_s n(0,s)!} \prod_{i=1}^r \frac{(n_i-1)!}
{(n(i,1)-1)!\prod\limits_{s\not = 1} n(i,s)!}.
$$
The lemma follows. 
\end{proof}

\subsection{The out-types of $\cS$-ary trees: 
proof of
  Theorem~\ref{thm:S-ary-out}}
\label{sec:proof-S-ary-out}

We argue as in Section~\ref{sec:proof-S-ary-profile}. By
Lemma~\ref{lem:link} and Theorem~\ref{thm:basic} (in particular
Properties (a) and (b) of this theorem), the number of $\cS$-ary
trees having $n(i,s)$ non-root vertices of out-type $(i;s)$ is
$1/n_r/\prod_i (n_i-1)!$   times  the number of $\cS$-functions from
$V\setminus\{0^1\}$ to $V$ satisfying  (F) that are
injective on each $V_i$ and  have the
same distribution of out-types. This number is given by the second
part of the following lemma. Theorem~\ref{thm:S-ary-out} follows, in
the case $\ell=0$. 
\begin{Lemma}\label{thm:injections-out}
$1$. The number of $\cS$-functions from $V\setminus\{0^1\}$ to $V$,
  injective on each $V_i$,  satisfying
 $(\rm F)$, and in which each  $v \in V$ has a prescribed out-type
 $(i_v;s_v)$ is, assuming compatibility,
$$
\frac 1 {\prod\limits_{i=0}^{r-1} n_i} \prod_{i,s} n(i,s)! {{n_{i-s}} \choose {n(i,s)}},
$$
where $n(i,s)$ is the number of vertices of out-type $(i,s)$.

\noindent
$2$. Let $n(i,s)$ be non-negative integers, for $i\in\llbracket 0,r
\rrbracket$ and $s\in \cS$, satisfying the compatibility conditions of
an out-type distribution.
The number of $\cS$-functions from $V\setminus\{0^1\}$ to $V$, injective on
each $V_i$,  satisfying  $(\rm F)$ and in which, for
  all  $i\in \llbracket 0,r \rrbracket$ and $s\in \cS$, exactly
  $n(i,s)$  vertices have   out-type $(i;s)$ is    
$$
\frac {\prod\limits _{i=0}^r (n_i-1)! \prod\limits_{i=1}^r n(i,1)} 
{\prod \limits_{i=0}^{r-1} n_i} \prod_{i,s} {n_{i-s} \choose n(i,s)}.
$$
\end{Lemma}
\begin{proof}
1. For any $i \in \llbracket 0, r\rrbracket$, and $s\not =1$, we
choose  in $V_{i-s}$ the (distinct) images  of the $n(i,s)$ vertices having out-type
$(i,s)$. There are ${n_{i-s} \choose n(i,s)}n(i,s)!$ ways
to do so.

When $i \in \llbracket 1, r\rrbracket$ and $s=1$,  we
choose  in $V_{i-1}\setminus\{(i-1)^1\}$ the images  of the $n(i,s)-1$ vertices different from $i^1$
having out-type $(i,1)$. There are
$$
{n_{i-1}-1 \choose n(i,1)-1}(n(i,1)-1)!
= {n_{i-1} \choose n(i,1)}\frac{n(i,1)!}{n_{i-1}}
$$
 ways to do so.

Since there is no vertex of out-type $(0;1)$, this concludes the proof of
the first result.

\medskip
\noindent
2. The argument used to prove the second part of
Lemma~\ref{thm:functions-out} can be copied \it{verbatim}.
\end{proof}

\subsection{The in-types: proof of Theorem~\ref{thm:cayley-in}}
\label{sec:proof-cayley-in}
Assume $\cS=\llbracket m, 1\rrbracket$. We argue as in Section~\ref{sec:proof-cayley-profile}. By
Lemma~\ref{lem:link} and Theorem~\ref{thm:basic} (in particular
Property (b) of this theorem), the number of $\cS$-embedded Cayley
trees having $n(i,\bmc)$  vertices of in-type $(i;\bmc)$ is
$n!/n_r/\prod_i (n_i-1)!$   times  the number of $\cS$-functions from
$V\setminus\{0^1\}$ to $V$ satisfying (F) and having the
same distribution of in-types. This number is given by the second
part of the following lemma. Theorem~\ref{thm:cayley-in} follows, in
the case $\ell=0$ (that is, $n(i, \bmc)=0$ if $i<0$).

 \begin{Lemma}\label{thm:functions-in}
 Let $\cS=\llbracket m,1\rrbracket$, with $m\le 1$.\\
 $1$.  The number of $\cS$-functions from $V\setminus\{0^1\}$ to $V$ satisfying
 $(\rm F)$ and in which each  $v \in V$ has a prescribed
in-type $(i_v;\bmc_v)$ is, assuming compatibility,
$$
\frac{\prod\limits_{i=0}^r (n_i-1)!}{\prod\limits_{b\ge 0, s}  b!^{n_s(b)}}
  \prod_{i=0}^{r-1} c_{i^1}^1,
$$
where $n_s(b)$ is the number of vertices $v$ that
have exactly $b$ pre-images at abscissa $a(v)+s$:
$$
n_s(b)=\sharp \{v \in V: c_v^s=b\}.
$$

\noindent
$2.$ Let $n(i,\bmc)$ be non-negative integers, for $i\in\llbracket 0,r
\rrbracket$ and $\bmc\in \ns^{2-m}$, satisfying the compatibility
conditions of an in-type distribution.
The number of $\cS$-functions from $V\setminus\{0^1\}$ to $V$ satisfying
 $(\rm F)$ and in which, for
  all  $i\in \llbracket 0,r \rrbracket$ and $\bmc\in \ns^{2-m}$, exactly
  $n(i,\bmc)$  vertices have in-type $(i;\bmc)$ is    
$$
\frac {n_r\prod\limits_{i=0}^r(n_i-1)!^2}
{\prod\limits_{i, \bmsc}n(i,\bmc)!
\prod\limits_{ b\ge 0,s}b!^{n_s(b)}}
\prod\limits_{i=1}^{r} n(i,1),
$$
where  $n_s(b)$ is the number of vertices $v$ that
have exactly $b$ pre-images at abscissa $a(v)+s$, and 
$n(i,1)$ is the number of vertices of out-type $(i;1)$.
Equivalently,
\begin{align*}
n_s(b)=\sum_{i} \sum_{\bmsc: c^s=b} n(i,\bmc),
\quad \quad
n(i,1)=\sum_{\bmsc} c^1 n(i-1, \bmc).
\end{align*}
\end{Lemma}
\begin{proof}
  1. For $0\le i \le r$, let us choose the images of the $n_i-1$ vertices of
  $V_i\setminus\{i^1\}$. Exactly 
$c_{(i-s)^k}^{s}-\chi_{s=1=k,i>0}$ 
of  them have image   $(i-s)^k$, for all $j$ and $k$. Thus the number of
  $\cS$-functions satisfying the required properties is
$$
\frac{(n_0-1)!}{\prod\limits
  _{s,k}c_{(-s)^k}^{s}!}\prod_{i=1}^r \frac{(n_i-1)!\,c_{(i-1)^1}^{1}}{\prod\limits
  _{s,k}c_{(i-s)^k}^{s}!},
$$
which is equivalent to the first result.

\medskip
2. As an intermediate problem, let us prescribe the in-type
$(i;\bmc_{i^1})$ of all vertices of the form $i^1$, and the number $\tilde
n(i;\bmc)$ of vertices of $V_i\setminus\{i^1\}$ having in-type
$(i;\bmc)$, for all $\bmc\in \ns^{2-m}$. Clearly,
$$
\tilde n (i,\bmc)= n (i,\bmc)- \chi_{\bmsc=\bmsc_{i^1}}.
$$
The number of ways to assign types to  vertices of
$V_i\setminus\{i^1\}$ is
$$
\frac{(n_i-1)!}{\prod_{\bmsc}\tilde n(i, \bmc)!}
=
\frac{(n_i-1)!}{\prod_{\bmsc} n(i, \bmc)!}\, n(i, \bmc_{i^1}) .
$$
Using the first result, we conclude that the number of functions such
that $i^1$ has in-type $(i;\bmc_{i^1})$ and $\tilde
n(i;\bmc)$ of vertices of $V_i\setminus\{i^1\}$ having in-type
$(i;\bmc)$ is
$$
\frac {\prod\limits_{i=0}^r(n_i-1)!^2}
{\prod\limits_{i, \bmsc}n(i;\bmc)!
\prod\limits_{ b\ge 0,s}b!^{n_s(b)}}
\left( \prod\limits_{i=0}^{r-1} c_{i^1}^1 \right)\left( \prod_{i=0}^r n(i, \bmc_{i^1}) \right).
$$

Finally, let us only prescribe the values $n(i,\bmc)$. That is, we
want to sum the above formula over all possible in-types of the vertices
$i^1$, for $i=0,\ldots, r$. Note that only the two rightmost products
depend on the choice of these types. We are thus led to evaluate
$$
\sum_{i=0}^r\sum_{\bmsc_{i^1}\in \ns^{2-m}} \left( \prod\limits_{i=0}^{r-1}
c_{i^1}^1 \right)\left( \prod_{i=0}^r n(i, \bmc_{i^1}) \right)
=
\left(\sum_{\bmsc_{r^1}} n(r,\bmc_{r^1})\right)
\prod_{i=0}^{r-1} \left(\sum_{\bmsc_{i^1} } c_{i^1}^1n(i,
\bmc_{i^1})\right)
.
$$
The first sum, over $\bmc_{r^1} $, is
$n_r$. For $i\le r-1$, the sum over $\bmc_{i^1}$ is the number $n(i+1,1)$ of
vertices of out-type $(i+1;1)$.   This gives the
second result of the lemma.
\end{proof}

\subsection{The complete types: proof of Theorem~\ref{thm:cayley-complete}}
Assume $\cS=\llbracket m,- 1\rrbracket \cup\{1\}$. We argue as in Section~\ref{sec:proof-cayley-profile}. By
Lemma~\ref{lem:link} and Theorem~\ref{thm:basic} (in particular
Property (c) of this theorem), the number of $\cS$-embedded Cayley
trees having $n(i,s,\bmc)$  vertices of type $(i;s;\bmc)$ is
$n!/n_r/\prod_i (n_i-1)!$   times  the number of $\cS$-functions from
$V\setminus\{0^1\}$ to $V$ satisfying (F) and having the
same type distribution. This number is given by the second
part of the following lemma. Theorem~\ref{thm:cayley-complete} follows.

\begin{Lemma}\label{thm:functions-complete}
 Let $\cS=\llbracket m, -1\rrbracket \cup\{1\}$.\\
   $1$.  
  The number of  $\cS$-functions from $V\setminus\{0^1\}$ to $V$ satisfying
 $(\rm F)$  and  in which each  $v \in V$ has a prescribed type
$(i_v;s_v;\bmc_v)$ is, assuming compatibility,
$$
\frac{\prod\limits_{i,s}n(i,s)!}{\prod\limits_{b\ge 0, s}  b!^{n_s(b)}\prod\limits_{i=1}^r n(i,1)}
  {\prod\limits_{i=0}^{r-1} c_{i^1}^1},
$$
where $n(i,s)$ is the number of vertices of out-type $(i;s)$ and
$n_s(b)$ is the number of vertices $v$ that 
have exactly $b$ pre-images at abscissa $a(v)+s$. That is,
$$
n(i,s)=\sharp\{v\in V: i_v=i \hbox{ and } s_v =s\},
\quad \quad n_s(b)=\sharp \{v \in V: c_v^s=b\}.
$$

\noindent
$2.$ Let $\bmc_0=(0,\ldots,0, c_0^1) \in \ns^{2-m}$.
Let $n(i,s,\bmc)$ be non-negative integers, for $i\in\llbracket 0,r
\rrbracket$, $s\in \cS$ and $\bmc\in \ns^{2-m}$, satisfying the
compatibility conditions of a type distribution.
The number of $\cS$-functions from $V\setminus\{0^1\}$ to $V$ satisfying
 $(\rm F)$ and in which $0^1$ has in-type $(0;\bmc_0)$ and  for
  all  $i\in \llbracket 0,r \rrbracket$, $s\in \cS$  and $\bmc\in \ns^{2-m}$, exactly
  $n(i,s,\bmc)$  vertices have type $(i;s;\bmc)$ is    
$$
\frac { c_{0^1}^1 n_r\prod\limits_{i=0}^r(n_i-1)! \prod\limits_{i,s} n(i,s)!}
{\prod\limits_{i,s, \bmsc}n(i,s,\bmc)!
\prod\limits_{ b\ge 0,s}b!^{n_s(b)}
\prod\limits_{i=1}^r n(i,1)}
\prod\limits_{i=1}^{r-1} \left( \sum\limits_{b> 0} b n_1(i,1,b)\right),
$$
where $n(i,s)$ is the number of vertices of out-type $(i;s)$, $n_s(b)$ is the number of vertices $v$ that
have exactly $b$ pre-images at abscissa $a(v)+s$, and $n_1(i,1,b)$ is
the number of vertices having out-type $(i,1)$ and $b$ pre-images  in $V_{i+1}$.
Equivalently,
\begin{align*}
n(i,s)= \sum_{\bmsc} n(i,s,\bmc),
\quad \quad
n_s(b)=\sum_{i,t} \sum_{\bmsc: c^s=b} n(i,t,\bmc)
+ \chi_{s=1, b=c_{0}^1},
\quad \quad
n_1(i,1,b)= \sum_{\bmsc: c^1=b} n(i,1,\bmc).
\end{align*}
\end{Lemma}
\begin{proof}
  1. For $0\le i \le r$ and $s\not = 1$, let us choose the images of the $n(i,s)$ vertices of
  out-type $(i,s)$. Exactly $c_{(i-s)^k}^{s}$ of
  them have image   $(i-s)^k$, for all  $k$. For $s=1$ and
  $i\ge 1$, one
  vertex of out-type $(i;1)$, namely $i^1$, has image $(i-1)^1$ by Condition (F). Thus the number of
  $\cS$-functions satisfying the required properties is
$$
\left( \prod_{i=0}^r \prod_{s\not = 1}\frac{n(i,s)!}{\prod_k c_{(i-s)^k}^s!}\right)
\left( \prod_{i=1}^r
\frac{(n(i,1)-1)!\,c_{(i-1)^1}^1}
{\prod_k c_{(i-1)^k}^1!} \right)
,
$$
which is equivalent to the first result.

\medskip
2. As an intermediate problem, let us prescribe the in-type
$(i;\bmc_{i^1})$ of all vertices of the form $i^1$ (their out-type is
forced), and the number $\tilde 
n(i,s,\bmc)$ of vertices of $V_i\setminus\{i^1\}$ having type
$(i;s;\bmc)$. Clearly,
$$
\tilde n (i,s,\bmc)= n (i,s,\bmc)- \chi_{i\ge 1,s=1,\bmsc=\bmsc_{i^1}}.
$$
The number of ways to assign types to  vertices of
$V_i\setminus\{i^1\}$ is
$$
\frac{(n_i-1)!}{\prod_{s,\bmsc}\tilde n(i, s,\bmc)!}
=\left\{
\begin{array}{ll}
 \displaystyle   \frac{(n_i-1)!}{\prod_{s,\bmsc} n(i, s,\bmc)!}\, n(i,1, \bmc_{i^1}), 
& \hbox{if } i\ge 1;
\\
  \displaystyle \frac{(n_0-1)!}{\prod_{s,\bmsc} n(0, s,\bmc)!},
& \hbox{otherwise.}
\end{array}\right.
$$
Using the first result, we conclude that the number of functions such
that $i^1$ has in-type $(i;\bmc_{i^1})$ and $\tilde
n(i,s,\bmc)$ of vertices of $V_i\setminus\{i^1\}$ have in-type
$(i;s;\bmc)$ is
$$
\frac {\prod\limits_{i=0}^r(n_i-1)! 
\prod\limits_{i,s}n(i,s)!}
{\prod\limits_{i, s,\bmsc}n(i,s,\bmc)!
\prod\limits_{ b\ge 0,s}b!^{n_s(b)}
\prod\limits_{i=1}^{r}n(i,1)}
 \prod\limits_{i=0}^{r-1} c_{i^1}^1  \prod\limits_{i=1}^r n(i,1, \bmc_{i^1}) 
.
$$

Finally, let us only prescribe the in-type of $0^1$ and the values $n(i,s,\bmc)$. That is, we
need to sum the above formula over all possible in-types of the vertices
$i^1$, for $i=1,\ldots, r$. Note that only the rightmost two products
depend on the choice of these types. We are thus led to evaluate
$$
\sum_{i=1}^r\sum_{\bmsc_{i^1}\in \ns^{2-m}}
 \prod\limits_{i=0}^{r-1} c_{i^1}^1  \prod\limits_{i=1}^r n(i,1, \bmc_{i^1}) 
=
 c_{0^1}^1 \left(\sum_{\bmsc_{r^1} } n(r,1,\bmc_{r^1})\right) 
\prod_{i=1}^{r-1} \left(\sum_{\bmsc_{i^1} } c_{i^1}^1n(i,1,
\bmc_{i^1})\right)
.
$$
Given that the  sum over $\bmc_{r^1} $ is $n_r=n(r,1)$, this gives the
second result of the lemma.
\end{proof}

\section{A bijection for general embedded trees}
\label{sec:general}

In this section we adapt
 the bijection of Section~\ref{sec:basic} to the
case where $\ell <0$ and $\min \cS=-1$. 
The main ideas of the bijection are similar: given a function in an
appropriate class, we cut its cycles where they reach their minima, and connect
the resulting pieces together to construct a tree; finally, we
rearrange some subtrees so as to  ensure the conservation of types.

This bijection is however more intricate that the previous one. In
particular, it is twofold: we split our class of functions
 into two subsets, and  use a different construction on each
of them (Propositions~\ref{prop:casA} and~\ref{prop:casB}). 
The (disjoint) union of the images of the two bijections forms a set of
trees which will be related to $\cS$-embedded trees in the next
section.
Moreover, our bijection lacks the right/left symmetry one could expect
from the symmetry of $\cS$. The trees we consider have a marked vertex $r^q$
at abscissa $r$, but no marked vertex at abscissa $\ell$. The vertex
$\ell^1$, however, plays a role similar to $r^q$, but the conditions
satisfied by the vertices on the path from $r^q$ to the root, or on
the path from $\ell^1$ to the root, are not symmetric.
\medskip

In the rest of this section,  $\ell < 0$ and $r\geq 0$ are
integers and $n_\ell ,\ldots, n_r$ is a 
sequence of positive integers. Let $V= \cup_{i=\ell}^r V_i$ with
$V_i=\{i^1, \ldots, i^{n_i}\}$. We extend  the  
notation and definitions of Section~\ref{sec:basic}. In particular
$v\in V$ is a \emm vertex, and $a(v)\in\llbracket \ell , r\rrbracket$
is its \emm abscissa,. We equip $V$ with the same total order as
before:
$$
i^k \le j^p \Longleftrightarrow (i<j) \hbox{ or } (i=j \hbox{ and } k\le
p).
$$
The notions of $\cS$-function, of in/out/complete/types are defined as before.
Also, a rooted tree $T$ on the vertex set $V$
is  an \emm $\cS$-tree, if the parent of
any (non-root) vertex of $V_i$ belongs to $\cup_{s\in \cS} V_{i-s}$.

\begin{Definition}
A \emm  marked $\cS$-tree, is a pair $(T, r^q)$ where $T$ is 
an $\cS$-tree on the vertex set $V$, rooted at a vertex of $V_0$, and $r^q$
a distinguished vertex in $ V_r$. 
\end{Definition}
\begin{Definition}
Let $(T, r^q)$ 
be a  marked $\cS$-tree.  Let
$\ell ^1\wedge r^q$ denote the 
\emm meet,  of $\ell ^1$ and $r^q$ in $T$, that is, the common
ancestor of $\ell ^1$ and $r^q$ that is the farthest from the root. 
We consider  the following properties.
\begin{itemize}
\item[$(\rm T_1)$]  On the path going from $r^q$ to the root,
the first vertex belonging to $V_{i-1}$ is preceded by  $i^1$, for all
$i \in\llbracket 1, r\rrbracket$.
\item[$(\rm T_2')$] On the path going from $r^q$ to root, the vertex $1^1$
appears strictly before $\ell ^1\wedge r^q$. Moreover, on the path going from $\ell ^1$
to the root, 
the last vertex belonging to $V_{i-1}$ is followed by $i^1$, for all
$ i \in \llbracket \ell +1, 0\rrbracket $.
\item[$(\rm T_2'')$] On the path going from $r^q$ to the root, the vertex $1^1$
appears weakly after $\ell ^1\wedge r^q$. 
Moreover, $\ell ^1\wedge r^q$ lies at a positive abscissa, and on the path going from $\ell ^1$
to $\ell ^p\wedge r^q$,  the last vertex of $V_{i-1}$ is followed by
$i^1$,
for all $i \in \llbracket \ell +1, -1\rrbracket $. 
Finally, on the 
path from $r^q$ to the root, $0^1$ precedes the first vertex
of $V_{-1}$, if such a vertex exists; otherwise, 
$0^1$ is the root of the tree. 
\item[$(\rm T _2)$] Either $(\rm T_2')$ or $(\rm T _2'')$ holds.  
\end{itemize}
\end{Definition}

We now state our main result, which is the counterpart of
 Theorem~\ref{thm:basic} for negative trees. As will be
 shown in the next section, it implies all the enumerative results
 stated in Section~\ref{sec:main} in the case $\ell <0$.
\begin{Theorem}\label{thm:rooted}
  Let $n_\ell, \ldots, n_r$, $V$ and $\cS$ be as above.
There exists a bijection $\Phi$ between $\cS$-functions  $f:
V\setminus\{0^1\} \rightarrow V$ satisfying
$$
({\rm F}) \hskip 33mm f(i^1)= \left\{
\begin{array}{lll}
  (i+1)^1 & \hbox{if } \ell  \le i \le -2,
\\
v_0 \in V_0 & \hbox{if } i=-1,
\\
(i-1)^1  & \hbox{if } 1 \le i \le r,
\end{array} \right.\hskip 40mm \ 
$$
 and  marked $\cS$-trees on the
  vertex set $V$ satisfying  $(\rm T _1)$ and $(\rm T _2)$.

Moreover, this bijection 
\begin{enumerate}\item[(a)] preserves the number of vertices of out-type $(j;s)$, for all $j$ and $s$,
\item[(b)]  preserves the number of vertices of in-type $(i;\bmc)$, for all $i$ and $\bmc$.
\end{enumerate}
\end{Theorem}

\noindent {\bf Remark.}
Property (a) follows from (b). Indeed, as already explained in
the case of embedded trees, the number of vertices of out-type $(j;s)$
is completely determined if we know the number of vertices of in-type
$(i;\bmc)$, for all $i$ and $\bmc$. This is why  we will focus on (b)
in the proof. 
We have not found any way of preserving the complete types, and this
is why Theorem~\ref{thm:cayley-complete} only deals with non-negative
trees. 

\subsection{Setup of the bijection, and the  main two cases}
We now start describing the bijection.  Let $f$ be a function from
$V\setminus\{0^1\}$ to $V$ satisfying (F).  As before,
our bijection transforms  the digraph $G_f$ representing $f$.
First, for  each $i \in \llbracket \ell ,r \rrbracket \setminus \{0\}$, we
split the edge going from $i^1$ to $f(i^1)$ into two half-edges. 
We let $\tilde G_f$ be the digraph thus obtained, which contains
vertices, edges and half-edges.

It follows from  (F) that
the connected component of $\tilde{G}_f$ containing the vertex $i^1$, for
$i\in \llbracket \ell , r \rrbracket$, is of the form: 
\begin{center}
  \scalebox{0.8}{\input{piece1-LR.pstex_t}}
\end{center}
 We call \emph{piece} each of these components, and we
say that $i^1$ is the \emm source, of its piece. 
Each of the  remaining components of $\tilde G_f$ contains exactly one
cycle. We say that the smallest vertex in this cycle is the \emm 
source, of this connected component. We now define a partition 
$$V=\biguplus_{i=\ell }^r W_i$$ 
of the vertex set $V$ as follows: $v\in W_i$ if and only if the 
source of the
connected component of $\tilde G_f$ containing $v$ belongs to $V_i$.

Recall that  $v_0:=f(-1^1)$ belongs to $V_0$.
 We will prove the two following propositions, which, taken together,
imply Theorem~\ref{thm:rooted}.
\begin{Proposition}\label{prop:casA}
There exists a bijection between $\cS$-functions $f$ satisfying $(\rm F)$
 and such that $v_0$ belongs to $\bigcup_{i=\ell }^0 W_i,$
and  marked $\cS$-trees satisfying conditions $(\rm T _1)$ and $(\rm T
_2')$. This bijection satisfies~$(\rm b)$.
\end{Proposition}
\begin{Proposition}\label{prop:casB}
There exists a bijection between $\cS$-functions $f$ satisfying $(\rm F)$ and such that $v_0$ belongs to $\bigcup_{i=1}^r W_i,$ and
 marked $\cS$-trees satisfying conditions $(\rm T _1)$ and $(\rm T
 _2'')$. This bijection satisfies~$(\rm b)$.
\end{Proposition}

\subsection{Two kinds of concatenations: the graphs $L(i)$, $R(i)$.}
We will  prove the above two propositions separately, but two basic
constructions  are used in both cases. Each of them produces
a tree, denoted $L(i)$ or $R(i)$, from the
subgraph $\tilde G_f \cap W_i$ (the
restriction of $\tilde G_f$ to the vertex set $W_i$).
Both constructions are based on the concatenation of certain elementary pieces of 
graphs by decreasing minima: in the case of $L(i)$ this concatenation is performed 
from left to right, whereas for $R(i)$ it is performed from right to left.
The first kind of concatenation 
was used in  Section~\ref{sec:basic} to describe the
bijection $\Psi_1$. Here, depending on the case
(Proposition~\ref{prop:casA} or~\ref{prop:casB}), and on the value of
$i\in\llbracket \ell, r \rrbracket$, we will use one concatenation or the other.
\\
 
\noindent{$\bullet$ \bf The left concatenation $L(i)$}. This construction is
used only for $i\in\llbracket 0, r\rrbracket$, and it is similar to the one
used in Section~\ref{sec:basic}. 
Consider all the connected components of $\tilde G_f$ whose 
source belongs to
$V_i\setminus\{i^1\}$. In each of them, split
the edge entering the source
 into two half-edges. As in Section~\ref{sec:basic}, one obtains a
 piece of the form: 
\begin{center}
  \scalebox{0.8}{\input{piece2.pstex_t}}
\end{center}
with  $j=i+1$, or $j=i$ and $p \ge k$.
As before $i^k$ and $j^p $ are called the \emm source, and the \emm sink, of the
piece, respectively.  Then order the pieces  by decreasing sources,
including the 
piece rooted at $i^1$, and
concatenate them to form a path
(Figure~\ref{fig:portion-L}). We denote by $L(i)$  the tree on the vertex set $W_i$ consisting of this
path and all the subtrees that are attached to it.
This tree has a distinguished vertex $a_i\in V_i$ (the leftmost source, that
is, the greatest one) and it is rooted at the smallest source, $i^1$.
Note that the sources of the pieces are the lower records encountered on the
path going from $a_i$ to $i^1$. 
\begin{figure}[h]
  \centering
  \scalebox{1}{\includegraphics{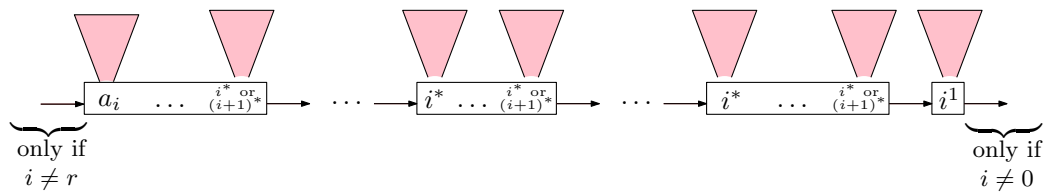}} 
  \caption{The graph $L(i)$. Pieces are ordered by decreasing 
source from left
to right. The leftmost and
rightmost half-edges are present only if $i\neq r$ and $i\neq 0$,
respectively. 
The vertex $a_i$ belongs to $V_i$.} 
\label{fig:portion-L}
\end{figure}

 The following fact is obvious by construction.
\begin{Observation}\label{obs:typeLi}
All the vertices belonging to $W_i$ 
that are not lower records on the path
from $a_i$ to $i^1$ have the same in-type in the function $f$
and in the graph $L(i)$.
\end{Observation}

\noindent{$\bullet$ \bf The right concatenation $R(i)$}. This construction is
used only for $i\in\llbracket \ell , 0\rrbracket$.
Consider all the connected components of $\tilde G_f$ whose
source belongs to $V_i\setminus\{i^1\}$.  
In each of them, split the edge \emm leaving, the source
 into two half-edges. One obtains a  piece of the form
\begin{center}
  \scalebox{0.8}{\input{piece2-R.pstex_t}}
\end{center}
with $j=i+1$ or $j=i$ with $p\ge k$. We call $i^k$ and $j^p$ the \emm
sink, and the \emm source, of the piece, respectively
(even though $i^k$ was called the source before splitting the edge!). 
Note that 
the minimum vertex now lies \emm to the right, of the
piece.
 Then order the pieces, including the piece containing $i^1$, by \emm
 increasing, sinks, and concatenate them to form a path 
(Figure~\ref{fig:portion-R}). We denote by $R(i)$  the tree on the vertex set $W_i$ consisting of this
path and all the subtrees that are attached to it. 
This tree is rooted at a vertex $b_i \in V_i$ (the rightmost sink, that is, the
greatest one) and it contains the vertex $i^1$.
Note that the sinks of the pieces are the lower records encountered on the path
that goes from $b_i$ to $i^1$ (and thus in the direction opposite to
edges).
\begin{figure}[h]
  \centering
  \scalebox{1}{\includegraphics{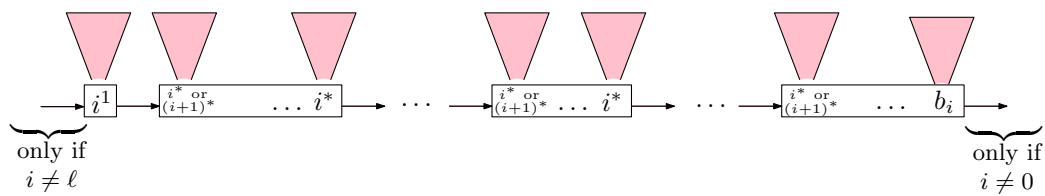}} 
  \caption{The graph $R(i)$. Pieces are ordered by decreasing 
sink from
right to left.  The leftmost and
rightmost half-edges are present only if $i\neq \ell $ and $i\neq 0$,
respectively. 
The vertex $b_i$ belongs to $V_i$.} 
\label{fig:portion-R}
\end{figure}

We now observe an important property of the right-to-left concatenation.
First, 
when opening the cycles to form the pieces,
the source $j^p$ of each piece 
is disconnected from one of its pre-images $i^k$,
which  belongs to
$V_i$. Then, during the
concatenation of pieces, the source $j^p$ 
is reconnected to  the sink of the piece on its left,
which is also an element of $V_i$. Therefore the in-type of each
source distinct from $i^1$ 
is preserved by the construction.  The in-types of all other
vertices are clearly preserved as well. 
\begin{Observation}\label{obs:typeRi}
All the vertices belonging to $W_i$,
distinct from $i^1$,
 have the same in-type in the function $f$
and in the graph $R(i)$.
\end{Observation}

We now prove  Propositions~\ref{prop:casA} and~\ref{prop:casB} separately.
The bijection of Proposition~\ref{prop:casA} is actually split  into
three closely related bijections.
In each case  the bijection reads $\Psi=\Psi_2\circ\Psi_1$ where
$\Psi_1$ is a bijection between the desired set of functions and the desired
set of trees, 
but does not satisfy Property~(b). As in
Section~\ref{sec:basic}, the second bijection $\Psi_2$ is a simple
re-arrangement of subtrees designed in such a way that
$\Psi_2\circ\Psi_1$ satisfies~(b). 

\subsection{Proof of Proposition~\ref{prop:casA}} 
Given a  marked $\cS$-tree $(T,r^q)$
satisfying~$(\rm T _1)$, we denote by $w_0$
the vertex following $1^1$ on the path from $r^q$ to
the root.  This vertex has abscissa $0$.
To prove Proposition~\ref{prop:casA} we distinguish three cases, 
discussed in the following  three lemmas.

\begin{Lemma}\label{lemma:casA1}
There exists a bijection between $\cS$-functions $f$ satisfying $(\rm
F)$ such that 

\begin{itemize}
\item $v_0:=f(-1^1)$ belongs to $\bigcup_{i=\ell }^0 W_i$ 
but neither to a cycle of the graph $\tilde{G}_f$
nor to the connected component of $\tilde{G}_f$ containing $0^1$,  
\end{itemize}
and marked $\cS$-trees $(T,r^q)$ satisfying $(\rm T _1)$ and $(\rm T
_2')$ such that 
\begin{itemize}
\item    $\ell^1\wedge r^q$ is neither  $0^1$ nor $w_0$. 
\end{itemize}
This bijection satisfies Property~$(\rm b)$ of Theorem~\rm{\ref{thm:rooted}.}
\end{Lemma}

\begin{Lemma}\label{lemma:casA2}
There exists a bijection between $\cS$-functions $f$ satisfying $(\rm
F)$ such that 
\begin{itemize}
\item      
$v_0:=f(-1^1)$ belongs to $\bigcup_{i=\ell }^0 W_i$ and 
to a cycle of the graph $\tilde{G}_f$,  
\end{itemize}
and  marked $\cS$-trees $(T,r^q)$ satisfying conditions $(\rm T _1)$ and
$(\rm T _2')$  such that 
\begin{itemize}
\item    $\ell ^1\wedge r^q$ is equal to $0^1$ but distinct from $w_0$.
\end{itemize}
 This bijection satisfies Property~$(\rm b)$ of Theorem~\rm{\ref{thm:rooted}.}
\end{Lemma}

\begin{Lemma}\label{lemma:casA3}
There exists a bijection between $\cS$-functions $f$ satisfying $(\rm
F)$ such that  
\begin{itemize}   
\item      
$v_0:=f(-1^1)$ belongs 
to the connected component of $\tilde{G}_f$ containing $0^1$ 
(and hence to $\bigcup_{i=\ell }^0 W_i$), 
\end{itemize}
and  marked $\cS$-trees $(T,r^q)$ satisfying conditions $(\rm T _1)$ and
$(\rm T _2')$ such that 
\begin{itemize}
\item       $\ell ^1\wedge r^q$ is equal to $w_0$. 
\end{itemize}
This bijection satisfies Property~$(\rm b)$ of Theorem~\rm{\ref{thm:rooted}.}
\end{Lemma}

 Since the connected component of $\tilde{G}_f$ containing $0^1$
contains no cycle, Proposition~\ref{prop:casA} follows immediately from
Lemmas~\ref{lemma:casA1}, \ref{lemma:casA2}, and~\ref{lemma:casA3}, by case
disjunction.

\medskip

\subsubsection{Proof of Lemma~\rm{\bf\ref{lemma:casA1}}}
\label{sec:A1}

Let $f$ be as in the statement of the lemma.
We first construct a marked tree $\Psi_1(f)$ from $\tilde G_f$. The
construction is depicted in Figure~\ref{fig:Psi-A}. 
A second transformation $\Psi_2$ will the  rearrange certain subtrees
of $\Psi_1(f)$.
\begin{itemize}
\item For $i\in \llbracket 1, r\rrbracket$ construct the left concatenation
$L(i)$. Concatenate all these pieces, 
 by decreasing value of $i$, to
obtain a path from the vertex $a_r\in V_r$ to the vertex $1^1$.
\item For $i\in \llbracket \ell , 0\rrbracket$ construct the right concatenation
$R(i)$. Concatenate all these
pieces, by increasing value of $i$, to
obtain a path from $\ell ^1$ to the vertex $b_0\in V_0$.
\item Add an edge from $1^1$ to $v_0$.  Since $v_0\in
\bigcup_{i=\ell }^{0} W_i$, this connects the two previously constructed
components. 
\end{itemize}
We let $(T,a_r):=\Psi_1(f)$ be the marked tree thus obtained. It is rooted at
$b_0$. It is clearly an $\cS$-tree.
\begin{figure}[h]
  \centering
  \scalebox{1}{\includegraphics{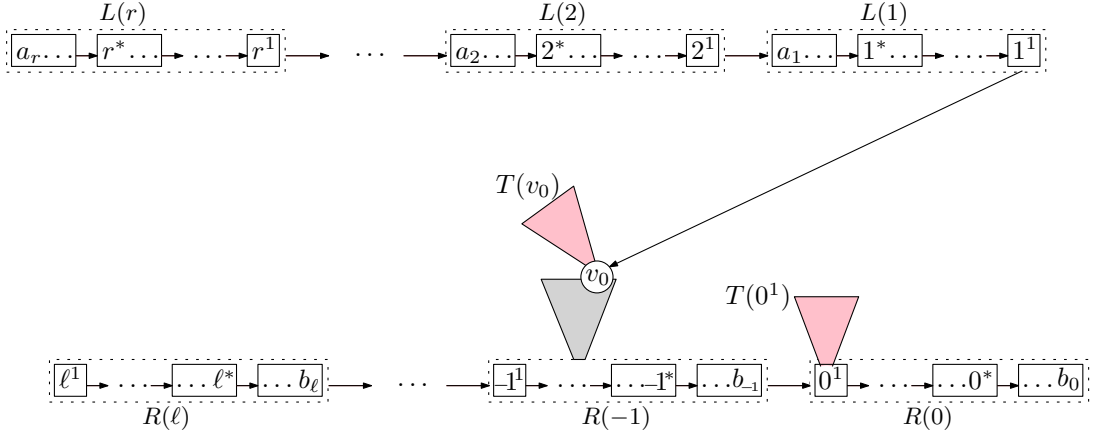}} 
  \caption{The bijection $\Psi_1$ of Lemma~\ref{lemma:casA1}. 
In this example,
the vertex $v_0$ belongs to one of the subtrees of $R(-1)$. In
general, $v_0$ could belong to any subtree attached to the path
$\ell ^1\rightarrow b_0$, except the subtree $T(0^1)$.
  The tree $T(v_0)$, attached to $v_0$ on the path from $a_r$ to the
  root, will be
exchanged with $T(0^1)$ in the construction $\Psi_2$.
}
\label{fig:Psi-A}
\end{figure}

\noindent\boite {\bf The marked tree $(T,a_r)$ satisfies Properties
  $(\rm T _1)$ and $(\rm T _2')$.}
Let $i \in \llbracket 1,r\rrbracket$. Along the distinguished path of
$L(i)$, all vertices have abscissa at least $i$, 
and the rightmost vertex is $i^1$. The vertex that follows $i^1$ on
the path from $a^r$ to $b_0$ is  $a_{i-1}\in V_{i-1}$ if $i>1$,
or $v_0 \in V_0$ if $i=1$. This implies  that $(\rm T _1)$ holds.

Now let $i\in \llbracket \ell+1, 0\rrbracket$.  Along the
distinguished path of $R(i)$, all vertices have abscissa at least $i$,
and  the leftmost vertex  is $i^1$. 
The vertex that precedes $i^1$ on
the path from $\ell^1$ to $b_0$ is  $b_{i-1}\in V_{i-1}$.
This implies  that the second part of $(\rm T _2')$ holds.

Finally, since $v_0\in \bigcup_{i=\ell}^{0} W_i$, it is clear by
construction that $1^1$ appears strictly before 
$\ell ^1\wedge a_r$   on the path from $a_r$ to $b_0$. Hence the first
part of $(\rm T _2')$ holds.

\medskip
\noindent\boite {\bf The meet $\ell ^1\wedge a_r$ is neither $w_0$ nor $0^1$.}
Observe that the vertex $v_0$ follows $1^1$ on the
path from $a_r$ to the root. Hence:
\begin{Observation}
In the marked tree $\Psi_1(f)=(T,a_r)$, the vertex $w_0$ is $v_0=f(-1^1)$.
\end{Observation}
 The meet $\ell ^1\wedge a_r$ belongs to
the path going from $\ell ^1$ to
$b_0$. By assumption, $w_0\equiv v_0$ is not on a cycle of
$\tilde{G}_f$, and  is not $0^1$.
Hence  $v_0$ does not belong to the path going from $\ell ^1$ to
$b_0$, and thus cannot be equal to $\ell ^1\wedge a_r$. 
Moreover, the fact that $v_0$ does not belong to the component of $\tilde{G}_f$ containing $0^1$
implies that $\ell ^1\wedge a_r\neq 0^1$.

\medskip
\noindent\boite {\bf The map $\Psi_1$ is injective.} 
Let us start from the marked tree $(T,a_r)$ and  reconstruct the
function $f$. First, for $i\in \llbracket 1,r\rrbracket $, the graph
$L(i)$ and the 
pieces that constitute it can be recovered by splitting  into two
half-edges each edge that
enters a lower record on the path from $a_r$ to $w_0$. Similarly, on the path that goes from $b_0$ to $\ell^1$
(visited in this direction),
we split into two half-edges all edges that leave a lower record to
recover the graphs $R(i)$, for $i\in \llb \ell, 0 \rrb$, and their
pieces. Then we close
each piece  that does not contain a vertex of the form $i^1$ to form a cycle.
One thus recovers the graph $\tilde G_f$.
Finally, we add an edge from $i^1$ to $(i+1)^1$
for $i\in\llbracket \ell , -2\rrbracket$, an edge from $i^1$ to $(i-1)^1$
for $i\in\llbracket  1,r\rrbracket$, and  an edge from $-1^1$ to 
$w_0$ to recover the graph $G_f$. 

\medskip
\noindent\boite {\bf The map $\Psi_1$ is surjective.} 
Let $(T,r^q)$ be a  marked $\cS$-tree
rooted at $\rho\in V_0$, 
satisfying $(\rm T _1)$ and $(\rm T _2')$.  Let $w_0$ be the
vertex that follows $1^1$ on the path from $r^q$ to $\rho$.
Assume that the meet $\ell ^1\wedge r^q$ is distinct
from $0^1$ and $w_0$.
We first split the edge $1^1\rightarrow w_0$ into two
half-edges, thus creating two connected components: 
one of them contains $r^q$ and
$1^1$, while the other contains $w_0$,  $\rho$,  and $\ell ^1$
(by $(\rm T_2')$).

We first consider the path going from $r^q$  to $1^1$ in the first 
component. Lower records on this
path are called \emm sources,, and vertices preceding the sources are called
\emm sinks, 
(we consider $1^1$ as a sink).
We now split into two half-edges each edge that enters a source on the
path. This gives  a number of \emm pieces,, each of them carrying a
\emm distinguished path, going from a source to a sink.
By  $(\rm T _1)$, each vertex $i^1$ for $i\in\llbracket 1,r \rrbracket$ is
the source and the sink of a piece.
Take  all the pieces containing a vertex of the form $i^1$, for
$i\geq1$, and concatenate them by
adding an edge from $i^1$ to
$(i-1)^1$ for $i\in\llbracket 2, r\rrbracket$.
Transform each of the other pieces into a cycle by 
connecting its sink to its source.

We now visit the path going 
from the root $\rho$ to $\ell^1$ (in this direction).
Lower records on this path  are called  \emm sinks,, and  vertices
preceding a sink (in 
the same ``wrong'' direction) are called \emm sources,
(we consider $\ell^1$ as a source).
We now split all the edges between sinks and sources, and thus obtain a collection of
\emm pieces,. By $(\rm T _2')$, 
 there is a piece of source and sink $i^1$ for $i\in
\llbracket \ell ,0 \rrbracket$.
Take all the pieces containing a  vertex of the form $i^1$,
for $i\le -1$, 
and concatenate them by adding an edge from $i^1$ to $(i+1)^1$ for
$i\in\llbracket \ell , -2\rrbracket$. In all the remaining pieces,
merge the two extremal half-edges to form a cycle.  

Finally, add an edge from $1^1$ to $0^1$, from $-1^1$ to the vertex
$w_0$, and let $H$ be the graph thus obtained.
By construction, $H$ is the graph $G_h$ of a function $h: V\setminus \{0^1\}
\rightarrow V$ satisfying (F). 

Let us prove that $h$ is an $\cS$-function. 
It suffices to check that
the edges we have created are $\cS$-edges. Since $1$ and $-1$ belong
to $\cS$, this is clear for the edges that start from a vertex $i^1$,
for $i\in \llb \ell, r\rrb\setminus\{0\}$.
Consider a  piece  of
source $i^k$, with $i \in \llbracket 1,r \rrbracket$ and $k\not = 1$.
Its sink $j^p$ is followed, on the path from $r^q$ to $1^1$,  by a
lower record of abscissa $i$ (because 
by $(\rm T_1)$, $i^1$ is one of the lower records), say $i^m$. Since
the edge $j^p \rightarrow i^m$ was an $\cS$-edge of $T$, the edge $j^m
\rightarrow i^k$ that we create to construct $H$ is  also an $\cS$-edge.
In brief, the out-type of the sink $j^p$ has not changed.
A similar result holds for pieces of 
sink $i^k$, with $i \in
\llbracket \ell,0 \rrbracket$ and $k>1$: when we close them to form a cycle, the
in-type of the source does not change.
Therefore $h$ is a $\cS$-function.

It remains to prove that $h$ satisfies the three statements of
Lemma~\ref{lemma:casA1} dealing with $h(-1^1)=w_0$.
By construction, $w_0$ belongs in $\tilde G_h$ to a component whose source
lies at a nonpositive abscissa. That is,
$h(-1^1)\in\cup_{i=\ell }^0 W_i$ (the sets
$W_i$'s being understood with respect to the function $h$). 
Given that $\ell ^1\wedge r^q\neq w_0$ by assumption, the vertex $w_0$
does not belong to the path of $T$ going from $\ell ^1$ to the root
$\rho$. Hence it cannot be found in a cycle of $G_h$.
The vertex of the  path of $T$ going from $\ell ^1$ to the root
$\rho$ to which $w_0$ is attached is  $\ell ^1\wedge r^q$, 
which by
assumption is different from $0^1$. Hence $h(-1^1)$ does not belong to
the  component of source  $0^1$ in $\tilde G_h$.

Finally, it is clear by construction that $\Psi_1(h)=(T,r^q)$, so $\Psi_1$ is
surjective.\\

\noindent$\bullet$ {\bf Re-arranging subtrees: the bijection $\Psi_2$.}
We say as before that a vertex $v\in V$ is \emm frustrated, if its in-type is not the
same in $f$ and in $(T, r^q)=\Psi_1(f)$.
 We claim that  the vertices of $\cup_{i=\ell }^0 W_i \setminus
\{0^1,v_0\}$ are not frustrated. 
 This is a direct
consequence of Observation~\ref{obs:typeRi} and of the fact that to concatenate
$R(i-1)$ to $R(i)$, for $i\in\llbracket \ell +1, -1 \rrbracket$ we add
a new incoming 
edge to the vertex $i^1$ coming from $V_{i-1}$, which 
 compensates the deletion of the edge $(i-1)^1 \rightarrow
i^1$ in the construction of $\tilde{G}_f$ from $f$. Together with
Observation~\ref{obs:typeLi}, this implies:
\begin{Observation}\label{obs:type-casA1}
Any vertex distinct from $v_0$, $0^1$ and from the lower records
 of the path from $r^q$ to $1^1$ is not frustrated.
\end{Observation}

We first ``correct'' simultaneously the in-types of $0^1$ and $v_0$.
Let us denote by $T(0^1)$ the subtree attached to the vertex $0^1$ on the path
from $\ell ^1$ to the root in $T$, and by $T(w_0)$ the subtree
attached to the vertex 
$v_0=w_0$ on the path from $r^q$ to the root in $T$. 
By assumption, $v_0$ does not belong to $T(0^1)$. Moreover, $0^1$
cannot belong to $T(w_0)$ (it has no image by $f$, and is by
assumption distinct from $v_0$). Hence the subtrees $T(0^1)$ and
$T(w_0)$ are disjoint. Let us exchange them, and denote by $\tilde T$ 
the resulting tree. Then the in-type of $0^1$ in $\tilde{T}$ equals the in-type
of $v_0$ in the function $f$: indeed, edges contributing to these
in-types  are in both cases  the
edges coming from $T(v_0)$, 
plus an edge coming from $V_{-1}$
(this edge joins $b_{-1}$ to $0^1$ in $\tilde T$ and  $-1^1$ 
to $v_0$ in $f$).
Similarly, the  in-type of $v_0$ in
$\tilde T$ equals the in-type of $0^1$ in $f$, since the edges
contributing to these in-types are in both cases
the edges coming from $T(0^1)$ and the edge coming from $1^1$.
Finally, note that the operation 
$(T,r^q)\mapsto(\tilde T,r^q)$ is an involution since
the exchange of subtrees does not modify the 
marked vertex of the tree.

It remains to correct the in-types  of the lower records of the path
going from $r^q$ to $1^1$ in $\tilde T$. 
We proceed as in  the proof of Theorem~\ref{thm:basic} in
Section~\ref{sec:basic}. 
First, clearly, Observations~\ref{obs:frustrated-i}
and~\ref{obs:frustrated-ii}
hold  for $i\in\llbracket 1,r-1 \rrbracket$, as does
Observation~\ref{obs:frustrated-r}.
As in Section~\ref{sec:basic}, we exchange the subtrees attached to
adjacent frustrated sources of abscissa $i \in \llb 1, r\rrb$. This
corrects the in-type of all of them. 
Let $\Psi_2(T,r^q)$ be the tree obtained after performing these exchanges,
and let $\Psi(f)=\Psi_2\circ\Psi_1(f)$.  Since we have corrected all in-types,  $\Psi$ satisfies
Property~(b). Moreover $\Psi_2$ is again an involution (the lower
records on the path from $r^q$ to $1^1$ do not change when exchanging
subtrees). In particular $\Psi$ is a bijection, and
Lemma~\ref{lemma:casA1} is proved.
\qed

\medskip
\noindent{\bf Remark.}
The above construction $\Psi_1$ could be
applied just as well to functions $f$ such that  $v_0$ is on a cycle of $\tilde G_f$ or in the component of $0^1$.
However, we have not been able (and we
believe that it is not possible) to define the ``re-arranging'' bijection $\Psi_2$
in these two cases. This is why we had to split Proposition~\ref{prop:casA} into
three separate lemmas, based on three slightly different constructions.


 \subsubsection{Proof of Lemma~\rm{\bf\ref{lemma:casA2}}}
The bijection and the proof are very close 
to those of Lemma~\ref{lemma:casA1}, but 
we  need to introduce a variant of the right concatenation.
We assume that $v_0$ belongs to a cycle of the graph
$\tilde{G}_f$. Let $i_0\le 0 $ be  the abscissa of its source.

\smallskip
\noindent{$\bullet$ \bf A variation on
 $R(i_0)$: the graphs   $\tilde{R}(i_0)$ and $\tilde C(i_0)$.} 
Instead of constructing $R(i_0)$ as before, from all components of $\tilde
G_f$ having their source in $V_{i_0}$, we  ignore the 
component  containing $v_0$, and form a
smaller right concatenation $\tilde R(i_0)$ with the remaining 
components (Figure~\ref{fig:portion-R-tilde}).
Then, we open the  cycle containing $v_0$  at the edge
 entering $v_0$. This gives a tree, denoted by $\tilde{C}(i_0)$. This tree
has a distinguished path from $v_0$ to a vertex 
$u\in V_{-1}\cup V_0\cup V_1$ 
(and $u\in V_0$  happens only if $0\in \cS$).
The following analogue of Observation~\ref{obs:typeRi} holds.
\begin{Observation}\label{obs:typeRi-tilde}
All the vertices belonging to 
$W_{i_0}\setminus\{i_0^1,v_0\}$ 
have the same in-type in the function $f$
and in the graph $\tilde R(i_0) \cup \tilde{C}(i_0)$.
\end{Observation}

\begin{figure}[h]
  \centering
 \scalebox{1}{\includegraphics{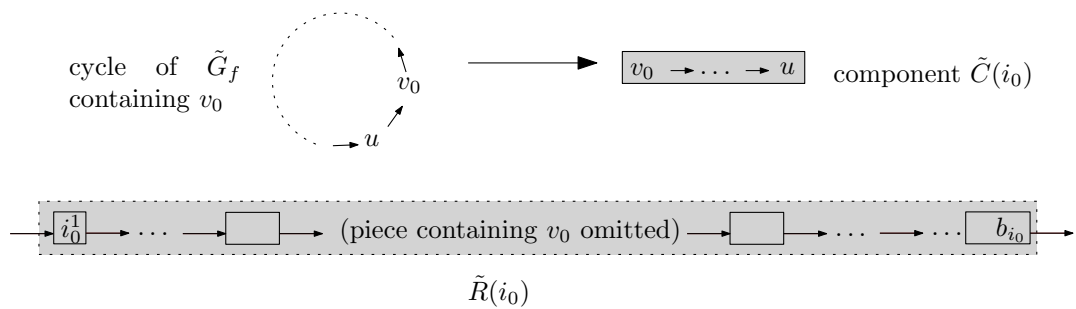}} 
  \caption{The component $\tilde C(i_0)$ is obtained by cutting the cycle of
  $\tilde G_f$ containing $v_0$ at the edge entering $v_0$. The small
    right concatenation $\tilde R(i_0)$ is constructed in a similar way as
    $R(i_0)$, but omitting the piece that would have contained $v_0$.
The subtrees
attached to the distinguished paths of the pieces are not represented.
} 
\label{fig:portion-R-tilde}
\end{figure}


With this construction at hand, wee are now ready to prove Lemma \ref{lemma:casA2}.
We construct a tree $\Psi_1(f)$ from $\tilde G_f$, as
 depicted in Figure~\ref{fig:Psi-A2}.

\begin{itemize}
\item For $i\in \llbracket 1, r\rrbracket$ construct the left concatenation
$L(i)$. Concatenate all these
pieces, by decreasing value of $i$, to
obtain a path from the vertex $a_r\in V_r$ to the vertex $1^1$.
\item For $i\in \llbracket \ell , 0\rrbracket\setminus \{i_0\}$ construct the right concatenation
$R(i)$. Construct also the components $\tilde{R}(i_0)$ and 
 $\tilde{C}(i_0)$. Concatenate $R(\ell), \dots, \tilde{R}(i_0), \dots, R(0)$, by increasing value of $i$, to
obtain a path from $\ell ^1$ to 
a vertex $b_0\in V_0$.
\item Consider the distinguished path of $\tilde{C}(i_0)$, 
which goes from $v_0$ to
$u$. Add an edge from $1^1$ to $v_0$, and
an edge from $u$ to $0^1$. 
This connects the  previously constructed components. 
\end{itemize}
Let $(T,a_r):=\Psi_1(f)$ the marked tree thus obtained. It is rooted
at $b_0$.
 It is clearly an $\cS$-tree  (the edge that goes from 
$u$  to  $0^1$
 is an $\cS$-edge since there was in $G_f$ an $\cS$-edge going from
$u$ to $v_0 \in V_0$).

\begin{figure}[h]
  \centering
  \scalebox{1}{\includegraphics{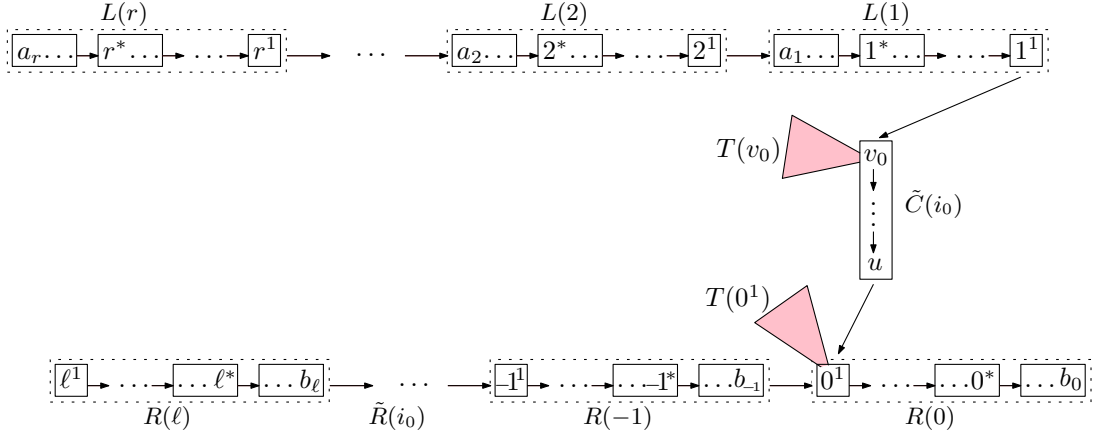}} 
  \caption{The bijection $\Psi_1$ of Lemma~\ref{lemma:casA2}.
  The  subtrees $T(0^1)$ and $T(v_0)$ will be exchanged by the construction $\Psi_2$.
}
  \label{fig:Psi-A2}
\end{figure}

\medskip
\noindent\boite {\bf The marked tree $(T,a_r)$ satisfies Properties $(\rm T _1)$ and $(\rm T _2')$.}
 The proof can be copied \emm verbatim, from the proof
of Lemma~\ref{lemma:casA1}. As before, the vertex $w_0$ of  the marked tree
$\Psi_1(f)=(T, a_r)$ is $v_0=f(-1^1)$.

\medskip

\noindent\boite {\bf The meet $\ell ^1\wedge a_r$ is equal to $0^1$
and distinct from $w_0$.}
This is clear by construction.

\medskip
\noindent\boite {\bf The map $\Psi_1$ is injective.}
Given $(T,a_r)$, one  recovers the graph $\tilde C(i_0)$ by cutting the edge
entering $w_0$ and the one entering $0^1$ on the path from $a_r$
to the root. Closing this piece restores a cycle of the function $f$.
Then,  as in the proof of Lemma~\ref{lemma:casA1}, we  recover
the remaining
pieces by locating the 
lower records of
the path from $r^q$ to $1^1$, and of the path from the root to
$\ell^1$.
Given the pieces, one
recovers the graph $\tilde G_f$ by closing each piece  containing 
no vertex of the form $i^1$. Finally, one adds an edge from $i^1$ to $(i+1)^1$
for $i\in\llbracket \ell, -2\rrbracket$,  an edge from $i^1$ to $(i-1)^1$
for $i\in\llbracket r, 1\rrbracket$, and  an edge from $-1^1$ to 
$w_0$ to recover the graph $G_f$. This shows that $\Psi_1$ is injective. 

\medskip
\noindent\boite {\bf The map $\Psi_1$ is surjective.} Let $(T,r^q)$ be a  marked
$\cS$-tree
rooted at $\rho\in V_0$, 
satisfying~$(\rm T _1)$ and $(\rm T _2')$.
As above, let $w_0$ be the  vertex that follows $1^1$ on the path from
$r^q$ to the root. Assume  that the meet $\ell ^1\wedge r^q$ is
equal to $0^1$ but distinct from $w_0$.

On the
path from $r^q$ to $\rho$, one first meets $w_0$ and, strictly later,
the meet of $\ell^1$ and $r^q$, namely $0^ 1$. We  split the edge
entering $w_0$ and the one entering $0^1$, thus 
creating three connected components. 
One of them contains the path from $r^q$ to $1^1$, another one contains the
path from $\ell ^1$ to $\rho$, and the third one contains the vertex $w_0$.
We transform the latter into a cycle by merging the two half-edges
inherited from the splitting.
We call $C$ the component thus obtained.

We now consider the path going from $r^q$  to $1^1$ in the first 
component. We treat this part 
as in the proof of
Lemma~\ref{lemma:casA1}: we obtain a graph 
which consists of a collection of cycles and a component
containing the path $r^1\rightarrow (r-1)^1\rightarrow\dots\rightarrow
1^1$. 
On each cycle, the smallest vertex lies at a positive abscissa.

We now consider the path going from $\ell ^1$ to the root $\rho$. 
We treat this part 
as in the proof of
Lemma~\ref{lemma:casA1}: we obtain a graph 
which consists of a collection of cycles and a component
containing the path $\ell ^1\rightarrow (\ell
+1)^1\rightarrow\dots\rightarrow -1^1$.
On each cycle, the smallest vertex lies at a non-positive abscissa.

Finally, add an edge from $1^1$ to  $0^1$, from $-1^1$ to $w_0$, and let $H$ be
the graph thus obtained.
 By construction, $H$ is the graph $G_h$ of a function $h: V\setminus \{0^1\}
\rightarrow V$ satisfying (F). 
We check as in the proof of Lemma~\ref{lemma:casA1} that $h$ 
is an $\cS$-function.

By construction $v_0:=h(-1^1)=w_0$ belongs to a cycle of the function $h$, namely the
unique cycle of the component $C$, the source of which has abscissa at
most $0$. 
 Thus  $h$ belongs to the set of
$\cS$-functions considered in  Lemma~\ref{lemma:casA2}.

Finally, it is clear by construction that $\Psi_1(h)=(T,r^q)$, so $\Psi_1$ is
surjective.\\

\noindent\boite{\bf Re-arranging subtrees: the bijection $\Psi_2$.}
We say  that a vertex $v\in V$ is \emm frustrated, if its in-type is not the
same in $f$ and in $(T,r^q)=\Psi_1(f)$.
 Note that 
Observation~\ref{obs:type-casA1} still holds (since Observation~\ref{obs:typeRi-tilde} is the analogue of 
Observation~\ref{obs:typeRi} for the components $\tilde R(i_0)$ and $\tilde
C(i_0)$).

We first correct the in-types of $0^1$ and $v_0$ by exchanging the
subtrees $T(0^1)$ and $T(v_0)$ that are attached to them 
in the pieces $R(0)$
and $\tilde C(i_0)$, respectively ($0^1$ is an ancestor of $v_0$, and
 these subtrees are disjoint).
Let $\tilde T$ be the tree thus
obtained. Then the in-type of $0^1$ in $\tilde{T}$ equals the in-type
of $v_0$ in the function $f$:
 indeed  the edges contributing to these in-types  are, in both cases, the edges coming from $T(v_0)$, plus an edge coming from $u$, plus an edge
coming from 
$V_{-1}$ (this edge joins $b_{-1}$ to $0^1$
in $\tilde T$, 
and  $-1^1$ to $v_0$ in $f$). 
Similarly, the
in-type of $v_0$ in $\tilde T$ equals the in-type of $0^1$ in $f$,
since the edges contributing to these in-types  are in both cases all edges coming from $T(0^1)$ and the edge coming from $1^1$. 
Finally,  the operation $(T,r^q)\mapsto(\tilde T,r^q)$ is an involution since
the exchange of subtrees does not change the marked vertex.

We  treat the path going from $r^q$ to $1^1$ in $\tilde T$ 
as in the proof of Lemma~\ref{lemma:casA1}.
That is, we exchange pairwise
 the trees attached to successive frustrated
vertices along this path.
Let $\Psi_2(T)$ be the marked tree obtained after performing these exchanges for
all $i\in\llbracket 1, r\rrbracket$, and  let $\Psi(f)=\Psi_2\circ\Psi_1(f)$. 
Since we have corrected all types, $\Psi$ satisfies Property~(b).
Moreover $\Psi_2$ is again an involution.  In particular $\Psi$ is
bijection, and Lemma~\ref{lemma:casA2} is proved.
\qed

 \subsubsection{Proof of Lemma~{\bf\rm{\bf\ref{lemma:casA3}}}}
We assume that $v_0=f(-1^1)$ belongs to the
connected component of $\tilde G_f$ of source $0^1$.

\noindent $\bullet$ {\bf A variation of $\tilde G_f$: the graph $\hat G_f$.}
Recall from Section~\ref{sec:A1} the construction of  the graph
$\tilde G_f$, obtained
 by cutting into two half-edges all edges that
leave a vertex of the form $i^1$. If $v_0\not = 
0^1$, we create a new graph $\hat G_f$ having one more cycle than
$\tilde G_f$, as follows. Let 
$u$ be the vertex preceding $0^1$ on
the path that goes in $\tilde G_f$ from $v_0$ to $0^1$. Replace the edge
$u\rightarrow 0^1$ by an edge $u\rightarrow v_0$, thus creating a
new $\cS$-edge, a new cycle, and a new graph $\hat G_f$. If $i_0$ is
the smallest abscissa occurring on this cycle, then $i_0 \le 0$.
If $v_0=0^1$, we let $\hat G_f=\tilde G_f$. 

\begin{figure}[h]
  \centering
   \scalebox{0.8}{\input{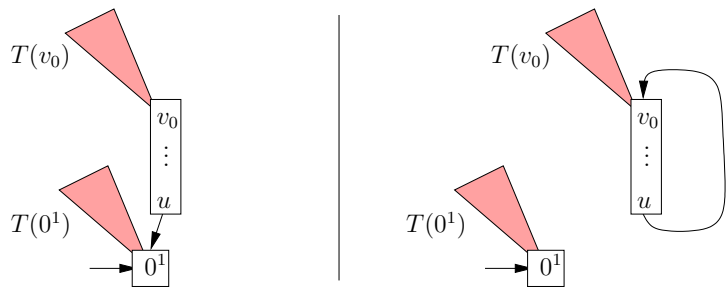}}
  \caption{The component containing $0^1$ in $\tilde G_f$ gives rise, in
    $\hat G_f$, to two components. One of them is a cycle containing $v_0$.
} 
  \label{fig:portion-R-hat}
\end{figure}

Now apply to  $\hat G_f$ all the transformations applied to  $\tilde G_f$ in
Section~\ref{sec:A1}, that led to the definition of $\Psi_1$: open the
cycles before or after their source (depending on the abscissa of the source), connect the
resulting pieces by decreasing or increasing minima (depending again
on the abscissas of the sources), and finally add an edge from $1^1$
to $v_0$. We call the piece containing $v_0$ the \emm special piece,. 
Let $(T,a_r):=\Psi_1(f)$ be the marked tree thus obtained. It is
rooted at $b_0$. It is clearly a marked $\cS$-tree.

\begin{figure}[h]
  \centering
  \scalebox{1}{\includegraphics{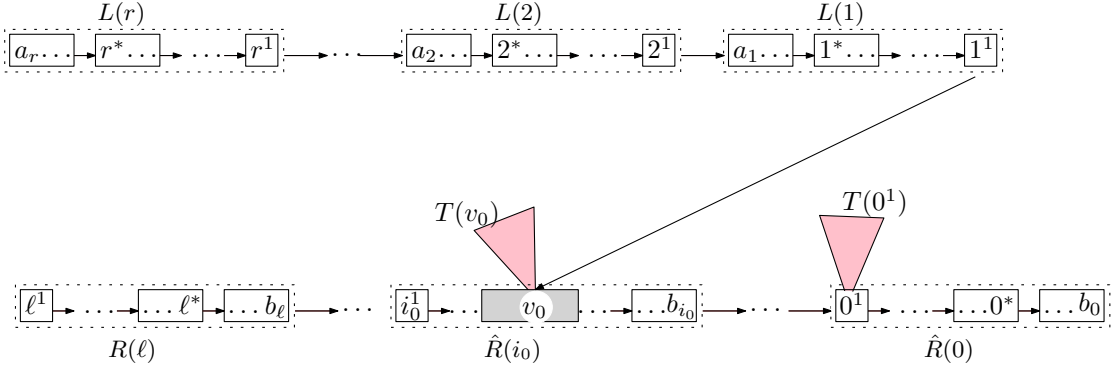}} 
  \caption{The bijection $\Psi_1$ of Lemma~\ref{lemma:casA3}. The special
piece containing $v_0$ is shaded. 
The smallest abscissa of the cycle containing $v_0$ is denoted $i_0$.
The subtrees $T(0^1)$ and $T(v_0)$ 
 will be exchanged  by the construction  $\Psi_2$. 
}
  \label{fig:Psi-A3}
\end{figure}

On checks as in Section~\ref{sec:A1} that the tree $(T,r^q)$ satisfies
$(\rm T_1)$ and $(\rm T_2')$. Moreover, $w_0=v_0=\ell^1 \wedge r^q $,
and thus the condition of Lemma~\ref{lemma:casA3} holds.

\medskip
\noindent\boite {\bf The map $\Psi_1$ is injective.} One recovers the
graph $\hat G_f$ in the same way one recovers $\tilde
G_f$ in Section~\ref{sec:A1}. If $w_0\not
=0^1$,  the edge that enters $v_0$ in the
cycle containing $v_0$ is then cut into two half-edges, and the
outgoing half-edge is re-directed to $0^1$. This gives the graph $\tilde
G_f$, from which one reconstructs $f$ easily.

\medskip
\noindent\boite {\bf The map $\Psi_1$ is surjective.} Let $(T,r^q)$ be a marked
$\cS$-tree
rooted at $\rho\in V_0$, 
satisfying  $(\rm T _1)$ and $(\rm T _2')$. 
Let $w_0$ be the vertex that follows $1^1$ on the path from $r^q$ to
the root.
 Assume  that $w_0$ is the meet $\ell ^1\wedge r^q$. In particular,
it lies on the path from $\ell^1$ to the root.

We first construct from $(T,r^p)$ a functional graph $\hat H$ in the same
way we constructed $H$ in the proof of Lemma~\ref{lemma:casA1}. If
$w_0$ is distinct from $0^1$, it belongs to a cycle of $\hat H$: we
cut the edge of this cycle entering $w_0$ into two half-edges,
and redirect the outgoing half-edge onto $0^1$. 

 The graph $H$ thus obtained is the graph $G_h$ of a function $h: V\setminus \{0^1\}
\rightarrow V$ satisfying (F), which is checked as before to be an
$\cS$-function.

By construction $h(-1^1)=w_0$ belongs to the connected component containing $0^1$
in $\tilde G_{h}$,
so that $h$ belongs to the set of
$\cS$-functions considered in  Lemma~\ref{lemma:casA3}.

Finally, it is clear by construction that $\Psi_1(h)=(T,r^q)$, so $\Psi_1$ is
surjective.

\medskip

{\noindent$\bullet$ \bf Re-arranging subtrees: the bijection $\Psi_2$.}
The only vertices that are likely to be frustrated are the lower
records of the path going from $a_r$ to $1^1$, and the vertices $0^1$
and $v_0$.
These two vertices belong to the
path going from $\ell^1$ to the root.  We first correct their in-types
by swapping the subtrees $T(0^1)$ and $T(v_0)$ that are attached to
them 
in their respective pieces.
Let $\tilde T$ be the tree thus obtained. 
Then the in-type of $0^1$ in $\tilde{T}$ equals the in-type
of $v_0$ in the function $f$: indeed the edges contributing to these
in-types  are, in both cases, the
edges coming from $T(v_0)$, plus an edge coming from $V_{-1}$
(this edge joins $b_{-1}$ to $0^1$ in
$\tilde T$, and  $-1^1$ to $v_0$ in $f$). 
Similarly, the in-type of $v_0$ in $\tilde T$ equals the in-type of
 $0^1$ in $f$, since edges contributing to these in-types are, in
 both cases,  the edges coming from $T(0^1)$, plus an edge coming
 from $1^1$, plus an edge coming from a vertex having the same
 asbcissa as $u$ (this edge joins $u$ to $0^1$ in $f$; in $\tilde T$,
 it joins  $u$ to $v_0$, unless $u$ is the minimum on the
 cycle containing $v_0$. In this case, the abscissa of $u$ is $i_0$,
 $v_0$ is a source in $\tilde T$, and is the endpoint of another sink
 of abscissa $i_0$).
Finally,  the map $(T,a_r)\mapsto(\tilde T,a_r)$ is again an involution.

The frustrated vertices lying on the path going from $a_r$ to $1^1$ in
$\tilde T$ are treated as before, by  exchanging the trees attached to successive frustrated
vertices. Let $\Psi_2(T)$ be the marked tree obtained after performing these exchanges for
all $i\in\llbracket 1, r\rrbracket$.
As before $\Psi_2$ is an involution, so that  $\Psi:=\Psi_2\circ
\Psi_1$ is a bijection. It satisfies Property~(b),
and Lemma~\ref{lemma:casA3} is proved.

\subsection{Proof of Proposition~\ref{prop:casB}}

We now assume that $v_0=f(-1^1)$ belongs to $W_{i_0}$, with $i_0\ge
1$. We first perform some surgery on the piece $L(i_0)$ containing $v_0$.

\smallskip
\noindent{$\bullet$ \bf Surgery on $L(i_0)$: the graphs $\tilde L(i_0),
A(i_0)$ and $ B(i_0)$}.

Recall that $L(i_0)$ consists of a 
path going from the vertex $a_{i_0}$ to the vertex $i_0^1$, to which trees are
attached. One of these trees contains the vertex $v_0$, and we let $v$ be the
attachment point of this tree on the path. Note that we have $a(v)\geq i_0
\geq 1$, and in particular $v\neq v_0$ since $v_0 \in V_0$.

We now define two other vertices $x_0$, $y_0$ as follows (Figure~\ref{fig:portion-L-tilde}). Consider the path of $L(i_0)$ going from $v_0$ to $v$, and let us distinguish two
cases. If there
exists a vertex of negative abscissa along this path, we let
$x_{-1}$ (resp. $y_{-1}$) be the first (resp.  last) vertex of $V_{-1}$
encountered on the path from $v_0$ to $v$, and we let $x_0$ (resp. $y_0$) be the
vertex preceding $x_{-1}$ (resp. following $y_{-1}$) on the path.
Note that since $\max \cS= 1$
 and $\min \cS=-1$, the vertices $x_{-1}$ and $y_{-1}$ are
 well defined;
moreover, $x_0$ and $y_0$ belong to $V_0$.
Now, cut the edge between $x_0$ and $x_{-1}$, and the edge between $y_{-1}$ and
$y_0$. Among the three connected
components thus created, we call $\tilde L(i_0)$ the one containing $v$, 
we call $A(i_0)$ the one containing $x_{-1}$, and we call $B(i_0)$ the one
containing $v_0$. 
If all the vertices on the path
from $v_0$ to $v$ have a nonnegative abscissa, we let $\tilde
L(i_0):=L(i_0)$,
 $A(i_0):=B(i_0):=\varnothing$, and $y_0:=x_0:=v_0$. 
In this case, the vertices $x_{-1}$ and $y_{-1}$ are not defined.
\begin{figure}[h]
  \centering
  \scalebox{1}{\includegraphics{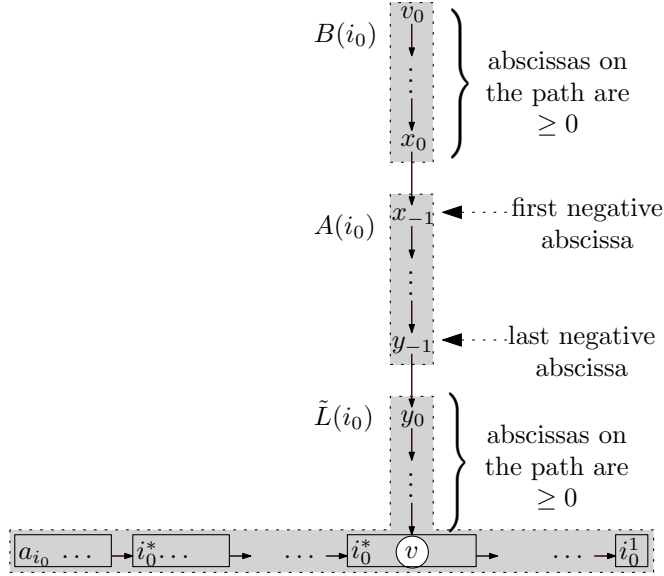}}
  \caption{The graphs $\tilde L (i_0)$, $A(i_0)$, and $B(i_0)$, in the case
where there is a vertex of negative abscissa between $v_0$ and $v$. 
Otherwise, $\tilde L(i_0)=L(i_0)$, and $A(i_0)=B(i_0)=\varnothing$. The
subtrees attached to the paths are not represented.}
  \label{fig:portion-L-tilde}
\end{figure}

 By construction, we have the following analogue of
Observation~\ref{obs:typeLi}.
\begin{Observation}\label{obs:typeLi-tilde}
Any vertex of the graph $\tilde L(i_0) \cup A(i_0)\cup
 B(i_0)$, distinct from $v_0$, $y_0$ and $x_{-1}$ and from the
 lower records  of the path going from $a_{i_0}$ to $i_0^1$,  has the same
 in-type in this graph and in the function $f$. 
\end{Observation}


We now proceed with the description of the bijection, depicted in
Figure~\ref{fig:Psi-B}. We perform the following operations.
\begin{itemize}
\item For $i\in \llbracket 0 ,r\rrbracket$, 
construct the left concatenation $L(i)$. Construct the graph $\tilde L(i_0)$
from $L(i_0)$, and concatenate  the pieces
$$
L(r), L(r-1), \dots, L(i_0-1), \tilde L(i_0), L(i_0+1), \dots, L(1), L(0)
$$
to obtain a path from the vertex $a_r\in V_r$ to the vertex $0^1$. Note
that $y_0$ belongs to a subtree attached to  this path.
\item For $i\in \llbracket \ell , -1\rrbracket$ construct the right
  concatenation
$R(i)$. Concatenate all these pieces 
by increasing value of $i$, to
obtain a path from $\ell ^1$ to the vertex $b_{-1}\in V_{-1}$.
\item Add an edge from $b_{-1}$ to $y_0$. This connects the two previously
constructed components. 
\item If $A(i_0)\neq \varnothing$ (equivalently, if $B(i_0)\neq \varnothing)$, add
an edge from $0^1$ to  $x_{-1}$, and
 an edge from $y_{-1}$ to  $v_0$. This connects the components
$A(i_0)$ and $B(i_0)$ 
to the previously constructed tree.
\end{itemize}
 Let $\Psi_1(f):=(T,a_r)$ be the marked tree thus obtained.
It is rooted at 
 $x_0$ if $A(i_0)\neq \varnothing$
and  at $0^1$ otherwise. 
It is easily checked to be an $\cS$-tree.
\begin{figure}[h]
  \centering
  \scalebox{1}{\includegraphics{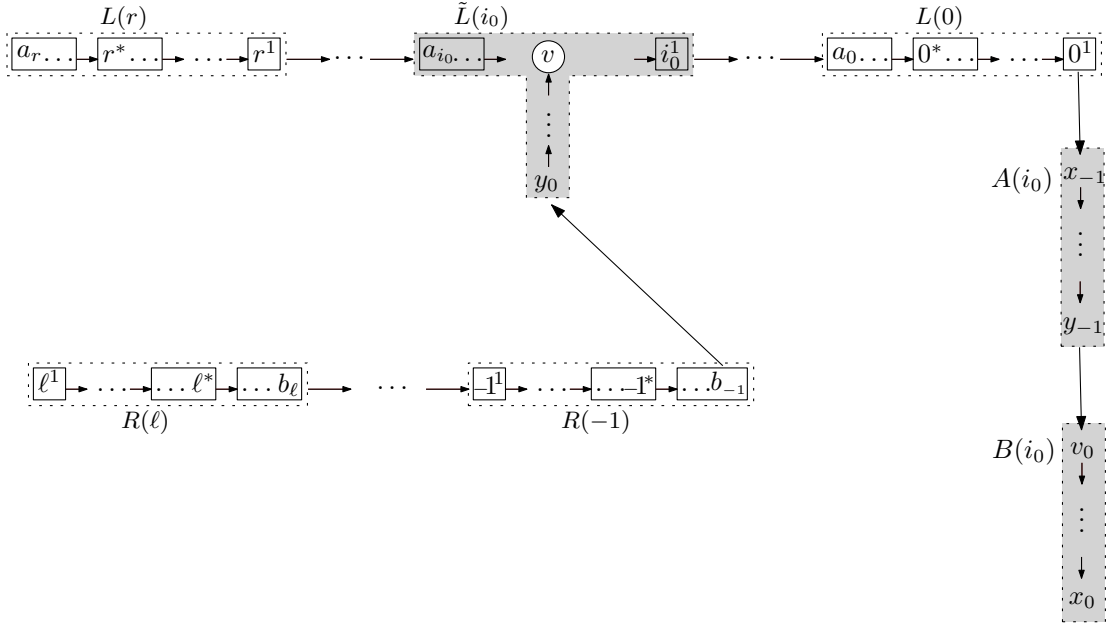}} 
  \caption{The bijection $\Psi_1$ of Proposition~\ref{prop:casB}. } 
  \label{fig:Psi-B}
\end{figure}

As in the previous proofs, we are first going to show that the mapping $\Psi_1$
is a bijection between the sets described  in Proposition~\ref{prop:casB}. Then
we will describe another bijection, $\Psi_2$, such that $\Psi_2\circ\Psi_1$
satisfies  Property~(b),  which $\Psi_1$ lacks.

\medskip
\noindent\boite {\bf The marked tree $(T,a_r)$ satisfies Properties $(\rm T _1)$ and $(\rm T _2'')$.}
 The fact that $T$ satisfies $(\rm T _1)$ is proved as in the proof of
Lemma~\ref{lemma:casA1}. 
The meet $\ell^1\wedge a_r$ is equal to
$v$, and since $i_0\geq 1$, it appears weakly before the
vertex $1^1$ on the path from $a_r$ to the root. Hence the first part
of $(\rm T _2'')$ holds. 
As underlined above, the abscissa of $v$,
being at least $i_0$, is positive. 
The arguments proving the second part of $(\rm T _2'')$ are the same
as those proving  $(\rm T _2')$  in Lemma~\ref{lemma:casA1} (the
concatenation of $R(\ell), \ldots, R(-1)$ is the same in both proofs).
Finally,  all vertices on the path
from $a_r$ to $0^1$ have a nonnegative abscissa. Hence either $0^1$ is the root
of the tree, and there are no vertices of negative abscissa on the path joining $a_r$ to the root, or
$0^1$ is not the root, in which case there are such vertices, and $0^1$
precedes the first of them  (which is $x_{-1}$). This
establishes the last part of~$(\rm T _2'')$.

\medskip

\noindent\boite {\bf The map $\Psi_1$ is injective.} 
Let us start from the marked tree $\Psi_1(f)=(T,a_r)$  and 
reconstruct $f$. 
First, on the path going from $\ell^1$ to the root,
we cut the edge that leaves the last vertex of abscissa $-1$, denoted $b_{-1}$.
The other endpoint of this edge is $y_0$.
By cutting the path that goes from $a_r$ to $0^1$ after each vertex of
the form $i^1$, we recover  the graphs $L(i)$, for $i\in
\llbracket0,r\rrbracket\setminus\{i_0\}$, and the 
graph $\tilde L(i_0)$. Similarly, by cutting, in the path that goes
from $\ell^1$ to $b_{-1}$, each edge that enters a vertex of the form
$i^1$, we recover the graphs $R(i)$, for ${i\in \llbracket \ell ,
-1\rrbracket}$. 
If $0^1$ is not the root of the tree, we recover moreover the vertices
$x_{-1}$ (it follows
$0^1$ on its path to the root) and $y_{-1}$ (the last vertex of negative abscissa on
this path), so that we can reconstruct the graphs $A(i_0)$ and $B(i_0)$. In all
cases, we can then reconstruct the graph $L(i_0)$.
We then proceed as in the proof of Lemma~\ref{lemma:casA1}:
in
each $L(i)$ (resp $R(i)$), locate each left-to-right (resp.
right-to-left) lower record, and split the edges
entering (resp. leaving) each of them. Then close the pieces
 whose root is not of the form
 $i^1$ to recover a cycle of the original function $f$. 
One thus recovers the graph $\tilde G_f$.
Finally, add an edge from $i^1$ to $(i+1)^1$
for $i\in\llbracket \ell , -2\rrbracket$, an edge from $i^1$ to $(i-1)^1$
for $i\in\llbracket r, 1\rrbracket$, and  an edge from $-1^1$ to $v_0$ 
(the vertex that follows $y_{-1}$ on its path to the root)
to recover the graph $G_f$. 

\medskip

\noindent\boite {\bf The map $\Psi_1$ is surjective.} Let $(T,r^q)$ be a marked
$\cS$-tree
rooted at $\rho\in V_0$, 
satisfying $(\rm T _1)$ and $(\rm T _2'')$.
Denote $v:=\ell ^1\wedge r^q$. By 
$(\rm T _2'')$, we have $a(v)\geq 1$.

 Consider the path going from $\ell ^1$ to $v$ in $T$, and 
let $b_{-1}$ be the last vertex of negative abscissa on this path, and
$y_0$ the next vertex on this path. 
Note that $b_{-1}$ (resp. $y_0$) has abcissa $-1$ (resp. $0$).
Split the edge joining  $b_{-1}$ to $y_0$. 
This gives two connected components, one
containing the path from $r^q$ to the root (including the vertex $v$),
the other  the path from $\ell ^1$ to $b_{-1}$.

Now visit  the  path going $b_{-1}$ to $\ell ^1$ (in this ``wrong''
direction).
Lower records on this
path 
are called \emm sinks,, and vertices preceding the
sinks (in the same ``wrong'' direction) are called
\emm sources, 
(we consider $\ell^{1}$ as a source).
We now split  all edges between sinks and sources, and thus obtain
 a number of \emm pieces,. 
By  $(\rm T _2'')$, each vertex $i^1$ for $i\in\llbracket \ell ,-1 \rrbracket$ is
the source and the sink of a piece.
Take  all the pieces containing a  vertex $i^1$, for $i\leq -1$, and
concatenate them  by adding an edge from $i^1$ to
$(i+1)^1$ for $i\in\llbracket \ell , -2\rrbracket$.
In the remaining pieces, merge the two extremal half-edges to form a cycle.

We now describe a step of the reverse bijection that applies only if $0^1$ is not the
root of $T$. In that case,  $(\rm T _2'')$ implies that  the vertex $0^1$ 
belongs to the path from $r^q$ to the root,
and is followed by a vertex of $V_{-1}$, say $x_{-1}$.  Let $y_{-1}$
be the last vertex of 
$V_{-1}$ on the path from $0^1$ to the root, and  let $v_0\in V_0$ be the
vertex following it. We now split the edge between $0^1$ and $x_{-1}$, and the
edge between $y_{-1}$ and $v_0$. We call $\tilde A$ and $\tilde B$ the connected components
containing $x_{-1}$ and $v_0$ after the splitting, respectively.
We re-connect the components $\tilde A$ and $\tilde B$ to the connected component of $r^q$ by
adding an edge from $x_0$ to $x_{-1}$ and from $y_{-1}$ to $y_0$. This
concludes the step 
that is specific to the case where $0^1$ is not the root of $T$.
Otherwise, we denote 
$\tilde A:=\tilde B:=\varnothing$, and
 $y_0=x_0:=v_0$. 

We now consider the path going from $r^q$  to $0^1$. 
Lower records on this
path are called \emm sources,, and vertices preceding the sources are called
\emm sinks, 
(we consider $0^1$ as a sink).
We now split  each edge that enters a source of the path,
and thus obtain a number of \emm pieces,. 
By  $(\rm T _1)$ and $(\rm T _2'')$, each vertex $i^1$ for $i\in\llbracket 0,r \rrbracket$ is
the source and the sink of a piece.
Take  all the pieces containing a  vertex of the form $i^1$, for
$i\geq0$, and concatenate them by
adding an edge from $i^1$ to
$(i-1)^1$ for $i\in\llbracket 1, r\rrbracket$.
Transform each of the remaining pieces into a cycle by merging the two
extremal half-edges. 
Finally, add an edge from $-1^1$ to  $v_0$.

 Let $H$ be the graph
thus obtained.   By construction, $H$ is the graph $G_h$ of a function $h:
V\setminus \{0^1\} \rightarrow V$ satisfying (F). 
One easily checks that $h$ is an $\cS$-function, using the same
arguments as in the proof of Lemma~\ref{lemma:casA1},
 plus the facts
that $h(y_{-1})=y_0 \in V_0$ and $h(x_0)=x_{-1} \in V_{-1}$.

Let us now prove that $h$ satisfies the condition of
Proposition~\ref{prop:casB}.
By construction, the source of the component of $\tilde G_h$
containing $v_0=h(-1^1)$ is the last lower record encountered 
(weakly) 
before $v$ on the path
from $r^q$ to $v$ in $T$. By  $(\rm T _2'')$ this source appears (weakly) before
$1^1$, and since $1^1$ is a lower record by $(\rm T _1)$, the abscissa of this
source is at least $1$. In other words, we have 
$v_0\in\cup_{i=1}^r W_i$,
 the sets $W_i$ being understood with
respect with the function $h$.

Finally, note that all abscissas are nonnegative on the paths from $y_0$ to $v$
and from $v_0$ to $x_0$. 
This implies that $\tilde A$ and
$\tilde B$ coincide with the pieces $A(i_0)$ and $B(i_0)$ which one would build
from the function $h$. From that it is clear  that $\Psi_1(h)=(T,r^q)$, so that $\Psi_1$ is surjective.

\medskip

{\noindent$\bullet$ \bf Re-arranging subtrees: the bijection $\Psi_2$.}
We say  that a vertex $v\in V$ is \emm frustrated, if its in-type is not the
same in $f$ and in $(T,a_r)=\Psi_1(f)$.
We claim that the vertices belonging to $\cup_{i=\ell }^{-1} W_i$ are
not frustrated. This is a direct
consequence of Observation~\ref{obs:typeRi} and of the fact that to concatenate
$R(i-1)$ to $R(i)$, for $i\in\llbracket \ell +1, -1 \rrbracket$, we add a new incoming
edge to the vertex $i^1$ coming from $V_{i-1}$, which 
 compensates the deletion of the edge $(i-1)^1 \rightarrow
i^1$ in the construction of $\tilde{G}_f$ from $f$.
Similarly,  Observations~\ref{obs:typeLi} and~\ref{obs:typeLi-tilde}
give:
\begin{Observation}\label{obs:type-casB}
Any vertex of 
$\cup_{i=0 }^{r} W_i$ 
 distinct from $v_0,y_0,x_{-1}$, and from the lower records of
the path from $r^q$ to $0^1$ is not frustrated.
\end{Observation}

In fact,  $y_0$ is not frustrated: 
 indeed, the edge coming from
  $b_{-1}$ in $T$ compensates either the loss of the edge coming from $y_{-1}$ in
$L(i_0)$ (if $y_0\neq v_0$) or the loss of the edge coming from $-1^1$
in $f$ (if 
$y_0=v_0)$, and 
since $y_{0}$ does not belong to the distinguished path of
$L(i_0)$, Observation~\ref{obs:typeLi} enables us to conclude.
Similarly, $x_{-1}$ (if it exists) 
is not frustrated, since the edge coming from $0^1$ in $T$ compensates
the edge coming from $x_0$ in $L(i_0)$. 
Finally,
$v_0$ is not frustrated either: either it is equal to
$y_0$ (if $0^1$ is the root of $T$), or the edge coming from $y_{-1}$ in $T$
compensates the loss of the edge coming from $-1^1$ in $f$.
Therefore, we can strengthen  our previous observation as follows.
\begin{Observation}\label{obs:type-casB-bis}
Any vertex of  $\cup_{i=0 }^{r} W_i$ 
 distinct  from the lower records of
the path from $r^q$ to $0^1$ is not frustrated.
\end{Observation}
We correct the in-types of these sources by a second map $\Psi_2$.
We proceed as in the proof of Theorem~\ref{thm:basic} in
Section~\ref{sec:basic}, 
by exchanging the subtrees attached to adjacent frustrated vertices.
 Let $\Psi_2(T,r^q)$ be the tree thus obtained,
and $\Psi(f)=\Psi_2\circ\Psi_1(f)$. 
One again, $\Psi_2$ is an involution.  In particular $\Psi$ is
bijection, which satisfies Property~(b) thanks to the construction
$\Psi_2$.  Proposition~\ref{prop:casB} is proved.

\section{Enumeration of general embedded trees}
\label{sec:enum-neg}
In this section, we prove the  enumerative results of
Section~\ref{sec:main} in the case $\ell<0$.
These results follow from the bijection of 
Theorem~\ref{thm:rooted}, combined with the enumeration of
$\cS$-functions (which remains an elementary exercise). We also need to
relate the $\cS$-trees 
occurring in Theorem~\ref{thm:rooted} to the $\cS$-embedded
Cayley trees  of Section~\ref{sec:main}. This is done in the following lemma.
We adopt  the same notation as in the previous section:  $V= \cup_{i=\ell}^r V_i$ with
$V_i=\{i^1, \ldots, i^{n_i}\}$, and $\cS\subset \zs$ satisfies $\min
\cS=-1$ and $\max \cS=1$.  The \emm type distribution, of a tree is the
collection  of numbers $n(i,s,\bmc)$ (with $i\in
 \llbracket 0,r\rrbracket$, $s\in \cS$ and 
$\bmc \in \ns^{3}$)
 giving the number of vertices of type $(i;s;\bmc)$.

\begin{Lemma}\label{lem:linkneg}
   The number of 
 $\cS$-embedded Cayley trees having a prescribed type distribution is
$$
\frac 1 {n_\ell n_r} \frac {n!}{\prod_{i=\ell}^r (n_i-1)!}
$$
times the number of marked $\cS$-trees satisfying Conditions $(\rm T_1)$ and $(\rm T_2)$ of
Theorem~{\rm\ref{thm:rooted}} and having the  same  type distribution (as
always, 
$(n_\ell, \ldots, n_r)$ 
denotes the profile of the tree, and $n$
its size).
\end{Lemma}
\begin{proof}
Equivalently, we want to prove that the number  of
 $\cS$-embedded Cayley trees having a prescribed type distribution 
and \emm  two marked vertices,, one at abscissa $\ell$ and the other
at abscissa~$r$, is 
$$
 \frac {n!}{\prod_i (n_i-1)!}
$$
times the number of marked $\cS$-trees satisfying  $(\rm T_1)$ and $(\rm T_2)$ and having the  same  type
distribution. We will construct a 1-to-$n!/\prod_i (n_i-1)!$
correspondence between marked
$\cS$-trees satisfying $(\rm T_1)$ and $(\rm T_2)$ and doubly marked
$\cS$-embedded Cayley trees, preserving  the type distribution.

Let $(T,r^q)$ be a marked $\cS$-tree on $V$ satisfying $(\rm T_1)$ and
$(\rm T_2)$. Let  us mark, in addition, the vertex $\ell^1$.  
For $\ell \le i \le r$,  let us rename
the vertex $i^1$ by $i^k$, for a $k$ chosen in
$\{1, \ldots, n_i\}$; conversely, let us rename $i^k$ by $i^1$.  This gives
an arbitrary  $\cS$-tree $T_1$, rooted at a vertex of $V_0$, with two
marked vertices, one at abscissa $\ell$ and one  at abscissa $r$. This tree may
or may not satisfy  $(\rm T_1)$ and $(\rm T_2)$. 
The number of different trees $T_1$ that can be constructed  from $T$
in such a way is $\prod_{i=\ell}^r n_i$.  
The tree $T$ can be recovered
from $T_1$ by restoring vertex names: indeed, the vertex
$\ell^1$ in $T$ is at the position of the marked vertex of abscissa
$\ell$ in $T_1$, and the position of $i^1$ in $T$, for $i>\ell$, is
prescribed by $(\rm T_1)$ and $(\rm T_2)$.

Let us now assign labels from $\{1, \ldots,
n\}$, with $n=\sum_i n_i$, to the vertices of $T_1$, in such a way
that the labels $a$ and $b$ assigned to $i^k$ and $i^p$ satisfy
$a<b$ if $k<p$. There are $n!/\prod_i n_i!$ ways to do so. Finally,
erase all names $i^k$ from the tree, for all $i$ and $k$. This gives
an arbitrary rooted $\cS$-embedded
Cayley tree $T_2$, with a marked vertex at abscissa $\ell$ and one at
abscissa $r$. The tree
$T_1$ can be recovered from $T_2$ by renaming the vertices of abscissa
$i$ with $i^1, \ldots, i^{n_i}$ in the unique way that is consistent
with the order on labels: if two vertices of labels $a$ and $b$, with
$a<b$, lie at abscissa $i$, then  their names $i^k$ and $i^p$ must
satisfy $k<p$. 

 The marked $\cS$-tree $T$ has given rise to $n!/\prod_i (n_i-1)!$
doubly marked embedded trees $T_2$. Moreover,
$T_2$, $T_1$ and $T$ have the same type distribution. The result
follows.
\end{proof}

In what follows, we count trees by counting functions, using the
correspondence of Theorem~\ref{thm:rooted}. When we prescribe the
types of vertices, or the number of vertices of a certain type, we
assume as in Section~\ref{sec:enum-nonneg} that the natural \emm
compatibility conditions, hold.

Recall that Theorem~\ref{thm:cayley-profile} is already proved, even
when $\ell<0$,  thanks to the fourth remark that follows its statement.

\subsection{The profile of $\cS$-ary trees: proof of  Theorem~\ref{thm:S-ary-profile}}
We argue as in Section~\ref{sec:proof-S-ary-profile}.
The number of $\cS$-ary trees of vertical profile $(n_ \ell,
\ldots, n_r)$ is obtained by divising by $n!$ the number of
injective $\cS$-embedded Cayley trees with the same profile.
By Lemma~\ref{lem:linkneg}, the number of injective
$\cS$-embedded Cayley trees having vertical profile $(n_ \ell,
\ldots, n_r)$ is $n!/n_\ell /n_r /\prod_i (n_i-1)!$ times the number
of marked injective $\cS$-trees satisfying $(\rm T_1)$ and $(\rm
T_2)$.
By Theorem~\ref{thm:rooted} (and in particular Property (b)), the
number of such trees is also the number of $\cS$-functions from
$V\setminus\{0^1\}$ satisfying (F) that are injective on each
$V_i$. This number is given by the following
lemma. Theorem~\ref{thm:S-ary-profile} follows, in the case $\ell<0$.

\begin{Lemma}
  The number of $\cS$-functions from
$V\setminus\{0\}$, injective on each $V_i$ and satisfying $(\rm F)$ is
$$
n_0\, {\sum_{s\in \cS} n_{-s}  \choose n_0 -1}
 \prod_{i=\ell \atop i\not = 0}^r{\sum_{s\in \cS} 
n_{i-s} -1 \choose n_i -1} \prod_{i=\ell}^r(n_i-1)!.
$$
\end{Lemma}
\begin{proof}
   We proceed as in the proof of Lemma~\ref{thm:injections-profile}.
For $i\not\in\{0, -1\}$,
 we choose the (distinct) images of the elements of
$V_i\setminus\{i^1\}$ in the set 
$\cup_s V_{i-s} \setminus\{f(i^1)\}$, 
where
$f(i^1)=(i-1)^1$ if $i>1$ and $f(i^1)=(i+1)^1$ if $i<0$. There are ${\sum_s
n_{i-s} -1 \choose n_i -1} (n_i-1)!$ ways to do so.

For $i=0$, we choose the (distinct) images of the elements of $V_0\setminus\{0^1\}$
in the set $\cup_s V_{-s}$. There are ${\sum_s n_{-s}  \choose n_0 -1} (n_0-1)!$
ways to do so.

For $i=-1$, we first choose the image 
$v_0$
of $-1^1$ in the set $V_0$: there are
$n_0$ ways to do so.  Then, we choose the (distinct) images of elements of
$V_{-1}\setminus \{-1^1\}$ in the set $\cup_s V_{-1-s} \setminus \{v_0\}$.
There are ${\sum_s n_{-1-s} -1  \choose n_{-1} -1 } (n_{-1} -1)!$ ways to do so.
\end{proof}

\subsection{The out-types of $\cS$-embedded Cayley trees: proof of
  Theorem~\ref{thm:cayley-out}}
We argue as in Section~\ref{sec:proof-cayley-out}. By
Lemma~\ref{lem:linkneg} and Theorem~\ref{thm:rooted} (in particular
Property (a)), the number of $\cS$-embedded Cayley trees having
$n(i,s)$ non-root vertices of out-type $(i;s)$ is  $n!/n_\ell /n_r
/\prod_i (n_i-1)!$ times the number 
of  $\cS$-functions from
$V\setminus\{0^1\}$ to $V$ satisfying (F) and having the same distribution of
out-types. This number is given by the following 
lemma. Theorem~\ref{thm:cayley-out} follows, in the case $\ell<0$.
\begin{Lemma}\label{lem:funtctions-out-neg}
 $1$. The number of $\cS$-functions from $V\setminus\{0^1\}$ to $V$ satisfying
 $(\rm F)$ and in which each  $v \in V$  has a prescribed
out-type $(i_v;s_v)$ is, assuming compatibility,
$$
n_\ell  n_r\prod_{i=\ell }^r n_i^{c(i)-1},
$$
where  $c(i)$ is the number of vertices whose image lies in $V_i$:
$$
c(i)=\sharp\{v\in V: i_v-s_v =i\}.
$$

\noindent
$2$. Let $n(i,s)$ be non-negative integers, for $i\in\llbracket \ell,r
\rrbracket$ and $s\in \cS$, satisfying the compatibility conditions of
an out-type distribution. The number of $\cS$-functions from $V\setminus\{0^1\}$ to $V$ satisfying
 $(\rm F)$ and in which, for
  all  $i\in \llbracket \ell,r \rrbracket$ and $s\in \cS$, exactly
  $n(i,s)$  vertices have   out-type $(i;s)$ is    
$$
\frac {n_\ell n_r \prod\limits_{i=\ell} ^r (n_i-1)!\prod\limits_{i=\ell}^rn_i^{c(i)-1} \prod\limits_{i=\ell}^{-1} n(i,-1)\prod\limits_{i=1}^r n(i,1)}
{\prod\limits_{i,s} n(i,s)!},
$$
where $c(i)$ is the number of vertices whose image lies in $V_i$:
$$
c(i)= \sum_{s} n(i+s,s).
$$
\end{Lemma}
\begin{proof}
   We proceed as in the proof of Lemma~\ref{thm:functions-out}.

1. We first choose  the images of the $c(i)$
vertices whose image is in $V_i$, for $i\in \{\ell ,r\}$. There are
$n_i^{c(i)}$ possible choices. 
 For  $i\in \llbracket 0, r-1\rrbracket$ (resp.  $i\in \llbracket \ell +1,
-1\rrbracket$) we only choose in $V_i$  the images of the
vertices different from $(i+1)^1$ (resp. $(i-1)^1$). There are $n_i^{c(i)-1}$ possible choices. 

2. We first choose the out-type of every vertex, and then apply the
previous result. For all $i$ and $s$, we must choose the $n(i,s)$
vertices of $V_i$ that have out-type $(i;s)$, keeping in mind that $i^1$ has
out-type $(i;1)$ for $i\ge 1$, out-type $(i;-1)$ for $i\le -1$, and
out-type $(0;\varepsilon)$ for $i=0$. Thus the number of ways to
assign the out-types is 
$$\label{eq:factor-out}
\frac{(n_0-1)!}{\prod\limits_s n(0,s)!}  \prod_{i=\ell}^{-1} \frac{(n_i-1)!}
{(n(i,-1)-1)!\prod\limits_{s\not = -1} n(i,s)!}\prod_{i=1}^r \frac{(n_i-1)!}
{(n(i,1)-1)!\prod\limits_{s\not = 1} n(i,s)!}.
$$
The lemma follows. 
\end{proof}

\subsection{The out-types of $\cS$-ary trees: Proof of
  Theorem~\ref{thm:S-ary-out}} 
We argue as in Section~\ref{sec:proof-S-ary-out}. By
Lemma~\ref{lem:linkneg} and Theorem~\ref{thm:rooted}, the number of
$\cS$-ary trees having $n(i,s)$ non-root vertices of out-type $(i;s)$ is
$1/n_\ell/n_r/\prod_i (n_i-1)!$   times  the number of $\cS$-functions from
$V\setminus\{0^1\}$ to $V$ satisfying  (F) that are
injective on each $V_i$ and  have the
same distribution of out-types. This number is given by the second
part of the following lemma. Theorem~\ref{thm:S-ary-out} follows, in
the case $\ell<0$. 
\begin{Lemma}
$1$. The number of $\cS$-functions from $V\setminus\{0^1\}$ to $V$,
  injective on each $V_i$,  satisfying
 $(\rm F)$, and in which each  $v \in V$ has a prescribed out-type
 $(i_v;s_v)$ is, assuming compatibility,
$$
\frac 1 {\prod\limits_{i=\ell+1}^{r-1} n_i} \prod_{i,s} n(i,s)! {{n_{i-s}} \choose {n(i,s)}},
$$
where $n(i,s)$ is the number of vertices of out-type $(i,s)$.

\noindent
$2$. Let $n(i,s)$ be non-negative integers, for $i\in\llbracket \ell,r
\rrbracket$ and $s\in \cS$, satisfying the compatibility conditions of
an out-type distribution.
The number of $\cS$-functions from $V\setminus\{0^1\}$ to $V$, injective on
each $V_i$,  satisfying  $(\rm F)$ and in which, for
  all  $i\in \llbracket \ell,r \rrbracket$ and $s\in \cS$, exactly
  $n(i,s)$  vertices have   out-type $(i;s)$ is    
$$
\frac {\prod\limits _{i=\ell}^r (n_i-1)! \prod\limits_{i=\ell}^{-1} n(i,-1)\prod\limits_{i=1}^r n(i,1)} 
{\prod \limits_{i=\ell+1}^{r-1} n_i} \prod_{i,s} {n_{i-s} \choose n(i,s)}.
$$
\end{Lemma}
\begin{proof}
   We proceed as in the proof of Lemma~\ref{thm:injections-out}.

1. For 
\begin{itemize}
\item $i\in \llbracket -1, 0\rrbracket$ and $s\in \cS$, 
\item or $i \in \llbracket \ell, -2\rrbracket$, and $s\in \cS\setminus\{-1\}$,
\item or $i \in \llbracket 1, r\rrbracket$, and $s\in \cS\setminus\{1\}$, 
\end{itemize}
we choose  in $V_{i-s}$ the (distinct) images  of the $n(i,s)$
vertices having out-type $(i,s)$. There are ${n_{i-s} \choose
  n(i,s)}n(i,s)!$ ways to do so.

When $i \in \llbracket 1, r\rrbracket$ and $s=1$,  we
choose  in $V_{i-1}\setminus\{(i-1)^1\}$ the images  of the $n(i,s)-1$
vertices different from $i^1$ 
having out-type $(i,1)$. There are
$$
{n_{i-1}-1 \choose n(i,1)-1}(n(i,1)-1)!
= {n_{i-1} \choose n(i,1)}\frac{n(i,1)!}{n_{i-1}}
$$
 ways to do so.

When $i \in \llbracket \ell,-2\rrbracket$ and $s=-1$,  we
choose  in $V_{i+1}\setminus\{(i+1)^1\}$ the images  of the $n(i,s)-1$
vertices different from $i^1$ 
having out-type $(i,-1)$. There are
$$
{n_{i+1}-1 \choose n(i,-1)-1}(n(i,-1)-1)!
= {n_{i+1} \choose n(i,-1)}\frac{n(i,-1)!}{n_{i+1}}
$$
 ways to do so.

 This concludes the proof of the first result.

\medskip
\noindent
2. The argument used to prove the second part of
Lemma~\ref{lem:funtctions-out-neg} can be copied \it{verbatim}.
\end{proof}

\subsection{The in-types: Proof of Theorem~\ref{thm:cayley-in}}
Assume $\cS=\llbracket -1, 1\rrbracket$. We argue as in Section~\ref{sec:proof-cayley-in}. By
Lemma~\ref{lem:linkneg} and Theorem~\ref{thm:rooted} (in particular
Property (b) of this theorem), the number of $\cS$-embedded Cayley
trees having $n(i,\bmc)$  vertices of in-type $(i;\bmc)$ is
$n!/n_\ell/n_r/\prod_i (n_i-1)!$   times  the number of $\cS$-functions from
$V\setminus\{0^1\}$ to $V$ satisfying (F) and having the
same distribution of in-types. This number is given by the second
part of the following lemma. Theorem~\ref{thm:cayley-in} follows, in
the case $\ell<0$.

 \begin{Lemma}
 Let $\cS=\llbracket -1,1\rrbracket$.\\
 $1$.  The number of $\cS$-functions $f$ from $V\setminus\{0^1\}$ to $V$ satisfying
 Conditions~$(\rm F)$, \emm except maybe the condition $f(-1^1)\in V_0$,, and in which each  $v \in V$ has a prescribed
in-type $(i_v;\bmc_v)$ is, assuming compatibility,
$$
n_{-1}\,\frac{\prod\limits_{i=\ell}^r (n_i-1)!}{\prod\limits_{b\ge 0, s}  b!^{n_s(b)}}
 \prod_{i=\ell+1}^{-1} c_{i^1}^{-1}  \prod_{i=0}^{r-1} c_{i^1}^1,
$$
where $n_s(b)$ is the number of vertices $v$ that
have exactly $b$ pre-images at abscissa $a(v)+s$:
$$
n_s(b)=\sharp \{v \in V: c_v^s=b\}.
$$

\noindent
$2.$ Let $n(i,\bmc)$ be non-negative integers, for $i\in\llbracket \ell,r
\rrbracket$ and $\bmc\in \ns^{3}$, satisfying the compatibility
conditions of an in-type distribution.
The number of $\cS$-functions from $V\setminus\{0^1\}$ to $V$ satisfying
 $(\rm F)$ and in which, for
  all  $i\in \llbracket \ell,r \rrbracket$ and $\bmc\in \ns^{3}$, exactly
  $n(i,\bmc)$  vertices have in-type $(i;\bmc)$ is    
$$
\frac {n_\ell n_r\prod\limits_{i=\ell}^r(n_i-1)!^2}
{\prod\limits_{i, \bmsc}n(i,\bmc)!
\prod\limits_{ b\ge 0,s}b!^{n_s(b)}}
\prod\limits_{i=\ell}^{-1} n(i,-1)\prod\limits_{i=1}^{r} n(i,1),
$$
where  $n_s(b)$ is the number of vertices $v$ that
have exactly $b$ pre-images at abscissa $a(v)+s$, and 
$n(i,1)$ is the number of vertices of out-type $(i;1)$.
Equivalently,
\begin{align*}
n_s(b)=\sum_{i} \sum_{\bmsc: c^s=b} n(i,\bmc),
\quad \quad
n(i,1)=\sum_{\bmsc} c^1 n(i-1, \bmc).
\end{align*}
\end{Lemma}
\begin{proof} We proceed as in the proof of
  Lemma~\ref{thm:functions-in}.

  1. Let us first choose the images of the $n_{-1}$ vertices of
  $V_{-1}$. Exactly $c_{(-1-s)^k}^{s}$ of
  them have image   $(-1-s)^k$, for all $s$ and $k$.  For $i\in \llb
  \ell, r\rrb \setminus\{-1\}$, let us choose the images of the
  $n_i-1$ vertices of 
  $V_i\setminus\{i^1\}$. Exactly $c_{(i-s)^k}^{s}-\chi_{s=1=k=1,i>0}-\chi_{s=-1,k=1,i<-1}$ of
  them have image   $(i-s)^k$, for all $s$ and $k$.  
The first result follows.

\medskip
2. We first focus on functions that satisfy  Conditions (F), \emm except
maybe the condition $f(-1^1)\in V_0$,.  Let us first prescribe the in-type
$(i;\bmc_{i^1})$ of all vertices of the form $i^1$, for all $i$, and the number $\tilde
n(i;\bmc)$ of vertices of $V_i\setminus\{i^1\}$ having in-type
$(i;\bmc)$, for all $\bmc\in \ns^{3}$. Clearly,
$$
\tilde n (i,\bmc)= n (i,\bmc)- \chi_{\bmsc=\bmsc_{i^1}}.
$$
The number of ways to assign types to  vertices of
$V_i\setminus\{i^1\}$ is
$$
\frac{(n_i-1)!}{\prod_{\bmsc}\tilde n(i, \bmc)!}
=
\frac{(n_i-1)!}{\prod_{\bmsc} n(i, \bmc)!}\, n(i, \bmc_{i^1}) .
$$
Using the first result, we conclude that the number of functions such
that $i^1$ has in-type $(i;\bmc_{i^1})$ and $\tilde
n(i;\bmc)$ of vertices of $V_i\setminus\{i^1\}$ having in-type
$(i;\bmc)$ is
$$
n_{-1}\,\frac {\prod\limits_{i=\ell}^r(n_i-1)!^2}
{\prod\limits_{i, \bmsc}n(i;\bmc)!
\prod\limits_{ b\ge 0,s}b!^{n_s(b)}}
\left( \prod\limits_{i=\ell+1}^{-1} c_{i^1}^{-1} \right)
\left( \prod\limits_{i=0}^{r-1} c_{i^1}^1 \right)\left( \prod_{i=\ell}^r n(i, \bmc_{i^1}) \right).
$$

 Let us now only prescribe the values $n(i,\bmc)$ (still focussing on
 functions that may not satisfy $f(-1^1)\in V_0$). That is, we
need to sum the above formula over all possible in-types of the vertices
$i^1$, for $i=\ell,\ldots, r$. Note that only the three rightmost products
depend on the choice of these types. We are thus led to evaluate
\begin{multline*}
  \sum_{i=\ell}^r\sum_{\bmsc_{i^1}\in \ns^{3}}\left( \prod\limits_{i=\ell+1}^{-1} c_{i^1}^{-1} \right) \left( \prod\limits_{i=0}^{r-1}
c_{i^1}^1 \right)\left( \prod_{i=\ell}^r n(i, \bmc_{i^1}) \right)
=\\
\left(\sum_{\bmsc_{\ell^1}} n(\ell,\bmc_{\ell^1})\right)
\prod_{i=\ell+1}^{-1} \left(\sum_{\bmsc_{i^1} } c_{i^1}^{-1}n(i,
\bmc_{i^1})\right)
\prod_{i=0}^{r-1} \left(\sum_{\bmsc_{i^1} } c_{i^1}^1n(i,
\bmc_{i^1})\right)
\left(\sum_{\bmsc_{r^1}} n(r,\bmc_{r^1})\right)
=\\
n_\ell n_r \prod_{i=\ell}^{-2}n(i,-1) \prod_{i=1}^{r} n(i,1).
\end{multline*}
Thus the number of functions that satisfy  Conditions (F),  except
maybe the condition $f(-1^1)\in V_0$, and have $n(i;\bmc)$ vertices on
in-type $(i;\bmc)$ for all $i$ and $\bmc$ is
$$
n_\ell  n_{-1} n_r
\frac{\prod_{i=\ell }^{r} (n_i-1)!^2}
{\prod_{i,\bmc} n(i,\bmc)! \prod_{b\geq 0, s} b!^{n_s(b)}}
\prod_{i=\ell }^{-2}
n(i,-1)
\prod_{i=1}^{r}
n(i,1).
$$

It remains to prove that the proportion of these functions that also
satisfy $f(-1^1)\in V_0$ is $n(-1,-1)/n_{-1}$.
This follows from the existence of an in-type preserving bijection
between functions $f$ that satisfy all
conditions of (F) and have a marked vertex $v$ in $V_{-1}$, and
functions $g$ that satisfy (F), with the possible exception of
$g(-1^1)\in V_0$, and have a marked vertex $w$ in $V_{-1}\cap
g^{-1}(V_0)$ (that is, a vertex of out-type $(-1,-1)$). This bijection sends $(f,v)$ to $(g,w)$, where $w=v$, $g (-1^1) = f(v)$,
$g(v)=f(-1^1)$ and $g(x)=f(x)$ if $x\not\in\{-1^1,v\}$. As there are
$n_{-1}$ choices for the vertex $v$ in $f$, and $n(-1,-1)$ choices for
the vertex $w$ in $g$, this completes the proof of the second part of
the lemma.
\end{proof}

\section{Other approaches}
\label{sec:other}
We now present two other approaches for counting embedded trees:
the first one combines  recursive  descriptions of trees, functional
equations and the Lagrange 
inversion formula; the second one is based on the matrix-tree theorem. The first approach
proves the results on the profile and on the out-types, both for
$\cS$-embedded Cayley trees and for $\cS$-ary trees. The second one
proves the results on the profile and on the out-types of
$\cS$-embedded Cayley trees only. These methods are of course more
routine, but not bijective. They involve computing certain determinants that factor for
reasons that are not clear combinatorially. To our knowledge, they do
not prove the other results of this paper.

\subsection{Functional equations and the Lagrange inversion formula}
\label{sec:lagrange}
One can prove the results that deal with the vertical profile
(Theorems~\ref{thm:cayley-profile} and~\ref{thm:S-ary-profile}) and
with the out-type (Theorems~\ref{thm:cayley-out} 
and~\ref{thm:S-ary-out}) via  elementary recursive descriptions of trees and the
Lagrange inversion formula (LIF). We give the details of the proof of
Theorem~\ref{thm:cayley-profile}, and sketch the other three, which
are similar.

 We have been unable to reprove in this way the results that deal with the
in-type or the complete type.

\subsubsection{The vertical profile of $\cS$-embedded Cayley trees.}
Let $x=(x_i)_{ i\in \zs}$ be a sequence of indeterminates, and let
$A_0\equiv A_0(x)$
be the exponential \gf\  of $\cS$-embedded Cayley trees, where $x_i$ keeps track
of the number of  vertices lying
at abscissa $i$, for all $i\in\zs$. That is,
$$
A_0=\sum_T \frac{1}{|T|!}\prod_{v\in V} x_{a(v)},
$$
where
$|T|$ is the size of the tree $T$ (the number of vertices) and $V$
the vertex set of $T$. For $j \in
\zs$, let $A_j$ be the series obtained from $A_0$ by replacing each
$x_i$ by $x_{i+j}$. An $\cS$-embedded Cayley tree is obtained by
attaching to a root vertex (lying at abscissa $0$) a set of Cayley 
trees  whose roots lie in $\cS$. Hence
$$
A_0=x_0 \exp\left(\sum_{s\in \cS} A_s\right).
$$
Consequently, for all $i \in \zs$,
\beq\label{syst-A}
A_i=x_i \exp\left(\sum_{s\in \cS} A_{i+s}\right).
\eeq
Let now $n\equiv (n_i)_{\ell\le i \le r}$ be a sequence of
positive integers.  The number of $\cS$-embedded Cayley trees
of vertical profile $n$ is $|n|! [x^n] A_0$, where $|n|=\sum_i n_i$ and
$[x^n] A_0$ stands
for the coefficient of $x_\ell ^{n_\ell} \cdots x_r ^{n_r}$ in
$A_0$. We  use the following version  of the  Lagrange-Good
inversion formula~\cite{gessel-lagrange,good}: 
if for all $\ell\le i\le r$,
$$
F_i=x_i g_i(F_\ell, \ldots,F_r),
$$
then 
$$
[x^n]F_0\ \ =\ \  [x^n]\left( x_0 \prod_{i=\ell}^r  g_i(x)^{n_i}
\det\left( \delta_{ij}-\frac {x_i}{g_j(x)}\frac{\partial
    g_j(x)}{\partial x_i}\right)_{ \ell \le i, j \le r}\right).
$$
Hence, it follows from~\eqref{syst-A} that
\begin{eqnarray*}
[x^n]A_0 &=& [x^n]\left( x_0 \prod_{i=\ell}^r  \exp\left(n_i \sum_{s\in \cS}
x_{i+s}\right)
\det\left( \delta_{ij}-x_i\chi_{i-j \in \cS}\right)_{ \ell \le i, j \le r}\right),
\end{eqnarray*}
where by convention $x_i=0$ if $i<\ell $ or $i>r$.

We find convenient to use the following classical expression of
the above determinant in terms of \emm configuration of 
cycles,. Let $G\equiv G_{\ell,r}(\cS)$ be the digraph with vertices $\{\ell, \ldots, r\}$
and with an arc from $i$ to $j$ if and only if $i-j \in \cS$. A cycle
of $G$ is \emm elementary, if it never visits the same vertex twice. A
\emm configuration of cycles, if a set $C$ of elementary cycles such that
each vertex $i \in \llbracket\ell,r\rrbracket$ is contained in at most
one cycle of $C$.  We loosely write $i\in C$ when $i$ is contained in
a cycle of $C$. Then, by expanding the determinant, one finds
$$
\det\left( \delta_{ij}-x_j\chi_{i-j \in \cS}\right)_{ \ell \le i, j \le r}
= \sum_C (-1)^{|C|} \prod_{i \in C} x_i,
$$
where $|C|$ denotes the number of elementary cycles of $C$. Hence we can now
rewrite
\begin{eqnarray}
[x^n]A_0 &=& [x^n]\left( x_0 \prod_{i=\ell}^r  \exp\left(x_i \sum_{s\in \cS}
n_{i-s}\right) \sum_C (-1)^{|C|} \prod_{i \in C} x_i\right),\nonumber
\\
&=& \sum_C (-1)^{|C|} \prod_{i=\ell}^r\, [x_i^{n_i}]
\left(x_i^{\chi_{i=0}+\chi_{i \in C}} \exp\left(x_i \sum_{s\in  \cS}n_{i-s}\right) \right) \nonumber
\\
&=& \sum_C (-1)^{|C|} \prod_{i=\ell}^r \frac{\left(\sum_{s\in
      \cS}n_{i-s}\right)^{n_i-\chi_{i=0}-\chi_{i \in C}}}
{(n_i-\chi_{i=0}-\chi_{i \in C})!}\nonumber
\\
&=&\prod_{i=\ell}^r \frac{\left(\sum_{s\in
      \cS}n_{i-s}\right)^{n_i-\chi_{i=0}-1}}
{(n_i-\chi_{i=0})!} 
\sum_C (-1)^{|C|}\prod_{i=\ell}^r \left(\left(\sum_{s\in
      \cS}n_{i-s}\right)^{\chi _{i\not \in C}} \left(n_i
  -\chi_{i=0}\right)^{\chi _{i \in C}} \right)\label{sum-C}
\end{eqnarray}
where now $n_i=0$ if  $i<\ell $ or $i>r$. 

The following lemma shows that, under the hypotheses of
Theorem~\ref{thm:cayley-profile},  the sum over $C$ factors nicely.
Theorem~\ref{thm:cayley-profile} follows at once.

\begin{Lemma}\label{lem:cycles}
  Let   $\cS\subset \zs$ 
  such that $\max \cS=1$.  For $\ell \le 0 \le r$, let $G_{\ell,r}(\cS)$ be
  the  above defined  digraph. Define the following polynomial in the
  indeterminates $y_\ell, \ldots, y_r$:
$$
P_{\ell,r}= \sum_C (-1)^{|C|}\prod_{i=\ell}^r \left(\left(\sum_{s\in
      \cS}y_{i-s}\right)^{\chi _{i\not \in C}} \left(y_i
  -\chi_{i=0}\right)^{\chi _{i \in C}} \right),
$$
where the sum runs over configurations of cycles $C$ on the graph
$G_{\ell,r}(\cS)$, $|C|$ stands for the number of elementary cycles in
$C$ and $y_i=0$ if  $i<\ell $ or $i>r$. Clearly, 
$P_{0,0}=\chi_{0 \in  \cS}$. Assume now $|\ell|+r >0$. If $\min
\cS=-1$ or $\ell=0$, then
$$
P_{\ell,r}= \left(\sum_{s\in \cS} y_{-s}\right)
\prod_{i=\ell+1}^{r-1}y_i.
$$
\end{Lemma}
\begin{proof}
We first assume that $\ell=0$.  Define the auxilliary polynomial
$Q_r(y_0, y_1, \ldots, y_r)$ by
$Q_0(y_0)=1$ and for $r>0$,
$$
Q_r(y_0, y_1, \ldots, y_r)=\sum_C (-1)^{|C|}\prod_{i=1}^r \left(\left(\sum_{s\in
      \cS}y_{i-s}\right)^{\chi _{i\not \in C}} \left(y_i
  \right)^{\chi _{i \in C}} \right),
$$
where the sum now runs over configurations of cycles on the graph
$G_{1,r}(\cS)$, and $y_i=0$ if  $i<0 $ or
$i>r$. 

Given that $\max \cS=1$, an elementary cycle necessarily consists of
the vertices $i-s, i-s-1,  \ldots, i+1, i$ (visited in this order),
where $s \in \cS 
\setminus\{1\}$. In particular, two elementary cycles having no
vertex in common occupy disjoint \emm intervals, of vertices. This
allows us to write a recurrence relation for the polynomials $Q_r$,  by
considering whether the vertex 1 belongs to the configuration 
$C$ or not. 
For $r>0$,
$$
Q_r(y_0, \ldots, y_r)= \left( \sum_{s \in \cS}y_{1-s}\right) Q_{r-1}
(y_1, \ldots, y_r)
- \sum_{s\in \cS \setminus\{1\}}  Q_{r+s-1}(y_{1-s},
\ldots, y_r) \left(\prod_{i=1}^{1-s}y_i\right).
$$
The first term corresponds to configurations not containing $1$. The sum over $s \in \cS \setminus\{1\}$ corresponds to the choice of the
size of the cycle containing 1, which contains $1-s$ vertices. It follows by induction on $r$ that 
\beq\label{Q}
Q_r(y_0, \ldots, y_r)= \prod_{i=0}^{r-1} y_i.
\eeq
Similarly,
$$
P_{0,r}= \left( \sum_{s \in \cS}y_{-s}\right) Q_{r}
(y_0, \ldots, y_r)
-(y_0-1) 
\sum_{s\in \cS \setminus\{1\}}
Q_{r+s}(y_{-s},
\ldots, y_r) \left(\prod_{i=1}^{-s}y_i\right).
$$
It then follows from~\eqref{Q} that
$$
P_{0,r}=\left( \sum_{s \in \cS}y_{-s}\right) \prod_{i=1}^{r-1} y_i,
$$
as announced in the lemma.

We now assume that $\ell<0$ and that $\min \cS=-1$. That is,
$\cS=\{\pm 1\}$ or $\cS=\{\pm 1,0\}$, and in particular $\cS$ is
  symmetric.  There are
now two types of elementary cycles, those reduced to a loop (if $0 \in
\cS$) and those consisting of two neighbour vertices $i$ and $i+1$. By
considering whether the vertex 0 belongs to the configuration $C$ or
not, and whether this cycle consists solely of the vertex $0$, or of
the vertices $0$ and $1$, or of the vertices $0$ and $-1$, one obtains 
\begin{multline*}
  P_{\ell,r}= \left( \sum_{s \in \cS}y_{-s}\right) Q_r(y_0, \ldots, y_r)
Q_{|\ell|}(y_0,y_{-1}, \ldots, y_\ell)
\\
-(y_0-1)Q_r(y_0, \ldots, y_r)Q_{|\ell|}(y_0,y_{-1}, \ldots,
y_\ell)\chi_{0\in\cS}
\\
-(y_0-1)y_1Q_{r-1}(y_1, \ldots, y_r)Q_{|\ell|}(y_0,y_{-1}, \ldots,
y_\ell)
\\
-(y_0-1)y_{-1}Q_{r}(y_0, \ldots, y_r)Q_{|\ell|}(y_{-1}, \ldots,
y_\ell)
\end{multline*}
with $Q_{-1}=0$.
The announced expression of $P_{\ell,r}$ now follows
from~\eqref{Q}. This concludes the proof of the lemma, and our first
alternative proof of Theorem~\ref{thm:cayley-profile}.
\end{proof}

\subsubsection{The vertical profile of $\cS$-ary trees.}
The proof of Theorem~\ref{thm:S-ary-profile}, which deals with the
profile of $\cS$-ary trees, is very similar. One
starts from the system of equations
$$
A_i=x_i \prod_{s\in \cS}(1+A_{i+s}),
$$
for all $i \in \zs$. The LIF now gives
\begin{eqnarray*}
[x^n]A_0 &=& [x^n]\left( x_0 \prod_{i=\ell}^r\prod_{s\in \cS}  (1+x_{i+s})^{n_i}
\det\left( \delta_{ij}-\frac{x_i}{1+x_i}\chi_{i-j \in \cS}\right)_{
  \ell \le i, j \le r}\right) .
\end{eqnarray*}
Again, we express the determinant as a sum over configurations of
cycles. This yields
\begin{eqnarray*}
[x^n]A_0 &=&\sum_C (-1)^{|C|}\prod_{i=\ell}^r
[x_i^{n_i}]\left( x_i^{\chi_{i=0}+\chi_{i\in C}}
(1+x_i)^{\sum_sn_{i-s}-\chi_{i\in C}}\right)
\\
&=&\prod_{i=\ell}^r \frac{(\sum_{s\in \cS} n_{i-s}-1)!}{(n_i-
  \chi_{i=0})! (\sum_{s\in \cS} n_{i-s}-n_i+\chi_{i=0})!}
\\
&&\ \hbox{} \hskip 30mm \sum_C (-1)^{|C|}\prod_{i=\ell}^r \left(\left(\sum_{s\in
      \cS}n_{i-s}\right)^{\chi _{i\not \in C}} \left(n_i
  -\chi_{i=0}\right)^{\chi _{i \in C}} \right)
\end{eqnarray*}
where by convention $x_i=n_i=0$ if $i<\ell $ or $i>r$. We recognize the
same sum over $C$ as in~\eqref{sum-C}, and Lemma~\ref{lem:cycles}
then yields Theorem~\ref{thm:S-ary-profile}.

\subsubsection{The out-type of $\cS$-embedded Cayley trees.}
In order to prove
Theorem~\ref{thm:cayley-out},
we start from the system
\beq\label{eq-func-cayley}
A_i=x_i \exp\left(\sum_{s\in \cS} x_{i+s,s}A_{i+s}\right),
\eeq
which is a refined version of \eqref{syst-A} where for each
$(j,s)\in\mathbb{Z}\times \cS$ the 
 indeterminate $x_{j,s}$ keeps
track of the number of vertices of out-type $(j;s)$.

We first apply the LIF with respect to the variables $x_i$:
\begin{eqnarray*}
  [x^n]A_0&=&[x^n] \left(x_0 \prod_{i=\ell}^r \exp\left(n_i \sum_s
x_{i+s,s}x_{i+s}\right)
\det(\delta_{ij}-x_i x_{i,i-j}\chi_{i-j \in \cS})_{\ell \le i,j\le
  r}\right).
\end{eqnarray*}
The expansion of the determinant now reads
$$
\sum_C(-1)^{|C|}\prod_{(i,i-s)\in C}x_i x_{i,s},
$$
where  we write $(i,i-s) \in C$ when the arc $(i,i-s)$ belongs to
one of the cycles of $C$. This gives
\begin{eqnarray}
  [x^n]A_0&=&\sum_{C} (-1)^{|C|}\left(\prod_{(i,i-s)\in C}x_{i,s}\right) \prod_{i=\ell}^r\, [x_i^{n_i}]
\left( x_i^{\chi_{i=0}+\chi_{i\in C}} \exp\left( x_i \sum_s n_{i-s} x_{i,s}\right)\right).\label{lif}
\end{eqnarray}
We now extract the coefficient  of $\prod_{i, s}
x_{i,s}^{n(i,s)}$, with 
$$
n_i=\chi_{i=0} +\sum_{s\in \cS} n(i,s).
$$
We obtain for this coefficient the following expression
$$
\sum_C (-1)^{|C|}\prod_{i=\ell}^r\prod_{s\in \cS}
\frac{n_{i-s}^{n(i,s)-\chi_{(i,i-s)\in C}}}{(n(i,s)-\chi_{(i,i-s)\in
    C})!}.
$$
After a few simple reductions, this gives
Theorem~\ref{thm:cayley-out}, provided the following counterpart of
Lemma~\ref{lem:cycles} holds: 
\beq\label{ide2}
P_{\ell,r}:= \sum_C (-1)^{|C|}\prod_{i=\ell}^r n_i^{\chi_{i\not \in C}}
\prod_{(i,i-s) \in C}  n(i,s)
=
\prod_{i=\ell}^{-1}n(i,-1) \prod_{i=1}^{r}n(i,1),
\eeq
where $n_i=\chi_{i=0}+ \sum_s n(i,s)$.
The proof of this identity is similar to the proof of
Lemma~\ref{lem:cycles}. One proceeds by induction on $r+|\ell|$, first
for $\ell=0$ 
and then for $\ell<0$, after introducing the following auxilliary
polynomial:
$$
Q_r=\sum_C (-1)^{|C|} \prod_{i=1}^r n_i^{\chi_{i\not \in C}}
\prod_{(i,i-s) \in C}  n(i,s),
$$
where the sum now runs over configurations of cycles on the graph
$G_{1,r}(\cS)$. A recurrence relation on $Q_r$ implies that $Q_r=
\prod_{i=1}^r n(i,1)$. Expressing $P_{\ell,r}$ in terms of the
$Q_i$'s, as in the proof of Lemma~\ref{lem:cycles}, finally
establishes~\eqref{ide2}.

\subsubsection{The out-type of $\cS$-ary trees.}
In order to prove Theorem~\ref{thm:S-ary-out}, 
we start from
\beq\label{eq-func-S-ary}
A_i=x_i \prod_{s\in \cS}(1+x_{i+s,s}A_{i+s}).
\eeq
The calculation is similar to the previous one. In particular, one uses again~\eqref{ide2}.
\medskip

 \subsubsection{A variant for the out-type of $\cS$-ary trees.}
One can also enrich the first calculation of this
section by adding
weights $x_{i,s}$ on vertices of out-type $(i;s)$, and thus prove
directly Theorem~\ref{thm:cayley-out}, in the
form~\eqref{cayley-profile-refined}. One starts again
from~\eqref{eq-func-cayley}. Extracting the coefficient of $x^n$
gives~\eqref{lif}. Then, one does \emm not, extract the coefficient of $\prod_{i, s}
x_{i,s}^{n(i,s)}$, but uses instead the following refinement of Lemma~\ref{lem:cycles}:
\begin{multline}\label{refined}
  P_{\ell,r}:= \sum_C (-1)^{|C|}\left(\prod_{(i,i-s)\in C} x_{i,s}\right)\prod_{i=\ell}^r \left(\left(\sum_{s\in
      \cS}y_{i-s}x_{i,s}\right)^{\chi _{i\not \in C}} \left(y_i
  -\chi_{i=0}\right)^{\chi _{i \in C}}
\right)\\ \prod_{i=\ell}^{-1} x_{i,-1} \prod_{i=1}^r x_{i,1}\left(\sum_{s\in \cS} x_{0,s}y_{-s}\right)
\prod_{i=\ell+1}^{r-1}y_i.
\end{multline}
The proof is a straightforward extension of the proof of Lemma~\ref{lem:cycles}, using
\begin{eqnarray*}
  Q_r(y_0, y_1, \ldots, y_r)&:=&\sum_C (-1)^{|C|}\left(\prod_{(i,i-s)\in C} x_{i,s}\right)\prod_{i=1}^r \left(\left(\sum_{s\in
      \cS}y_{i-s}\right)^{\chi _{i\not \in C}} \left(y_i
  \right)^{\chi _{i \in C}} \right)
\\
&=&
\left(\prod_{i=1}^r x_{i,1}\right) \left( \prod_{i=0}^{r-1} y_i\right).
\end{eqnarray*}
where the sum over $C$ is over configurations of cycles of $G_{1,r}(\cS)$.

\subsection{Application of the matrix-tree theorem}

We now apply the matrix-tree theorem to prove that the \gf\ of
$\cS$-embedded trees of vertical profile $(n_\ell, \ldots, n_r)$,
counted by the number of vertices of out-type $(i;s)$, for all $i$ and
$s$, is given by~\eqref{cayley-profile-refined}. This proves
Theorems~\ref{thm:cayley-profile} and~\ref{thm:cayley-out}
simultaneously.

We consider as before the vertex set $V=\cup_{i=\ell}^r V_i$, with $V_i=\{i^1,
i^2, \dots, i^{n_i}\}$, and we consider the digraph $K$ on $V$ where
an arc joins $i^p$ to $j^q$ if and only if $i^p  \not = j^q$ and
$j=i-s$ for some $s \in \cS$. This arc receives the weight $x_{i,s}$.
Recall that, on a digraph, a spanning tree is always rooted, with all
edges of the tree pointing towards the root vertex. Thus a spanning
tree of $K$ is precisely an $\cS$-tree, as defined in
Section~\ref{sec:general}. The out-type is defined as before.  It
is easy to see that the \gf\ of
$\cS$-embedded 
Cayley trees of profile $(n_\ell, \ldots, n_r)$ is $ n_0 n! /
\prod_{i=0}^r n_i!$ times the \gf\ of spanning trees of
$K$ rooted at  $0^{n_0}$. Hence~\eqref{cayley-profile-refined} is
equivalent to the following proposition.
\begin{Proposition}
The \gf\  of spanning trees of $K$ rooted at $0^{n_0}$ equals:
$$
\left( \prod_{i=\ell}^{-1} x_{i,-1} \right) \left(\prod_{i=1}^r x_{i,1} \right)\left(\prod_{i=\ell+1}^{r-1} n_i  \right)
\prod_{i=\ell}^{ r}\left(\sum_{s\in\cS}n_{i-s}x_{i,s}\right)^ {n_i-1}.
$$
\end{Proposition}
\begin{proof} We apply the weighted version of the matrix-tree
  theorem~(see~\cite[Thm.~3.6]{Tutte:orientedMatrixTreeRefAccordingToStanley}
  or~\cite[Thm.~5.6.8]{stanley-vol2}).
The (weighted) Laplacian matrix $M$ of $K$ has its rows and
columns  indexed by elements of $V$, and coefficients given by:
$$
M(i^p,j^q)=\left\{
\begin{array}{ll}
\displaystyle -x_{i,0}\chi_{ 0 \in \cS} +\sum_{s\in\cS} x_{i,s}n_{i-s}  & \mbox{if } i^p=j^q,
\\
\displaystyle -x_{i,i-j}\chi_{i-j \in \cS} &\mbox{otherwise.}
\end{array}
\right.
$$
In this matrix, the diagonal coefficient $\sum_{s\in\cS} x_{i,s}n_{i-s}
-x_{i,0}\chi_{ 0 \in \cS}$ 
is the (weighted) out-degree of any vertex of the form $i^p$.

By the matrix-tree theorem, the \gf\ of spanning trees of $K$ rooted
at $0^{n_0}$ is the determinant of the matrix $\tilde M$ obtained
from $M$ 
by removing the line and column indexed by $0^{n_0}$.

We consider $\tilde M$ as a linear
operator acting on the vector space on $\cs$ spanned by 
$V\setminus\{0^{n_0}\}$. We will first identify a number of eigenvectors
and eigenvalues of $\tilde M$, and then describe the action of $\tilde
M$ on the orthogonal complement of the subspace spanned by these
eigenvectors.

   Let $\ell \le j \le r$, $\alpha=(\alpha_1, \ldots, \alpha_{\tn_j}) \in
\cs^{\tn_j}$, and denote 
$$
V_j(\alpha)=\alpha _1 j^1 +\alpha _2 j^2
+ \cdots +\alpha _{\tn_j} j^{\tn_j},
$$
where $\tilde n_i=n_i-\chi_{i=0}$.
Denote also $W_j=V_j(1,\ldots, 1)$.
Using the definition of $\tilde M$, one computes:
\beq\label{MV}
\tilde M V_j(\alpha)= V_j(\alpha) \sum_{s\in \cS} x_{j,s}n_{j-s}
- \left(\sum_{q=1}^{\tn_j} \alpha_q \right) \sum_{i=\ell}^r x_{i,i-j}
\chi_{i-j \in \cS}W_i.
\eeq
Consequently, for $\ell\le j\le r$, the vector space formed of the
$V_j(\alpha)$ such that $\sum_{q=1}^{\tn_j} \alpha_q =0$ is an
eigenspace of dimension $\tn_j-1$, associated with the eigenvalue
$\sum_{s\in \cS} x_{j,s}n_{j-s}$. If $n=\sum_j n_j$ denotes the size of
the trees we are counting, we have thus identified an eigenspace of
$\tilde M$ of dimension $\sum_j (\tn_j-1)= n-1-(r-\ell+1)$, and found
in $\det (\tilde M)$ a factor
$$
\prod_{j=\ell}^r \left(\sum_{s\in \cS}
  x_{j,s}n_{j-s}\right)^{\tn_j-1}.
$$
The vectors $W_j$, for $\ell \le j\le r$, span the orthogonal
complement of this eigenspace, and by~\eqref{MV},
$$
\tilde MW_j = W_j \sum_{s\in \cS} x_{j,s}n_{j-s}
-\tn_j\sum_{i=\ell}^r x_{i,i-j}
\chi_{i-j \in \cS}W_i.
$$
Therefore,
\beq\label{MN}
\det(\tilde M)= \det(N)\prod_{j=\ell}^r \left(\sum_{s\in \cS}
  x_{j,s}n_{j-s}\right)^{\tn_j-1}
\eeq
where   $N=(N(i,j))_{\ell \le i,j \le r}$ is the square matrix of size
$(r-\ell+1)$  with coefficients
$$
N(i,j)= \left\{
\begin{array}{ll}
\displaystyle -\tn_i x_{i,0} \chi_{0\in \cS} + \sum_{s\in \cS} x_{i,s} n_{i-s}
&\mbox{if } i=j,
 \\
\displaystyle  - \tilde n_jx_{i,i-j} \chi_{i-j \in \cS} & \mbox{otherwise.}
\end{array}
\right.
$$
We expand $\det(N)$ first as a sum over permutations $\sigma$ of $\{\ell,
\ldots, r\}$, and then as a sum over configurations of cycles $C$ on the
graph $G_{\ell,r}(\cS)$ (a fixed point $i$ of $\sigma$ gives rise
either to a
loop of weight $-\tn_i x_{i,0} \chi_{0\in S}$ or to a point not
belonging to $C$, with weight $\sum_{s\in \cS} x_{i,s} n_{i-s}$). The
resulting expression coincides with the left-hand side
of~\eqref{refined} (with $y_i=n_i$),
and the identity~\eqref{refined} thus gives:
$$
\det(N)=\prod_{i=\ell}^{-1} x_{i,-1} \prod_{i=1}^r x_{i,1}\left(\sum_{s\in \cS} x_{0,s}n_{-s}\right)
\prod_{i=\ell+1}^{r-1}n_i.
$$
Together with~\eqref{MN}, this completes the proof of the proposition.
\end{proof}

\section{Final comments}

\subsection{Simpler proofs?} 
The bijection of Section~\ref{sec:general} is  fairly complicated. Can one find
simpler proofs of our results, for trees with negative labels? 
Many proofs of Cayley's formula exist, beyond the
three that we have adapted in this paper (namely, Joyal's bijective proof,
functional equations and Lagrange's formula, and the matrix-tree
theorem). It is possible that other proofs may be adapted to provide
simpler proofs of our results, especially for trees with negative
labels and for the distribution of in-types (which we can only address
via the bijection of Section~\ref{sec:general}). Finding such a proof
could also enlighten the questions raised in the following
subsections.

\subsection{The complete type}
\label{sec:complete-questions}
Is Theorem~\ref{thm:cayley-complete}, which
deals with the complete type of $\cS$-embedded trees,  as general as
it could? Does one really need to assume that the trees are
non-negative, and that $0\not \in \cS$? 

We do not know how to answer this question, but is it easy to see that, if
there exists a more general formula, the sets of functions
considered in Theorem~\ref{thm:basic} (for non-negative trees) and
in Theorem~\ref{thm:rooted} (for trees with negative abscissas) will
not allow us to prove it. More precisely,
when $0\in \cS$, or when the trees have negative abscissas, there exists no bijection   between the functions and the $\cS$-trees  of
Theorem~\ref{thm:basic} (or Theorem~\ref{thm:rooted}) that  would preserve the distribution of
complete types. Here are two simple counterexamples. For non-negative trees first, take $V=\{0^1,0^2,1^1\}$, and define the function $f$ by 
$$
f(1^1)=0^1 \quad \hbox{and} \quad f(0^2)=0^2.
$$
This function satisfies Condition (F) of Theorem~\ref{thm:basic} as soon as $\{0, 1\}\subset \cS$, but there exists no
$\cS$-tree with the same type distribution.
Now for trees with negative abscissas, take $V=\{-1^1, -1^2, 0^1, 0^2,
1^1\}$ and consider the following tree:

\begin{center}
\scalebox{0.8}{\input{contrex.pstex_t}}  
\end{center}
It satisfies Conditions $(\rm T_1)$ and $(\rm T_2)$ of
Theorem~\ref{thm:rooted}, 
but there exists no $\cS$-function satisfying (F) with the same type distribution.

\subsection{Trees embedded in trees}
\label{sec:treesintrees}
Following a seminar presenting this work in December 2011, Andrea
Sportiello discovered a remarkable formula that generalizes 
Theorem~\ref{thm:cayley-profile-pm},
and added evidence that more factorization results exist for
embeddings of trees in  general graphs.

Let $\cT$ be a finite rooted tree with vertex
set $\cA$ (called the set of abscissas) and root $\rho$. 
Let $T$ be a Cayley
tree with vertices labelled $1, 2, \ldots, n$. By a \emm $\cT$-embedding of
$T$,, we mean an assignment of abscissas to vertices of $T$,
that is, a map $a:\llbracket 1, n \rrbracket \rightarrow
\cA$ such that 
\begin{itemize}
\item the abscissa of the root of $T$ is $\rho$,
\item if $v$ and $v'$ are neighbours in $T$, then $a(v)$ and $a(v')$
  are neighbours in $\cT$.
\end{itemize}
In graph theoretic terms, we have a root preserving \emm morphism, of
$T$ to $\cT$. The \emm profile, of  this $\cT$-embedded tree is the collection
$(n_i)_{i\in \cA}$, where $n_i$ is the number of vertices of $T$ of
abscissa~$i$. 
The embedding is \emm surjective, if $n_i>0$ for all $i$. 
Then the number of surjective $\cT$-embedded Cayley trees having
profile  $(n_i)_{i\in \cA}$
is
$$
 {n_\rho}\,\frac {n!} {\prod_{i\in \cA} n_i!} \prod _{i\in
  \cA}\left( \left(\sum_{j \sim i} n_j\right)^{n_i-1} n_i^{\deg(i)-1}\right),
$$
where $n=\sum_i n_i$ is the size of the trees, $\deg(i)$ is the degree of $i$ in $\cT$ and $j\sim i$ means that
$i$ and $j$ are neighbours in $\cT$. It is easily checked that this
gives Theorem~\ref{thm:cayley-profile-pm} when  $\cT$ is the
tree on the vertex set $\llbracket \ell, r\rrbracket$ with an edge
between $i$ and $i+1$ for all $i \in \llbracket \ell, r-1\rrbracket$. 

Andrea Sportiello proved the above formula using the matrix-tree
theorem. We do not know of any bijective proof.

\bigskip 
\noindent{\bf (Important) note added to the proof (September 2012).}
After publication of this paper on ArXiv, our results have been
reproved and generalized by Bernardi and Morales~\cite{bernardi}. Their proof is very
elegant, combinatorial but not bijective. Their formulas are valid for
\emm any set , $\cS$, and for \emm general, embedded trees --- but of course
they do not always simplify into product forms. These formulas involve
a non-explicit sum over  a family of trees, which simplifies in some 
cases.

Let us give an example. Assume $\max \cS=1$, and take an integer sequence
$(n_\ell, \ldots ; n_0, \ldots, n_r)$.
%
If $\ell=0$,
the number of $\cS$-embedded Cayley trees having vertical profile
$(n_\ell, \ldots ; n_0, \ldots, n_r)$ is given by
Theorem~\ref{thm:cayley-profile}, 
which we rewrite as
$$
\displaystyle \frac {n!}{\prod\limits_{i=\ell}^r n_i!}\ \prod_{i=0}^{r-1}
n_i\ \prod_{i=\ell}^r \left( \sum_{s\in \cS} n_{i-s}\right)^{n_i-1}.
$$
If $\ell=-1$ (and $n_{-1}>0$), it follows from~\cite[Section~3]{bernardi} that
this number is
$$
\frac {n!}{\prod\limits_{i=\ell}^r n_i!}\ \prod_{i=0}^{r-1}
n_i\ \prod_{i=\ell}^r \left( \sum_{s\in \cS} n_{i-s}\right)^{n_i-1}\left(\sum_{s\in \cS, s\le -1} n_{-s-1}\right).
$$
The formula becomes more and more complex as $\ell$ decreases. If
$\ell=-2$ and $n_{-2}n_{-1}>0$, it reads
$$
\frac {n!}{\prod\limits_{i=\ell}^r n_i!}\ \prod_{i=0}^{r-1}
n_i\ \prod_{i=\ell}^r \left( \sum_{s\in \cS} n_{i-s}\right)^{n_i-1}\left(n_{-2}\sum_{s\in \cS, s\le -2} n_{-s-2}
+ \sum_{s\in \cS, s\le -1} n_{-s-2}\sum_{s\in \cS, s\le -1} n_{-s-1}
\right).
$$

Bernardi and Morales also answer the question raised in
Section~\ref{sec:complete-questions} on the generality of
Theorem~\ref{thm:cayley-complete} (the complete type). For non-negative trees with $\max
\cS=1$ and $0\in \cS$, they find an explicit, but complicated,
expression. In the other cases, the sum over trees does not seem to
simplify.

\spacebreak

\bigskip
\noindent 
{\bf Acknowlegements.} We thank  Philippe Marchal, Jean-François
Marckert and Andrea
Sportiello for interesting discussions about this work, some of which
led to the results of Section~\ref{sec:treesintrees}.

\bibliographystyle{plain}
\bibliography{biblio.bib}

\end{document}